\documentclass[a4paper,11pt]{amsart}
\usepackage{graphicx}
\usepackage{amsmath}
\usepackage{amssymb,amsfonts,scalerel}
\usepackage{bbm}
\usepackage[margin=1.5in]{geometry}
\usepackage{combelow}

\usepackage{helvet}

\RequirePackage{mathrsfs} \let\mathcal\mathscr

\numberwithin{equation}{section}

\newtheorem{theorem}{Theorem}[section]
\newtheorem{lemma}[theorem]{Lemma}

\newtheorem{corollary}[theorem]{Corollary}

\theoremstyle{definition}

\newtheorem{remark}[theorem]{Remark}

\usepackage{graphicx}
\usepackage{amsmath}
\usepackage{amsthm}
\usepackage{amssymb}
\usepackage{enumerate}
\usepackage{amstext}
\usepackage{amsfonts}
\usepackage{graphicx}

\numberwithin{equation}{section}

\newcommand{\re}{\text{Re}}

\newcommand{\PP}{\mathbb{P}}
\renewcommand{\AA}{\mathbb{A}}
\newcommand{\FF}{\mathbb{F}}
\newcommand{\ZZ}{\mathbb{Z}}

\newcommand{\NN}{\mathbb{N}}
\newcommand{\QQ}{\mathbb{Q}}
\newcommand{\RR}{\mathbb{R}}
\newcommand{\CC}{\mathbb{C}}
\newcommand{\TT}{\mathbb{T}}
\newcommand{\sumstar}{\sideset{}{^*}\sum}
\newcommand{\cupstar}{\sideset{}{^*}\bigsqcup}

\def\msquare{\mathord{\scalerel*{\Box}{gX}}}

\renewcommand{\hat}{\widehat}

\newcommand{\ki}{K_\infty}

\newcommand{\hJ}{\widehat J}

\newcommand{\x}{\mathbf{x}}
\newcommand{\y}{\mathbf{y}}
\newcommand{\w}{\mathbf{w}}
\newcommand{\uu}{\mathbf{u}}

\renewcommand{\c}{\mathbf{c}}
\newcommand{\f}{\mathbf{f}}
\renewcommand{\v}{\mathbf{v}}
\renewcommand{\d}{\mathbf{d}}
\newcommand{\e}{\mathbf{e}}

\newcommand{\z}{\mathbf{z}}
\renewcommand{\b}{\mathbf{b}}

\renewcommand{\r}{\mathbf{r}}

\newcommand{\ux}{{\underline{\mathrm{x}}}}
\newcommand{\ua}{\underline{\mathrm{a}}}
\newcommand{\ub}{\underline{\mathrm{b}}}
\newcommand{\uxi}{\underline{\xi}}
\newcommand{\uz}{{\underline{\mathrm{z}}}}
\newcommand{\uc}{{\underline{\mathrm{c}}}}
\newcommand{\ud}{\underline{\mathrm{d}}}
\newcommand{\ualf}{\underline{\alpha}}

\newcommand{\uZ}{\underline{\mathrm{Z}}}
\newcommand{\ue}{\underline{\mathrm{e}}}

\def\scrF{{\mathcal F}}

\def\scrO{{\mathcal O}}

\def\scrP{{\mathcal P}}

\def\scrV{{\mathcal V}}

\def\cO{{\mathcal{O}}}

\def\vecd{{\text{\boldmath$d$}}}

\def\vecnull{{\text{\boldmath$0$}}}

\def\GL{\operatorname{GL}}

\def\diag{{\textrm{diag}}}

\DeclareMathOperator{\rank}{rank}

\DeclareMathOperator{\meas}{meas}

\DeclareMathOperator{\ord}{ord}

\newcommand{\vp}{\varpi}

\renewcommand{\hat}{\widehat}
\renewcommand{\bar}{\overline}

\renewcommand{\mod}{\:\text{mod}\:}
\newcommand{\ve}{\varepsilon}
\newcommand{\F}{{\underline{F}}}
\newcommand{\G}{{\underline{G}}}
\newcommand{\cmatr}[2]{\left( \begin{matrix} #1 \\ #2 \end{matrix} \right) }

\newcommand{\matr}[4]{\left( \begin{matrix} #1 & #2 \\ #3 & #4 \end{matrix} \right) }
\newcommand{\smatr}[4]{\bigl( \begin{smallmatrix} #1 & #2 \\ #3 & #4 \end{smallmatrix} \bigr) }

\newcommand{\eone}{E_{1,1}}
\newcommand{\etwo}{E_{1,2}}
\begin{document}

\date{\today}

\title{Rational points on complete intersections over $\FF_q(t)$}

\author{P. Vishe}
\address{
Department of Mathematical Sciences\\
Durham University\\
Durham\\ DH1 3LE\\ United Kingdom}
\email{pankaj.vishe@durham.ac.uk}
\begin{abstract}
A Kloosterman refinement for function fields $K=\FF_q(t)$ is developed and used to establish the quantitative arithmetic of the set of rational points on a smooth complete intersection of two quadrics $X\subset \PP^{n-1}_{K}$ , under the assumption that $q$ is odd and $n\geq 9$. 
\end{abstract}

\maketitle
\setcounter{tocdepth}{1}
\tableofcontents

\thispagestyle{empty}
\section{Introduction}\label{sec:intro}
Let $X\subset 
\PP^{n-1}_K $ denote a smooth projective complete intersection defined over a global field $K$ of multi degree type $(d_1,...,d_R)$, i.e., it corresponds to the zero locus of a non-singular system of homogeneous polynomials $F_1(\x),...,F_R(\x)$ of degrees $d_1,...,d_R$ respectively. Establishing properties of the set of $K$-rational points on $X$, denoted by $X(K)$, is a key focus of Diophantine Geometry. An important tool in establishing the Hasse principle and weak approximation is presented by the Hardy-Littlewood circle method. A feature of this method is that it not only gives an existence of the rational points on $X$, but also provides an asymptotic formula for the number of rational points in an expanding box, establishing the quantitative arithmetic of $X(K)$. 

Let $K=\FF_q(t)$, let $\scrO=\FF_q[t]$ be the ring of integers in $K$ and let $K_\infty$ denote the completion of $K$ with respect to the $\infty$-norm on $K$, denoted by $|\cdot |$. Let $\TT=\{|x|<1\}\subset K_\infty$ be an analogue of the unit interval in this setting. The circle method starts with considering an integral
\begin{equation}\label{eq:circlebeg}
\int_{\TT^R}S(\ualf)d\ualf,
\end{equation}
where $d\ualf$ denotes a suitably normalised Haar measure and $S(\ualf)$ denotes a suitable exponential sum, made explicit in Section \ref{sec:circle}.  Given any $Q>0$, a version of the Dirichlet's approximation theorem (see \cite[Lemma 5.1]{Lee11}, \cite{lee-thesis}) gives
\begin{equation}\label{eq:Dirichlet}
\TT^R=\bigcup_{\substack{r\in \cO\\ |r|\leq q^Q\\ \scriptsize{\mbox{$r$
monic}}} }
\bigcup_{\substack{\ua\in \cO^R\\ |\ua|<|r|\\
\gcd(\ua,r)=1} }D(\ua,r,Q),\qquad\textrm{where}\qquad
D(\ua,r,Q)=\left\{\uxi\in \TT^R: |\uxi-\ua/r|< |r|^{-1}q^{-Q/R}\right\}.
\end{equation}
Here, given any $\ux\in K_\infty^R$, $|\ux|=\max\{|x_1|,|x_2|,...,|x_R|\}$ denotes the maximum norm of its co-ordinates. 

The study of rational points on low degree $d$ smooth hypersurfaces ($R=1$) has seen major advances over the years. However, this success has not been mirrored in the $R>1$ case, with Myerson's recent works being one of the notable exceptions. We will try to explain one of the major hurdles here. When $R=1$ and $K=\FF_q(t)$, \eqref{eq:Dirichlet} provides an exact splitting of $\TT$, effectively enabling us to utilise non-trivial cancellations in the averages 
$$\sum_{\substack{r\in \cO\\ |r|=q^Y\\ \scriptsize{\mbox{$r$
monic}}} }\sum_{\substack{a\in \cO\\ |a|<|r|\\
\gcd(a,r)=1} }S(a/r+z), $$
 usually called as a double Kloosterman refinement. This was a key idea in the author's previous work (w. Browning) \cite{Browning_Vishe15}. This idea was employed there to establish the quantitative arithmetic of cubic hypersurfaces over $\FF_q(t)$, as long as $n\geq 8$ and $\textrm{Char}(\FF_q)>3$. When $R\geq 2$, a major log-jam is posed by the fact that so far there is no known way for obtaining a suitable partition of $\TT^R$ with approximating fractions of the type $\ua/r$. The only other available approach is due to Munshi \cite{Munshi15}. When $K=\QQ$ and $R=2$, he essentially used a hybrid of two $1$-dimensional Kloosterman refinements. Upon translating his approach to the function field setting, it amounts to using approximating fractions of the type $(a_1/r_1,a_2/r_2)$, which in turn needs {\em too many} sets to cover $\TT^2$. Therefore, it fails to generalise beyond the $\vecd=(2,2)$ case in a fruitful way.
 
The primary goal of this paper is to overcome this lacuna by producing a refinement of \eqref{eq:Dirichlet}, which will present us with a suitable partition of $\TT^2$. This will provide a route for establishing a double Kloosterman refinement, capable of dealing with a system of two forms ($R=2$) over $K=\FF_q(t)$. We illustrate the utility of this new approach by providing an asymptotic formula for a suitable counting function for any smooth complete intersection of two quadrics ($\vecd=(2,2)$) defined over $K$, as long as, $n\geq 9$ and $2\nmid q$.

We begin with a survey of some existing results. For $X$ of the type $(d,...,d)$ over  $K=\QQ$, a long standing result by Birch \cite{Birch61} implies that $n> (d-1)2^{d-1}R(R+1)$ suffices for the Hasse Principle to hold.  This was generalised to a general $\vecd$ type by Browning and Heath-Brown \cite{Browning-Heath-Brown14}. In Birch's original setting, a recent major breakthrough was achieved by Myerson in \cite{Myer}, \cite{Myer1}, \cite{Myer2}, where he managed to obtain the Hasse principle as long as $n\geq d2^dR+R$ and $X$ is suitably {\em generic}. When $d=2$ and $3$, he is able to drop the genericity condition on $X$ and obtain results for all smooth complete intersections. However his results do not improve those of Birch's when $d$ and $R$ are relatively small. The above results use the Hardy-Littlewood circle method and therefore  also provide us with an asymptotic formula for the number of rational points on $X$, when counted in an expanding box. 

When $K=\FF_q(t)$, the Hasse Principle for $n> d_1^2+...+d_R^2$ is an easy consequence of the Lang-Tsen theory. Establishing weak approximation turns out to be a much harder task. A folklore conjecture predicts that $X$ should satisfy weak approximation as long as $n> d_1^2+...+d_R^2$. It is usually believed that perhaps with a lot more technical work, most of the previously mentioned results over $K=\QQ$ could be translated to the function field setting. This is seen in Lee's PhD thesis \cite{lee-thesis}, \cite{Lee11}, where he obtained an $\FF_q(t)$ analogue of \cite{Birch61}. A novelty is typically attained when one obtains better results over $\FF_q(t)$ as compared with the $\QQ$-setting, often aided by the proven analogue of the generalised Riemann hypothesis over function fields. 

 When $\vecd=(2,2)$ and $2\nmid\textrm{Char}(K)$, a conjecture of Colliot-Th\'el\`ene, Sansuc and Swinnerton-Dyer \cite[Sec 16]{CSS} predicts weak approximation to hold as long as $n\geq 6$. The geometry of a complete intersection of two quadrics is well understood and therefore the geometric methods have been quite effective. When $K$ is an arbitrary number field, weak approximation for $n\geq 9 $ was established by Colliot-Th\'el\`ene, Sansuc and Swinnerton-Dyer \cite{CSS0} and \cite{CSS}. This was improved by Heath-Brown in \cite{Heath-Brown18}, where he established the $n=8$ case. When $K=\FF_q(t)$, a remarkable result of Tian \cite{tian} establishes weak approximation as long as $2\nmid q$ and $n\geq 6$, settling the aforementioned folklore conjecture in this case. The methods in all these results however are purely geometric and fail to shed further light on the structure of rational points $X(K)$. Moreover, they do not generalise to be able to deal with a more general types of complete intersections. The only known improvement of Birch's result in this setting is due to Munshi \cite{Munshi15}, where for $K=\QQ$, he established the quantitative arithmetic as long as $n\geq 11$. Browning and Munshi \cite{Browning-Munshi13} established the quantitative arithmetic when $K=\QQ$ and $n\geq 9$ under the assumption that the singular locus of $X$ consists of a pair of conjugate singular points defined over $\QQ(i)$. When $\vecd=(2,3)$, Browning, Dietmann and Heath-Brown established an asymptotic formula for the Hasse principle as long as $n\geq 29$. Heath-Brown and Pierce \cite{HeathBrown_Pierce17} and Pierce, Schindler and Wood \cite{Pierce-S-W} investigated systems of quadratic forms attaining almost every integer value simultaneously.

\subsection{Main results} We start by stating our main results. From now on, we fix $K=\FF_q(t)$ and $\vecd=(2,2)$. While inspecting \eqref{eq:Dirichlet}, it is easy to construct sub-families of overlapping sets appearing there. For instance, the sub-family $$\{D((a,a),r,Q): \gcd(a,r)=1, r \textrm{ monic }, |r|\leq q^Q\}, $$
contains a lot of sets which overlap with each other. However, this phenomenon can be easily explained by the fact that they cover a region around $\{x_1-x_2=0\}\cap \TT^2$, a rational line segment of low height. The Diophantine approximation of rational points lying on $\{x_1-x_2=0\}$ is explained by the $R=1$ case in \eqref{eq:Dirichlet}. This rationale sets the stage for our partition of $\TT^2$. We first begin by writing $\TT^2$ as a union of nicely placed regions around rational lines of suitable height. Then using techniques in Diophantine approximation, we show that these lines repel each other, ensuring that our regions are disjoint. Now, around each individual line, we invoke the one dimensional Diophantine approximation to get rid of some of the overlapping sets to produce the required partition.

Before stating the result, we begin by making our notion of a {\em generalised line} concrete: given $d\in\scrO$, and a primitive vector $\uc\in \scrO^2$, we define the corresponding {\em generalised line} as \begin{equation}\label{eq:ducdef}L(d\uc):=\{\ua/r\in \TT^2\cap L_1(d\uc,k)\textrm{ for some } k\in \scrO: \gcd(a_1,a_2,r)=\gcd(d,k )=1\},
\end{equation} where $L_1(d\uc,k)$ denotes the affine line 
\begin{equation}
\label{eq:ducdef1}
L_1(d\uc,k):=\{\ux\in K_\infty^2: d\uc\cdot\ux=k\}.
\end{equation}
To clarify our previous comments, $|d\uc| $ will denote the height of $L(d\uc)$. Here and throughout the rest of this work, we say that $\uc=(c_1,c_2)\in\cO^2$ is primitive if $\gcd(c_1,c_2)=1$, and either $c_1$ is monic or $c_1=0$ and $c_2$ is monic. As a result, the relevant vectors $d\uc\neq(0,0)$.  The following theorem will feature our partition of $\TT^2$:
\begin{theorem}
Given any $Q>0$, we have the following:
\label{thm:split}
\begin{equation}\label{eq:Dirichlet2}
\TT^2=\bigsqcup_{\substack{r\textrm{ {\em monic}}\\ |r|\leq q^Q}}\qquad\bigsqcup\limits_{\substack{d\mid r \textrm{ {\em monic, }} \uc\in\scrO^2\textrm{ {\em primitive}}\\ |r|q^{-Q/2}\leq |d\uc|\leq |r|^{1/2}\\ |dc_2|<|r|^{1/2}}}\,\,\,
\bigsqcup_{\substack{\ua\in\scrO^2\\|\ua|<|r|\\ \gcd(\ua,r)=1\\ \ua/r\in L(d\uc)}} D(\ua,r,Q).
\end{equation}
\end{theorem}
Theorem \ref{thm:split} will eventually be proved in Section \ref{sec:split}. Let us give a brief explanation of how \eqref{eq:Dirichlet2} will be derived from \eqref{eq:Dirichlet}. We first begin by using the pigeon hole principle to prove that every rational $\ua/r$ lies on a generalised line of height at most $ |r|^{1/2}$. The extra condition $|dc_2|<|r|^{1/2}$ guarantees that these lines don't intersect each other at rationals of relatively {\em small} denominators. The rational points on each line of low height are much closer to each other and therefore,  we remove neighbourhoods around the rationals of relatively high denominator lying on these lines, as each such rational is sufficiently close to one with the denominator  $  \leq |d\uc|q^{Q/2}$, effectively handing us the condition $|r|q^{-Q/2}\leq |d\uc|$.   Finally, the condition $d\mid r$ is guaranteed from our definition of $L(d\uc)$.

It should be noted that the partition obtained in \eqref{eq:Dirichlet2} is purely based on some fundamental properties of the distribution of rationals in $\TT^2$, making it rather natural. Besides, any further refinements of \eqref{eq:Dirichlet} must address the fact that the rationals in $\TT^2$ accumulate on lines of low height, making our version in \eqref{eq:Dirichlet2} as pivotal for any future ones. Let us now briefly explain how Theorem \ref{thm:split} would lead to a double Kloosterman refinement. \eqref{eq:Dirichlet2} provides a decomposition of $\TT^2$ as a disjoint union of the sets $D(\ua,r,Q)$ placed at rationals $\ua/r$ lying on {\em lines} $L(d\uc)$ satisfying the conditions
\begin{equation}\label{eq:Thm1cond}
\begin{split}
|r|q^{-Q/2}\leq |d\uc|\leq |r|^{1/2},\,\,\, |dc_2|<|r|^{1/2},\,\,\, d\mid r.
\end{split}
\end{equation} 
An important observation to make here is that apart from the condition $d\mid r$, \eqref{eq:Thm1cond} only depends on the absolute values $|r|,|c_1|$, $|c_2|$ and $|d|$. We may therefore readily interchange the sums over $d\uc$ and $r$. After an application of Theorem \ref{thm:split} to \eqref{eq:circlebeg}, we are able to consider averages of the type
\begin{align}\label{eq:SUM}
\sum_{|d\uc|=q^{Y_1}}\sum_{\substack{|r|=q^{Y_2}\\d\mid r}}\,\,\sum_{\ua/r\in L(d\uc)}S(\ua/r+\uz).
\end{align}
For a fixed value of $\uz$, this presents us with a way to utilise oscillations in the values $S(\ua/r+\uz)$, for rationals $\ua/r$ appearing in \eqref{eq:SUM}. As far as our knowledge, this provides the first {\em classical} version of Kloosterman refinement for a system of forms. Theorem \ref{thm:split} should be able to be inductively generalised to produce partitions of $\TT^R$, for arbitrary values of $R$. We intend to return to this topic in a subsequent work.

We now move on to an application of Theorem \ref{thm:split}. Let $F_1(\x),F_2(\x)\in \cO[x_1,...,x_n]$ be two quadratic forms defining a smooth complete intersection. We fix $N\in\cO$ and a vector $\b$ such that $F_1(\b)\equiv F_2(\b)\equiv 0\bmod{N} $. An object of focus for us is the following affine counting function:
 given a non-zero parameter $P\in \cO$ and a smooth, compactly supported function $\omega $ over $K_\infty^n$, let
 \begin{equation}
\label{eq:Count}
N(P):=N_{X,K,\omega}(P,\b,N):=\sum_{\substack{\x\in\scrO^n\\ F_1(\x)=F_2(\x)=0\\ \x\equiv \b\bmod{N}}}\omega(\x/P).
\end{equation}
We apply Theorem \ref{thm:split} to detect the condition $F_1(\x)=F_2(\x)=0$ to obtain the following asymptotic formula for $N(P)$:
\begin{theorem}
\label{thm:count}
Let $X\subset \PP^{n-1}_K$ be a smooth complete intersection of two quadrics over $K=\FF_q(t)$ satisfying $2\nmid q$. Then given any non-zero $P\in\cO$ and  $\omega$, the characteristic function of a fixed suitable hypercube around a non-singular point $\x_0\in K_\infty^n$, there exists an $\ve_0>0$ and a constant $C_{\omega,F}>0$ such that we have
$$N(P)=C_{\omega,F} |P|^{n-4}+O(|P|^{n-4-\ve_0}), $$
as long as $n\geq 9$.
\end{theorem}
We thus establish the quantitative arithmetic for a pair of quadrics in the setting of the aforementioned folklore conjecture requiring $n>d_1^2+...+d_R^2$ and thus also record an improvement of \cite{Munshi15} in the function field setting. The asymptotic formula, without the condition $C_{\omega,F}>0$, could essentially be proved for the characteristic function $\omega$ of any fixed hypercube in $K_\infty^n$. However, for it to be meaningful, we must have $C_{\omega,F}>0$. This can only be guaranteed as long as the hypercube is close enough to a smooth point $\x_0\in K_\infty^n$. The hypothesis $2\nmid q$ is vital as well. For a fixed value of $d\uc$ in \eqref{eq:SUM}, we consider the sum
$$\sum_{\substack{|r|=q^{Y}\\d\mid r}}\sum_{\ua/r\in L(d\uc) }S(\ua/r+\uz),$$
which translates to considering averages of one dimensional exponential sums corresponding to the quadratic form $-c_2F_1+c_1F_2$,  a well studied object. 
This allows us to save a factor of size $O(|r|)$ upon utilising the average over $\ua$ as well as over $r$, when $\uc$ is generic. This amounts to obtaining a double Kloosterman refinement over each line $L(d\uc)$. The readers familiar with the exponential sums for quadratic forms may notice that when $n$ is even, in a generic situation, we are in fact able to save a factor of size $O(|r|^{3/2})$, which in theory would let us take care of the $n=8$ situation. However, when $r$ consists of non-generic primes, this saving is reduced to a factor of size $O(|r|^{1/2})$ instead, rendering the $n=8$ case out of our reach. See Remark \ref{rem:1} for a further explanation.

Ideally, one would also like to obtain some extra  cancellations from the sum over different lines $d\uc$ of a fixed height in \eqref{eq:SUM}. However, so far, we have been unable to do so. The most straightforward generalisation of our mthod to a general $R$ situation is likely to facilitate us to save a factor of the size $O(|r|)$ in the square-free case. It would be interesting to see if it could be used to save more. One must also be wary of saving more than a factor of size $O(|r|)$ in the square-free case, as these bounds need to be matched by their square-full counterparts.   Typically, while dealing with the rationals with $\ell$-full denominators $r$, where $\ell$ is large, one gives up on any saving from the extra averages over $\ua$ and $r$ but instead benefits from the sparseness of $\ell$-full numbers. This standard trick would not be enough to save a factor of size $O(|r|^{1+\delta})$. 
However, as is a feature of these methods, the results produced by them would be more remarkable when the total degree $d_1+...+d_R$ is relatively low. 
 
Finally, there are wider implications of obtaining analogous asymptotic formulae (where $\deg(P)$ remains fixed but $q\rightarrow\infty$) for the counting function \eqref{eq:Count} over $\FF_q(t)$. Let $X$ be a smooth complete intersection over $\CC$. Studying the geometry of the space of rational curves on $X$ is essential in understanding the rationality properties of $X$. Let $\overline{\mathrm{Mor}}_{a,b}(X,e)$ denote the Kontsevich moduli space of rational curves of degree $e$ on $X$ passing through points $a$ and $b$. When $a=b=0$, this coincides with the space of all rational curves of degree $e$ on $X$. Following a strategy of Ellenberg and Venkatesh, a previous work of the author (with Browning) \cite{Browning_Vishe17-2} establishes facts about the irredicubility of these moduli spaces when $a=b=0$, $R=1$ and $n$ large enough with respect to the degree of the hypersurface $d$. These techniques can be pushed further to obtain other geometric results regarding the aforementioned moduli spaces and their generalisations, as demonstrated by \cite{Browning_Vishe17-2} and subsequent works by Browning and Sawin \cite{Browning_Sawin17}, \cite{Browning_Sawin18} and M\^{a}nz\u{a}\cb{t}eanu \cite{Manzateanu18}. A feature of these methods is that they are able to establish the results for all smooth $X$ satisfying the given conditions, while the geometric methods are usually able to hand us results for smooth and generic $X$.

Techniques in this work are likely to facilitate us to obtain the irreducibility of $\overline{\mathrm{Mor}}_{a,b}(X,e)$, when $e$ is sufficiently large as compared with $n$, as long as $n\geq 9$ and $X$ is a smooth complete intersection of two quadrics, for suitably chosen $a$ and $b$. The dependence on $n$ would be better for larger values of $n$, as also seen in \cite{Manzateanu18}. We plan to obtain this in a follow up work. To facilitate this, we have tried to keep an explicit dependence on $q$ in the constants appearing in a large portion of our auxiliary estimates obtained in this work.
 
\subsection{Acknowledgements} We would like to thank Tim Browning and Roger Heath-Brown for helpful discussions and providing us with useful references. Special thanks are also due to Will Sawin whose generous help is greatly acknowledged.
\section{Auxiliary results for $\FF_q(t)$}
The objective of this section is to state and prove various auxiliary results about $K=\FF_q(t)$ which will be useful at various junctures in this manuscript.

\subsection{Notation} \label{sec:notation}
We will follow the notation in \cite[Sec 2]{Browning_Vishe15} closely. We refer the reader there for the proofs and explanations of many of the facts stated below. We will always assume that $2\nmid q$.  Throughout this work, for any real number $R$, let $\hat{R}:=q^R$. Let $\cO=\FF_q[t]$ be the ring of integers of $K$, and let $\Omega$ denote the set of places of $K$ including the infinite place. Given any finite prime $v\in \Omega$, let $\nu_v(x):=\ord_v(x)$ denote the standard valuation. Each valuation $\nu_v$ gives rise to an absolute value $|\cdot|_v$ on $K$. Throughout, we will write $|\cdot|=|\cdot|_\infty$. For each $v\in\Omega$, let $K_v$ denote the completion of $K$ with respect to the absolute value $|\cdot|_v$, and let $\scrO_v=\{x\in K_v: |x|_v\leq 1\}$. We also define  $$\scrO^\sharp:=\{b\in\scrO:b\textrm{ monic, }\vp^2\nmid b,\forall \vp \textrm{ prime }\},$$
to be the set of monic, square-free integers in $\scrO$.

An important role will be played by $K_\infty$, which can be identified with the set of truncated Laurent series with the coefficients in $\FF_q$. We will set $\TT=\{x\in K_\infty: |x|<1\}$. Let $d\alpha$ denote the Haar measure on $K_\infty$, normalised so that $$\int_{\TT}d\alpha=1.$$
Let $\psi:K_\infty\rightarrow \CC^*$ denote the non-trivial unitary character as defined in \cite[Sec 2.2]{Browning_Vishe15}. Given any $\x\in K_\infty^m$ for any $m\geq 1$, let $|\x|=\max_i\{|x_i|\}$ denote the maximum norm of the co-ordinates of $\x$.

Given a polynomial $f(\x)\in K_\infty[x]$, let $H_f$ denote the maximum of the $\infty$-norms of the coefficients appearing in the equation of $f$. Similarly, given any tuple $\underline{f}(\x)=(f_1(\x),...,f_R(\x))$ of polynomials $f_1,...,f_R$, $H_{\underline{f}}$ will denote the maximum of $H_{f_1},...,H_{f_R}$.

To distinguish between the integral over $\TT^2$ and over $K_\infty^n$ appearing in our work later, we will typically use the notation $\ux=(x_1,x_2)$ to denote a pair in $K_\infty^2$, and the notation $\x=(x_1,...,x_n)$ to denote a vector in $K_\infty^n$, with our notation $\ud=\d$ defined in Sec. \ref{sec:minor} being an exception.

We say that $\ux=(x_1,x_2)\in\scrO^2\setminus \underline{0}$ is `monic' if $x_1\neq 0$ is monic or $x_1=0 $ and $x_2$ is monic.  As already defined before Theorem \ref{thm:split}, $\ux$ will be called primitive if $\ux$ is monic and $\gcd(x_1,x_2)=1$.

Let $C\in M_{k}(\scrO)$ be an arbitrary $k\times k$ matrix. We will frequently use a Smith normal form to write $C=TDS$, where $S,T\in \GL_k(\scrO)$ be matrices satisfying $\det(S),\det(T)\in\FF_q^\times$. Here $D=\diag(\mu_1,...,\mu_n)$ is a diagonal matrix satisfying $\mu_1\mid\mu_2\mid...\mid \mu_n$.

Our integral bounds would require us to often integrate on regions of the form $\{\uz\in K_\infty^2: |z_i|=\hat{Z_i}\} $, where $Z_i\in\ZZ$. We will therefore introduce the following notation: given $\uZ\in \ZZ^2$, let 
\begin{equation}\label{eq:uzUZ}
\{\langle \uz\rangle=\langle \hat\uZ\rangle\}:=\{\uz\in K_\infty^2: |z_i|=\hat{Z_i}\}.
\end{equation}
In order to facilitate our optimisation process in Sec. \ref{sec:minor}, given any $x,y\in \scrO$, we define:
\begin{equation}\label{eq:a||b}
 y\mid x^\infty\Rightarrow \{\vp\mid y\Rightarrow \vp\mid x \}.
\end{equation}
Throughout, we will use the notation $A\ll B$ to denote $A\leq CB$ for some absolute constant $C$. For a large portion of this work, we have tried to keep the implied constant to be independent of $q$, which will mainly be useful in our future applications to arithmetic geometry.

\subsection{Some exponential integral bounds}

Given non-zero polynomials  $G_1, G_2\in K_\infty[x_1,\dots,x_n]$,  given $\gamma\in \ki$ and $\w\in
\ki^n$,
 integrals of the form 
\begin{equation}\label{eq:J}
J_\G(\ualf;\w)=\int_{\TT^n} \psi\left( \alpha_1 G_1(\x)+\alpha_2G_2(\x)+\w.\x\right)d\x
\end{equation}
will feature prominently in our work. Our goal here will be to build on the results in \cite[Sec. 2.4]{Browning_Vishe15} and obtain analogues of Lemmas 2.6 and 2.7 from there.
Generalising \cite[Lemma 2.6]{Browning_Vishe15} is relatively straightforward. We will therefore omit its proof. After noting $H_{\alpha_1 G_1+\alpha_2G_2}\leq \max\{|\alpha_1|H_{G_1},|\alpha_2|H_{G_2}\}$, a slight modification of \cite[Lemma 2.6]{Browning_Vishe15} gives us
\begin{lemma}\label{lem:J-easy}
We have 
$
J_\G(\ualf;\w)=0$ if $|\w|> \max\{1,|\alpha_1| H_{G_1}, |\alpha_2| H_{G_2}\}$.
\end{lemma}
We also need a generalisation of \cite[Lemma 2.7]{Browning_Vishe15}, obtained in the following lemma:
\begin{lemma}\label{lem:small}
Given any $\uZ=(Z_1,Z_2)\in \ZZ^2$ and for any $\w\in K_\infty^n$  satisfying $|\w|\leq  \max\{1,\hat{Z_1}, \hat{Z_2}\}H_\G$, we have 
 $$
\int\limits_{\langle\ualf\rangle=\langle \hat{\uZ}\rangle} J_\G(\ualf;\w)d\ualf= \int_{\Omega} \psi\left(\alpha_1 G_1(\x)+\alpha_2G_2(\x)
+\w.\x \right) d\x d\ualf,
$$
where $\langle\ualf\rangle=\langle \hat{\uZ}\rangle$ as in \eqref{eq:uzUZ} and
\begin{equation}
\begin{split}
\Omega=\{(\ualf,\x)\in \{\langle \ualf\rangle=\langle \uZ\rangle\}\times \TT^n: &|\alpha_1G_1(\x)|,|\alpha_2G_2(\x)|\leq \max\{1,H_\G\}\max\{1,
\hat{Z_1}^{1/2},\hat{Z_2}^{1/2}\},\\& |\alpha_1\nabla G_1(\x)+\alpha_2\nabla G_2(\x)+ \w|\leq H_\G
\max\{1,
\hat{Z_1}^{1/2},\hat{Z_2}^{1/2}\}\}.
\end{split}
\end{equation}
\end{lemma}
Note that the new ingredient here, as compared with  \cite[Lemma 2.7]{Browning_Vishe15}, is provided by the condition $|\alpha_1G_1(\x)|,|\alpha_2G_2(\x)|\leq \max\{1,H_\G\}\max\{1,
\hat{Z_1}^{1/2},\hat{Z_2}^{1/2}\}$. This will be obtained by utilizing the extra average over $\ualf$ in the integral. This refined bound will be useful in the proof of Lemma \ref{lem:red}.

\begin{proof}
Without loss of generality, let us assume that $Z_1\geq Z_2$. We may also assume that $Z_1\geq 0$, since otherwise, the lemma is trivial. For now, we proceed with an extra assumption $Z_2\geq Z_1/2$. 
 Let $$\Omega_0=\{\langle \ualf\rangle=\langle \hat\uZ\rangle\}\times\TT^n\setminus \Omega.$$ We  break  the integral over $\Omega_0$ into a sum of  integrals over 
smaller  regions. Let $\delta\in K_\infty$ be such that 
$|\delta|=\hat{Z_1}^{-1/2}$. 
We introduce dummy sums over $\ua\in \{\langle \ualf\rangle=\langle \hat\uZ\rangle\}/(\delta^{-1} \TT)^2$ and $\y \in (\TT/\delta\TT)^n$. Here, the sum over $\ua$ will run through a fixed set of coset representatives of  $ \{\langle \ualf\rangle=\langle \hat\uZ\rangle\}/(\delta \TT)^2$. Using the
change of variables $\ualf=\ua+\delta^{-1}\ub$, $\x=\y+\delta\z$,  we obtain
\begin{equation}
\label{eq:omg}
\begin{split}
&\int_{\Omega_0} \psi\left( \ualf\cdot \G(\x)+\w.\x\right) d\x d\ualf\\
&=
\sum_{\ua}\sum_{\y \in (\TT/\delta\TT)^n}
\int_{\{(\ub,\z)\in \TT^{n+2}: (\ua+\delta^{-1}\ub,\y+\delta\z)\in \Omega_0\}}\psi\left((\ua+\delta^{-1}\ub)\cdot\G(\y+\delta\z)+\w\cdot (\y+\delta\z)
\right)d\z d\ub.
\end{split}
\end{equation} 
For a fixed value of $\y$ and $\ua$, for any $i=1,2$ and for any $|\ub|,|\z|<1$,
\begin{align*}
|(a_i+\delta^{-1} b_i)G_i(\y+\delta\z)-a_i G_i(\y)|<\max\{\hat{Z_1}|\delta|H_\G+H_\G/|\delta|\}= H_\G\hat{Z_1}^{1/2}.
\end{align*}
Thus, if for some $\ua$ and $\y$ we have
\begin{equation}\label{eq:aigi}
 |a_i G_i(\y)|\geq \hat{Z_1}^{1/2}\max\{1,H_\G\}, \textrm{ for some }i\in\{1,2\},
\end{equation} then this implies that the above holds for all $(\ua+\delta\ub,\y+\delta\z)$ for all $|\ub|,|\z|<1$, further implying that all these points belong to $\Omega_0$. For such a choice of $\ua $ and $\y$, the integral over $b_i$ could be evaluated separately. Using the orthogonality of additive characters on $K$ (see \cite[Sec. 2.1]{Browning_Vishe15}), for any $\y$ satisfying \eqref{eq:aigi}, we have
\begin{align*}
\int_{|b_i|<1}\psi(\delta^{-1} b_i G_i(\y+\delta \z))db_i=0, \textrm{ since } &|G_i(\y+\delta \z)|\geq \max\{1,H_\G\}\hat{Z_1}^{1/2}/\hat{Z_i}\geq \hat{Z_1}^{-1/2}= |\delta|
 \end{align*}
Thus, the contribution from the values of $\ua$ and $\y$ satisfying \eqref{eq:aigi} to the corresponding inner integrals in \eqref{eq:omg} is zero. We may now assume that for remaining $\ua,\y$ we must have $$|(a_i+\delta^{-1} b_i)G_i(\y+\delta\z)|\leq \max\{1,H_\G\}|\delta|^{-1},$$
for all $|\ub|,|\z|<1$ and for $i=1,2$. For $\ua$ and $\y$ satisfying the above condition, they appear in \eqref{eq:omg} only if for some $|\z_0|<1$ and for some $|\ub_0|<1$,
\begin{align*}
|(\ua+\delta^{-1}\ub_0)\cdot \nabla \G(\y+\delta\z_0)+\w|>H_\G/|\delta|.
\end{align*} 
Since, $|(\ua+\delta^{-1}\ub_0)\cdot \nabla \G(\y+\delta\z_0)-\ua\cdot\nabla \G(\y)|<H_\G/|\delta|$, we must further have
\begin{align*}
|\ua\cdot \nabla \G(\y)+\w|>H_\G/|\delta|\Rightarrow \forall |\ub|,|\z|<1, |(\ua+\delta^{-1}\ub)\cdot \nabla \G(\y+\delta\uz)+\w|>H_\G/|\delta|.
\end{align*} 
We may now emulate the recipe of \cite[Lemma 2.7]{Browning_Vishe15} and utilise the integral over $\z$ to obtain that the inner integral in \eqref{eq:omg} vanishes if $\ua,\y$ satisfy
\begin{equation}\label{eq:gut}
|\ua\cdot\nabla \G(\y)+\w|>H_\G|\delta|^{-1},
\end{equation}
which would further imply that $(\ua+\delta^{-1}\ub,\y+\delta\z)\in\Omega_0$ for all $(\ub,\z)\in\TT^{n+2}$ and thus the whole contribution from \eqref{eq:omg} is $0$.

Recall that throughout, we have assumed that $Z_2\geq  Z_1/2$. If $Z_2<Z_1/2$, then this automatically implies $|\alpha_2G_2(\x)|<H_\G\hat{Z_1}^{1/2}$, rendering this condition as vacuously true. We may now fix $\alpha_2$ and modify the above process by utilising the integrals over over $\alpha_1$ as well as over $\x$ to get the required bound.
\end{proof}

\subsection{Quadratic exponential sum bounds}\label{sec:quad}
The bounds for exponential sums corresponding to a quadratic polynomial will play a key part in our analysis. Throughout, let 
\begin{align}\label{eq:fdef1}
f(\x)=F(\x)+\f\cdot\x+m,\end{align} be a quadratic polynomial in $\scrO[x]$. Here, $F(\x)=\x^tM\x$ be the leading quadratic form defined by an $n\times n$ symmetric matrix $M$ with entries in $\scrO$ and with a non-zero determinant. Let
\begin{align}
\label{eq:f*def}
F^*(\v)=\det(M)\v^t M^{-1}\v,
\end{align} denote the dual form of $F$. Let 
\begin{equation}\label{eq:S}
S_r(\v)=\sumstar_{|a|<|r|}\sum_{|\x|<|r|}\psi\left(\frac{af(\x)-\v\cdot\x}{r}\right),
\end{equation}
denote a complete quadratic exponential sum. It is well known that as long as a prime does not divide $\det(M)$, $\QQ$-analogues of these sums could be explicitly evaluated modulo any power of such a prime. Our main goal here will be to establish this in the function field setting, the focus of Lemma \ref{lem:Expsum'} below.

We will first begin by obtaining explicit bounds for the function field avatars of the Gauss sums, $\tau_r$ defined below. Given $r\in\scrO$, let $$\tau_r=\sum_{x\bmod{r}}\psi(x^2/r).$$ 
\begin{lemma}
\label{lem:gauss}
Let $\vp$ be a prime such that $|\vp|=q^L$ and let $q=p^{\ell_0} $, then for any integer $k$,
\begin{align*}
\tau_{\vp^k}=\begin{cases}
|\vp|^{k/2} & \text{if }k \textrm{ is even},\\
-|\vp|^{k/2}i_p^{L\ell_0} & \text{if }k \textrm{ is odd},
\end{cases}
\end{align*}
where,
\begin{align}\label{eq:ipdef}
i_p=\begin{cases}
-1 & \text{if }p\equiv 1\bmod{4},\\
-i & \text{if }p\equiv 3\bmod{4}.
\end{cases}
\end{align}
\end{lemma}
\begin{proof}
Let $k_0=\lfloor (k-1)/2\rfloor$. We begin by writing
$$\tau_{\vp^k}=\sum_{|a_0|,...|a_{k-1}|<|\vp|}\psi((a_0+a_1\vp...+a_{k-1}\vp^{k-1})^2/\vp^k)=\sum_{|a_0|,...|a_{k-1}|<|\vp|}\psi\left(\vp^{-1}\sum_{i=0}^{k-1}a_ia_{k-1-i}\right).$$
 Since $2\nmid q$, for any fixed choice of $a_{k_0+1},...,a_{k-1}$, the sum on the right hand side vanishes unless $a_{k_0+1}=...=a_{k-1}=0$. Therefore, $\tau_r=|\vp|^{k/2}$, if $k$ is even, and 
$$\tau_{\vp^{k}}=|\vp|^{(k-1)/2}\sum_{|a|<|\vp|}\psi\left(\vp^{-1}a^2\right)=|\vp|^{(k-1)/2}\tau_\vp.$$
The lemma now follows from the standard bounds for  quadratic Gauss sums over finite fields, cf. \cite[Eq. (6)]{Evans81} for example.
\end{proof}

 The following lemma will follow a proof similar to \cite[Lemma 26]{Heath-Brown96}. 
\begin{lemma}\label{lem:Expsum'} Let $f$ be a quadratic polynomial as \eqref{eq:fdef1}. Let $\vp$ be a prime satisfying $\vp\nmid \det(M)$. Let $|\vp|=q^L$, and $q=p^{\ell_0}$.
Then $$S_{\vp^k}(\v)=\psi(\bar{2}\f^t M^{-1} \v)\left(\frac{\det(M)}{\vp^k}\right)\tau_{\vp^k}^nK_n(-\bar{4}F_1(\f)+m,-\bar{4}F_1(\v),\vp^k).$$
Here, $K_n$ denotes the Kloosterman sum when $n$ is even and the Sali\'e sum when $n$ is odd, and $F_1(\x)=\x^t M^{-1}\x$, where the inverse could be assumed to be taken modulo $\vp^k$. As a consequence,
\begin{align}\label{eq:firstB}
|S_{\vp^k}(\v)|\leq |\vp|^{(n+1)k/2}|\gcd(F^*(\f)-4\det(M)m,F^*(\v),\vp^k)|^{1/2}.
\end{align} More explicitly, when $\f=\vecnull$ and $m=0$, we have:
\begin{align*}
S_{\vp^k}(\v)=\begin{cases}
|\vp|^{nk/2}(|\vp|^k\delta_{\vp^k\mid F^*(\v)}-|\vp|^{k-1}\delta_{\vp^{k-1}\mid F^*(\v)}), & \text{if }2\mid k,\\
\left(\frac{\det(M)}{\vp}\right)|\vp|^{kn/2}i_p^{L\ell_0n}(|\vp|^k\delta_{\vp^k\mid F^*(\v)}-|\vp|^{k-1}\delta_{\vp^{k-1}\mid F^*(\v)}), &\textrm{ if } 2\mid n,2\nmid k,\\
\left(\frac{-F^*(\v)}{\vp}\right)|\vp|^{k(n+1)/2}i_p^{L\ell_0(n+1)}, &\textrm{ if } 2\nmid n,2\nmid k,
\end{cases}
\end{align*}
with $i_p$ as in \eqref{eq:ipdef}. 
\end{lemma}
\begin{proof}
Since $aF(\x + M^{-1} (\f/2-\v/2a))+am =  af(\x) - \v.\x + aF_1(\f)/4+F_1(\v)/4a-\f^t M^{-1} \v/2$, where $F_1(\v)=\v^t M^{-1}\v$ modulo $\vp^k$. Therefore by a suitable change of variables,
\begin{align*}
S_{\vp^k}(\v)=\psi(\f^t M^{-1} \v/2)\sumstar_{|a|<|\vp|^k}\psi(- F_1(\v)/4a-a(F_1(\f)/4-m))\sum_{|\x|<|\vp|^k}\psi\left(\frac{aF(\x)}{\vp^k}\right).
\end{align*}
At this point, we use the fact that since  $\vp\nmid \det(M)$, $M$ may be diagonalised, i.e., $M=R^t\textrm{Diag}(\beta_1,...,\beta_n)R$. After changing the variable again to $\y=R\x $
\begin{align*}
\psi(-\f^t M^{-1} \v/2)S_{\vp^k}(\v)&=\sumstar_{|a|<|\vp|^k}\psi(- F_1(\v)/4a-a(F_1(\f)/4-m))\prod_{i=1}^n\sum_{|y_i|<|\vp|^k}\psi\left(\frac{a\beta_iy_i^2}{\vp^k}\right)\\&=\left(\frac{\det(M)}{\vp^k}\right)\tau_{\vp^k}^n\sumstar_{|a|<|\vp|^k}\psi(- F_1(\v)/4a-a(F_1(\f)/4-m))\left(\frac{a}{\vp^k}\right)^n,\\
&=\left(\frac{\det(M)}{\vp^k}\right)\tau_{\vp^k}^n\sumstar_{|a|<|\vp|^k}\psi(- F_1(\v)/4a-a(F_1(\f)/4-m))\left(\frac{a}{\vp^k}\right)^n
\end{align*}
using some standard Gauss sum manipulations. We thus end up with
\begin{align*}
S_{\vp^k}(\v)=\psi(\bar{2}\f^t M^{-1} \v)\left(\frac{\det(M)}{\vp^k}\right)\tau_{\vp^k}^nK_n(-\bar{4}F_1(\f)+m,-\bar{4}F_1(\v),\vp^k),
\end{align*}
where when $n$ is even, $K_n$ denotes the Kloosterman sum, and the Sali\'e sum when $n$ is odd. Using a standard bound for the Kloosterman sums, we get
\begin{align*}
|S_{\vp^k}(\v)|&\ll|\vp|^{(n+1)/2}|\gcd(F_1(\f)-4m,F_1(\v),\vp^k)|^{1/2}\\
&\ll |\vp|^{(n+1)/2}|\gcd(F^*(\f)-4\det(M)m,F^*(\v),\vp^k)|^{1/2},
\end{align*}
where $F^*(\v)=\det(M)F_1(\v) $, as before.
In the special case when $\f=\vecnull, m=0$, the sums $K_n$ simplify. We will henceforth assume that $\f=\vecnull, m=0$. If $2\mid k$, Lemma \ref{lem:gauss} gives
\begin{align*}
S_{\vp^k}(\v)&=|\vp|^{nk/2}\sumstar_{|a|<|\vp|^k}\psi(F_1(\v)a)=|\vp|^{nk/2}(|\vp|^k\delta_{\vp^k\mid F_1(\v)}-|\vp|^{k-1}\delta_{\vp^{k-1}\mid F_1(\v)}).
\end{align*} 

Similarly, when $k$ is odd and $n$ is even,
\begin{align*}
S_{\vp^k}(\v)&=\left(\frac{\det(M)}{\vp}\right)|\vp|^{k/2}i_p^{L\ell_0n}(|\vp|^k\delta_{\vp^k\mid F_1(\v)}-|\vp|^{k-1}\delta_{\vp^{k-1}\mid F_1(\v)}),
\end{align*} 
where $i_p$ is defined by \eqref{eq:ipdef}. Lastly, when both $n,k$ are odd, then 
 $$K_n(0,-\bar{4}F_1(\v),\vp^k)=\sumstar_{|a|<|\vp|^k}\psi(-F_1(\v)a)\left(\frac{a}{\vp}\right)=\left(\frac{-F_1(\v)}{\vp}\right)\tau_{\vp^k}.$$ 

The final bound follows from applying Lemma \ref{lem:gauss}, along with the fact that $\det(M)F_1(\v)\equiv F^*(\v)\bmod{\vp^k}$.
\end{proof}
The above lemma although is powerful, it only works when $M$ is invertible and $\vp$ doesn't divide $\det(M)$. When this is not the case, we may supplement this using the following bound, which is obtained using a standard squaring argument (see \cite[(4.17)]{HeathBrown_Pierce17}): 
\begin{lemma}
\label{lem:pointwise}
Let $S=\sum_{|\x|<\vp^k}\psi\left(\frac{f(\x)}{\vp^k}\right)$, where $f(\x)=\x^tM\x+\f\cdot\x+m$ is any quadratic polynomial.  Then,
\begin{align*}
|S|\leq |\vp|^{nk/2}\# N(\vp^k)^{1/2}.
\end{align*}
where $N(\vp^k)=\#\{\x\bmod{\vp^k}:\vp^k\mid M\x\} $.
\end{lemma}
\begin{proof}
The lemma follows from essentially squaring the sum and applying a change of variable $\x_3=\x_1-\x_2 $
\begin{align*}
|S|^2=\sum_{\x_1,\x_2\bmod{\vp^k}}\left(\frac{f(\x_1)-f(\x_2)}{\vp^k}\right)\leq \sum_{\x_2\bmod{\vp^k}}\left|\sum_{\x_3\bmod{\vp^k}}\left(\frac{(M\x_2+\f)\cdot \x_3}{\vp^k}\right)\right|\leq |\vp|^{nk}N(\vp^k).
\end{align*} 
The last equality follows from the fact that the difference between any two solutions $\x_2'$ and $\x_2''$ of $M\x_2+\f\equiv\vecnull\bmod{\vp^k}$  satisfy the equation $M(\x_2'-\x_2'')\equiv \vecnull\bmod{\vp^k}$.
\end{proof}
\subsection{Integer points on affine hypersurfaces}
In this work, we will need to supplement the integer point counting estimate in \cite[Lemma 2.9]{Browning_Vishe15} with two others, obtained in Lemmas \ref{lem:B2} and \ref{lem:Bro} below.

 Let $F(x)$ be a non-singular quadratic form in $\scrO[x_1,...,x_n]$. We need an estimate on the number of integer solutions of $F(\x)=x_{n+1}^2$, with an explicit dependence on $H_F$. This will be obtained by producing a slight generalisation of an $\FF_q(t)$-analogue of \cite[Theorem 2]{HeathBrown02}. We start by proving an auxiliary result (cf. \cite[Theorem 3]{HeathBrown98}). The proofs of these results are almost straightforward adaptations of those of  Heath-Brown in $\FF_q(t)$ setting. Therefore, we shall be brief.

\begin{lemma}
\label{lem:B1}
Let $F$ be  a  non-singular  ternary  quadratic  form in $\scrO[x_1,x_2,x_3]$ such that the binary form $F(x_1,x_2,0)$
is  also  non-singular.  Then there exists an absolute constant $A$ such that for  any  $k\in\scrO$, the equation $F(x_1,x_2,k)=0$ has at most $O((\log(BH_{F}))^A)$ solutions satisfying  $|x_1|,|x_2|\leq B $. 
\end{lemma}
\begin{proof}
We may diagonalise $F$ using a matrix $M$ with entries in $r^{-1}\scrO$, for some $r\in\scrO$ satisfying $|r|\ll H_{F}^A$, for a fixed constant $A$. We may also choose the last row to be $(0\,\, 0\,\, 1)$. This transforms (after possibly multiplying by a power of $r$) $F(\x)=0$ to 
\begin{equation}\label{eq:Frede}
aL_1(x_1,x_2,k)^2+bL_2(x_1,x_2,k)^2=ck^2,\end{equation}
where $L_1$ and $L_2$ are linearly independent linear forms over $\scrO$, and  $|a|,|b|,|c|, \|L_1\|, \|L_2\|\ll H_{F}^{A'}$. The problem of bounding the number of solutions of \eqref{eq:Frede} can be easily converted to that of estimating the number of solutions for the equation $x^2+dy^2=e$, for a fixed choice of $d,e\in\scrO$.

The bound now follows from a standard bound for the number of elements of a specified norm in quadratic extensions of $K$.
\end{proof}
This leads to our first main estimate:
\begin{lemma}
\label{lem:B2}
Let $F(\x)$ denote a non-singular quadratic form in $n\geq 2$. Then there exists a constant $A$ such that given any $B>0$,
\begin{equation}\label{eq:B1}
\#\{F(\x)=x_{n+1}^2:|\x|\leq B\}\ll_{\ve,q}(\log(H_{F}B))^A B^{n-1}.
\end{equation}
\end{lemma}
\begin{proof}
Following the steps in \cite[Section 5]{HeathBrown02}, we may find $ M\in\GL_n(\scrO)$ satisfying $|M|\ll 1$, and \begin{equation}
\label{eq:T1}\det(M)T_{11}\det(T_{ij})_{1\leq i,j\leq 2}\det(T_{ij})_{1\leq i,j\leq 3}\neq 0.\end{equation}
Here, $T$ is the defining matrix of the quadratic form $f(\y)=F(M\y)$. Since if $F(\x)=x_{n+1}^2$ for some $\x\in \scrO^n$ and some $x_{n+1}\in\scrO$, then $(\det(M)M^{-1}\x,\det(M)x_{n+1})$ is a solution of $f(\y)=y_{n+1}^2$, to establish \eqref{eq:B1}, it is enough to bound the set  $\{f(\x)=x_{n+1}^2:|\x|\ll B\}$. For any choice of $\uu\in\scrO^{n-1}$, we now set 
$$Q_{\uu}(x,y,z):=f(y,z\uu)-x^2.$$
The determinant of the matrix defining this form is a quadratic polynomial $D(\uu)$, say. This does not vanish since $D((1,0...,0))=-\det(T_{ij})_{1\leq i,j\leq 2}\neq 0$. Moreover, the form $Q_{\uu}(x,y,0)$ is non-singular, since $T_{1,1}\neq 0$. We now set $z=1$. Thus, we would like to bound $$\{Q_{\uu}(x,y,1)=0: |x|\ll H_{F}B^{A'},|y|\ll |B|, |\uu|\ll B\},$$
for some constant $A'$.
For any fixed value of $D(\uu)\neq 0$, we may invoke Lemma \ref{lem:B1} to get that we only have $O((\log(H_{F}B))^{A})$ choices for $(x,y)$, which suffices. On the other hand, there are only $O(B^{n-2})$ choices for $D(\uu)=0$, and for each of those, there are at most $O(B)$ choices for the pair $(x,y)$. Combining these bounds, we establish the lemma.
\end{proof}
We will also need a bound for the number of integer solutions to the equation $F(x,y)=z^2$, where $F(x,y)$ is a square-free irreducible polynomial of even degree. This will be an $\FF_q(t)$-analogue of a very special case of \cite[Theorem 5]{Bro_03}. We have kept the $(\log \hat{Z})^2$ factor in our bound below to have the appearing constant independent of $q$.
\begin{lemma}
\label{lem:Bro}
Let $F(x,y)\in\scrO[x,y]$ be a homogeneous square-free polynomial of even degree $2d$ and let $Z\in \NN$ such that $H_{F}\leq \hat{Z}^{A}$ for some positive constant $A$, then for any $\ve>0$
$$\#\{F(x,y)=z^2:|x|,|y|<\hat{Z}, x,y,z\in\scrO\}\ll_{\ve,d,A}\hat{Z}^{1+\ve}(\log\hat{Z})^2.$$
\end{lemma}
\begin{proof}
The proof of this theorem resembles closely with that of \cite[Theorem 5]{Bro_03}. We shall therefore be brief. Let $F_1(x,y,z)=F(x,y)-z^2$. Since $F$ is irreducible, the discriminant $\Delta_F(x,y)$ is a non-zero polynomial of degree  $O_d(1)$. If $\Delta_F(x,y)=0$, then the bound $$\#\{ |x|,|y|<\hat{Z}:x\Delta_F(x,y)=0\}\ll_d \hat{Z},$$
is rather straightforward. It is therefore enough to establish the bound 
$$\#\{F(x,y)=z^2:|x|,|y|<\hat{Z}, x,y,z\in\scrO, x\Delta_F(x,y)\neq 0\}\ll_{A,d,\ve}\hat{Z}^{1+\ve}(\log\hat{Z})^2. $$
As in \cite[Lemma 4]{HeathBrown02}, for some $r=O_d(\lceil\log(H_{F}\hat{Z})\rceil)$, and for any $P\geq P_0=\log^2(H_{F}\hat{Z})$, there exist primes $\vp_1,...,\vp_r$ satisfying $P\ll_{d}|\vp_j|\ll_d P$ and 
\begin{align*}
\#\{F(x,y)=z^2:|x|,|y|<\hat{Z}, x,y,z\in\scrO,x\Delta_F(x,y)\neq 0\}\leq \sum_{i=1}^rN(F,Z,\vp_i),
\end{align*}
where
\begin{equation}
N(F,Z,\vp)=\#\{F(x,y)=z^2:|x|,|y|<\hat{Z}, x,y,z\in\scrO,\vp\nmid x\Delta_F(x,y)\}.
\end{equation}
We may therefore focus on bounding $N(F,Z,\vp)$ for a prime $\vp$ satisfying 
 \begin{equation}
 |\vp|=O((\log^2(\hat{Z})\hat{Z}^{1+\ve}),
\end{equation}
where the implied constant is $\geq 1$.
 Let $(x_1,y_1,z_1),...,(x_m,y_m,z_m)$ be all distinct pairs in $N(F,Z,\vp)$. For any $1\leq j\leq m$, we must have
\begin{align*}
|z_j|\leq \hat{Z}^{d+A/2}.
\end{align*}

 Let $f(u,v)=F(1,u)-v^2$. For every $1\leq j\leq m$, let $(u_j,v_j)=(y_j/x_j,z_j/x_j^d)$. Then  $(u_j,v_j)\in \scrO_\vp^2$ and $f(u_j,v_j)=0$. 
There are $O_d(|\vp|)$ solutions of $f(u,v)\bmod{\vp\scrO_\vp}$. Upon a possible re-labelling, we may assume that there exists $1\leq k$ such that $(u_j,v_j)\equiv (u_1,v_1)\bmod \vp\scrO_\vp$ for all $1\leq j\leq k$, and $(u_j,v_j)\not\equiv (u_1,v_1)\bmod \vp\scrO_\vp$ for all $j>k$. The lemma will now follow upon showing that $k=O_{d}(1)$. This is achieved by producing a polynomial $g(u,v)$ of degree $O_{A_1,d,\ve}(1)$ which is not divisible by $f(u,v)$, such that $g(u_1,v_1)=...=g(u_k,v_k)=0$. Since $f(u,v)$ is irreducible, we may then resort to Bezout's theorem to infer $k=O_{d,A,\ve}(1)$.

Let $D$ be the minimal positive integer satisfying  
\begin{equation}D> \max\{2,(4d+A-1)/\ve\},\label{eq:Dbo}\end{equation}
and let $(a_1,b_1),...,(a_{2D},b_{2D})$ be an enumeration of the set $\{0,...,D-1\}\times \{0,1\}$. Let 
$$M=[u_i^{a_j}v_i^{b_j}]_{1\leq i\leq k,1\leq j\leq 2D}, $$
be a $k\times 2D$ matrix with $\scrO_\vp$ entries. If the rank of $M<2D$, using the fact that $\scrO_\vp$ is complete, this must produce a non-trivial  polynomial $g(u,v)$ of degree at most $D$ in $\scrO_\vp[u,v]$, which is at most linear in $v$, such that $g(u_1,v_1)=...=g(u_k,v_k)=0$. Since $g$ is at most linear in $v$, it must not be a multiple of $f(u,v)$, which would prove the lemma. The result is obvious if $k\leq 2D$. We may therefore assume that $k> 2D$. It is enough to show that all $2D\times 2D$ minors of $M$ vanish. Without loss of generality, let
$$\Delta=\det[u_i^{a_j}v_i^{b_j}]_{1\leq i\leq 2D,1\leq j\leq 2D}. $$
We will show that $\Delta$ vanishes as long as $D$ satisfies \eqref{eq:Dbo}.

 Since $\vp\nmid \Delta_F(u_1,v_1)$, we may use the lifting argument of Hensel's Lemma \cite[Lemma 5]{HeathBrown02} to prove that $u_i\equiv h(v_i)\bmod{\vp^{4D^2}}$, for some polynomial $h(z)\in \scrO_\vp[z]$. Upon making some elementary column operations over $\scrO_\vp$ as in \cite[Page 201]{Broberg02}, we may further prove that $$\vp^{D(2D-1)}\mid\Delta. $$
On the other hand, if $\Delta\neq 0$, then since $(u_j,v_j)=(y_j/x_j,z_j/x_j^d)$, and that $\vp\nmid x_j$, the valuation
\begin{align*}
\nu_\vp(\Delta)=\nu_\vp\left(\det[x_i^{D+d-1}u_i^{a_j}v_i^{b_j}]_{1\leq i\leq 2D,1\leq j\leq 2D}\right)=\nu_\vp\left(\det[x_i^{D+d-1-a_j-db_j}y_j^{a_j}z_j^{b_j}]_{1\leq i\leq 2D,1\leq j\leq 2D}\right).
\end{align*}
Here, note that $|z_j|< \hat{Z}^{A/2+d}$, $|x_j|,|y_j|<\hat{Z}$, and $x_j,y_jz_j\in\scrO$. Thus, 
\begin{align}\label{eq:12}
|\vp|^{\nu_\vp(\Delta)}\leq \hat{Z}^{2D(D+d-1)+(A/2+d)2D}=\hat{Z}^{2D(D+2d+A/2-1)}.
\end{align}

On the other hand, the condition $\vp^{D(2D-1)}\mid\Delta$ implies 
\begin{align}\label{eq:11}
|\vp|^{\nu_\vp(\Delta)}\geq |\vp|^{D(2D-1)}\geq \hat{Z}^{D(2D-1)+\ve D(2D-1)}\geq \hat{Z}^{D(2D-1)+\ve D^2},
\end{align}
Since $D\geq 2$. \eqref{eq:12} and \eqref{eq:11} give a contradiction if $D>(4d+A-1)/\ve$.
\end{proof}
\subsection{Bounds for the character sums}\label{sec:Char sums} We will need a bound on twisted averages of the quadratic exponential sums in Section \ref{sec:quad} over square-free moduli. In the light of Lemma \ref{lem:Expsum'}, this is equivalent to obtaining a suitable bounds for one dimensional character sums. This fact will simplify our work immensely as compared with bounding the averages of cubic exponential sums considered in \cite[Sec. 3]{Browning_Vishe15}. 

We begin by making our setting more explicit. Let  $N\in \ZZ_{>0}$ and  let 
$$
\chi_{\mathrm{Dir}}: (\mathcal{O}_\infty/ t^{-N} \mathcal{O}_\infty)^* \to \CC^*
$$ 
be a Dirichlet character.
Putting $x=t^{-1}$ and 
$A=\FF_q[x]$, we note that 
$(\mathcal{O}_\infty/ t^{-N} \mathcal{O}_\infty)^*\cong (A/ x^{N} A)^*$. As in \cite[Sec. 3.5]{Browning_Vishe15}, given $a\in K^*$ and $u\in \prod_{\vp} \scrO_\vp^* $, we may now define a Hecke character 
$\chi_{\mathrm{Hecke}}: I_K\to \CC^*$ via
$$
\chi_{\mathrm{Hecke}}(au)=\chi_{\mathrm{Dir}}(u_\infty).
$$
It is  constant on $K^*$ and gives a character on the id\`ele class group $I_K/K^*$. Using this construction, the first relevant character for us is  $\eta:\mathcal{O}\rightarrow \CC^*$, given by  
$$
\eta(r)=\chi_{\mathrm{Dir}}(r/t^{\deg r} )
$$
for any $r\in \mathcal{O}$.
Note that 
$r/t^{\deg r}\in \mathcal{O}_\infty^*$ for any $r\in \mathcal{O}$.
The second is a 
Dirichlet character $$\eta':(\cO/y\cO)^*\to \CC^*$$ modulo $y$, for some $y\in \scrO$. Let $Y=\deg(y)$.
 Our ultimate goal will be to establish the following bound for a character sum:
\begin{lemma}
\label{lem:Charsum}
Let $\eta$ and $\eta'$ be Hecke characters as above such that $\eta\otimes \eta'(x)$ is not equal to $|x|^{ib}$, for any $b\in\RR$. Let $\beta=\pm 1$ and given any $x\in\scrO$, let $\Omega(x)$ denote the number of prime factors of $x$ including their multiplicities. Let $S\subset\{b\in\scrO^\sharp: |b|\leq\hat{Z}\} $ be a subset of square-free integers of cardinality at most $O(Z)$. Then given any $\ve>0$,
\begin{align*}
\left|\sum\limits_{\substack{b\in\scrO^\sharp,|b|\leq \hat{Z}\\\gcd(b,S)=1}}\beta^{\Omega(x)}\eta(b)\eta'(b)\right|\ll_\ve \hat{Z}^{1/2+\ve}\hat{N+Y}^\ve.
\end{align*}
\end{lemma}
The proof of this result is standard and will follow that of \cite[Lemmas 3.4, 3.5]{Browning_Vishe15} closely.  To keep this paper self contained, we will include it here. We consider the Hecke $L$-function 
$$L(\eta\otimes\eta',s):=\sum_{x\in \scrO, x\textrm{ monic }}\frac{\eta(x)\eta'(x)}{|x|^s}.$$ This Dirichlet series is a-priori convergent for $\sigma:=\re(s)>1$. However, due to Tate's thesis, this function has a meromorphic continuation to the whole complex plane.  Moreover, Tate's thesis also implies that it is entire unless $\eta(x)\eta'(x)=|x|^{ib} $, for some real number $b$. As a consequence, unless $\eta(x)\eta'(x)=|x|^{ib}$, 
\begin{equation}\label{eq:Riemann}
L(\eta\otimes\eta',s)=P(q^{-s})=\prod_{j=1}^{N+Y}(1-\alpha_jq^{-s})
\end{equation}
is a polynomial of degree at most $N+Y$, with $|\alpha_j|=q^{1/2}$. This a standard fact about the Hecke $L$-functions over $\FF_q(t)$. We will give an outline of how it can be proved. The fact that the $L$-function is a polynomial of degree $O(N+Y)$ follows from proving that the averages $\sum_{|r|=\hat{R}}\eta(r)\eta'(r)$ vanish as long as $R\gg N+Y$. If $\eta'$ is non-trivial, note that the value of $\eta(r)$ only depends on the top $N$ coefficients appearing in the expression for $r$ as a polynomial in $\FF_q[t]$. One may thus write $r=t^{R-N}r_1+r_2$, and treat $r_1$ as fixed and average over $r_2$, which must vanish as long as $R-N\geq Y$, (see \cite[Prop 4.3]{Rosen}). If $\eta'$ is trivial, then $\eta$ must be non-trivial and this strategy can be recycled by working with $\eta$ instead. Further, $|\alpha_j|=q^{1/2}$, since the zeroes of this $L$-function lie on the $s=1/2$ line. \eqref{eq:Riemann} is a key in the proof of Lemma \ref{lem:Charsum}.

Since we are interested in a sum over square-free values, we proceed to study the Dirichlet series
\begin{align*}
F(s)=\sum\limits_{\substack{b\in\scrO^\sharp\\\gcd(b,S)=1}}\frac{\beta^{\Omega(b)}\eta(b)\eta'(b)}{|\vp|^s}=\prod_{\vp\notin S}\left(1+\frac{\beta\eta\otimes\eta'(\vp)}{|\vp|^s}\right).
\end{align*}
We will begin by obtaining a satisfactory bound for $|F(s)|$ for $\mathrm{Re}(s)=\sigma\geq 1/2+\ve$. This will be done in a manner completely analogous to \cite[Lem 3.4]{Browning_Vishe15}. We will obtain a good bound for $\sigma>1$, and a weaker bound for $\sigma>1/2$. The final bound will follow from a use of the Hadamard three circle theorem. We begin by noticing that, for $\sigma>1$, we have
\begin{equation}
\label{eq:Fsbound0}
|F(s)|\leq \zeta_K(\sigma),
\end{equation}
where $\zeta_K$ is the usual zeta function for $K=\FF_q(t)$.
Moreover, for any prime $\vp$ we have
\begin{align*}
1+\frac{\eta\otimes\eta'(\vp)}{|\vp|^s}=\left(1-\frac{\eta\otimes\eta'(\vp)}{|\vp|^s}\right)^{-1}\left(1+O\left(\frac{1}{|\vp|^{2\sigma}}\right)\right),
\end{align*}
leading us to
\begin{equation}
\label{eq:FsEs}
F(s)=\begin{cases}L(\eta\otimes \eta',s)E(s),&\textrm{ if }\beta=1\\
L(\eta\otimes\eta',s)^{-1}E(s),&\textrm{ if }\beta=-1,
\end{cases}
\end{equation}
where
\begin{equation}
\label{eq:Esdef}
E(s)=\prod_{\vp\notin S}(1+O(|\vp|^{-2\sigma}))\prod_{\vp\in S}(1+O(|\vp|^{-\sigma})).
\end{equation}
Using \eqref{eq:Esdef}, $E(s)$ is holomorphic in the half plane $\sigma>1/2$. Moreover, taking a logarithm of both sides, for any $\sigma\geq 1/2+\ve$, $ \ve>0$, it is easy to establish
\begin{equation}
\label{eq:Esbound}
\log |E(s)|\ll \log \zeta_K(2\sigma)+Z.
\end{equation}
Here the implied constant only depends on $\ve$ and is independent of $q$. Similarly, using \eqref{eq:Riemann}, we may obtain
\begin{align*}
\log |L(\eta\otimes\eta',s)|\ll (Y+N)|\log(1+q^{1/2-\sigma})|\ll Y+N.
\end{align*}
Combining this bound with the one in \eqref{eq:Esbound}, we obtain that for any $\sigma\geq 1/2+\ve$, 
\begin{equation}
\label{eq:Fbound1}
\log |F(s)|\ll \log \zeta_K(2\sigma)+Z+ Y+N.
\end{equation}
Note that since $1/E(s)$ is also analytic, and since the zeroes of $L(\eta\otimes \eta',s)$ lie on the $\sigma=1/2$ line, $\log F(s)$ is analytic in the half plane $\sigma>1/2$. Moreover,
\begin{align}
\label{eq:Fbound2}
\re(\log F(s))=\log |F(s)|\ll_\ve \log \zeta_K(2\sigma)+Z+ Y+N.
\end{align}
The rest of the argument will follow exactly from the one in \cite[Lem 8.4]{Browning_Vishe15}. Therefore, we will only sketch the idea here. First, Borel Carath\'eodory theorem can be used to bound $|\log F(s)|$ using our bound \eqref{eq:Fbound2} for $\re(\log F(s))$. This obtains a weaker bound for $|F(s)|$ when $\sigma\geq 1/2+\ve$. Then the Hadamard's three circle theorem can be used to obtain the following Lindel\"of type bound:
\begin{equation}
\label{eq:Fboundfinal}
|F(s)|\ll c(\ve)^{(Z+N+Y)^{1-\ve/2}}\ll_\ve (\hat{Z+N+Y})^\ve,
\end{equation}
for some absolute constant $c(\ve)$. 
\begin{proof}[Proof of Lemma \ref{lem:Charsum}] Perron's formula implies that the sum we need to estimate is equal to
\begin{align}
\label{eq:Ffinal1}
\frac{a_k}{k^{1/2}}=\frac{1}{2\pi i}\int_{2-i\infty}^{2+i\infty}F(s)\frac{\hat{Z}^sds}{s},
\end{align}
where $a_k=\sum\limits_{\substack{b\in\scrO^\sharp, |b|=k\\\gcd(b,S)=1}}\beta^{\Omega(b)}\eta(b)\eta'(b)$. The right hand side of \eqref{eq:Ffinal1} may be rewritten as
\begin{align*}
\int_{2-iT}^{2+iT}F(s)\frac{\hat{Z}^sds}{s}+O\left(\frac{\hat{(Z+N+Y})^\ve\hat{Z}^3}{T}\right),
\end{align*}
for any $T>0$. Using \eqref{eq:FsEs} and the fact that $L(\eta\otimes\eta',s)$ is an entire function with all its zeros lying on the line $\re(s)=1/2$, $F(s)$ is holomorphic in the half plane $\sigma>1/2$, the integral over the line joining $2-iT$ and $2+iT$ may be replaced by that of the three remaining sides of the rectangle joining $2+iT,1/2+\ve+iT, 1/2+\ve-iT, 2-iT$. The integral over horizontal sides can be bounded by
$$\frac{(\hat{Z+N+Y})^\ve\hat{Z}^2}{T}.$$
The remaining line segment joining $1/2+\ve-iT$ and $1/2+\ve+iT$ satisfies the bound
$$\ll \hat{Z}^{1/2+\ve}\hat{(Z+N+Y})^\ve \int_{|t|\leq T}(1+|t|)^{-1}dt\ll \hat{Z}^{1/2+\ve}\hat{(Z+N+Y})^\ve T^\ve.$$
Upon choosing $T=\hat{Z}^3$, we obtain the statement of the Lemma.
\end{proof}
\section{Proof of Theorem \ref{thm:split}}\label{sec:split}
The focus of this section is to prove Theorem \ref{thm:split}, the result providing us with a partition of $\TT^2$. Here is an outline of the proof. In Lemma \ref{lem:Diri}, we will begin by first showing that each rational point $\ua/r\in\TT^2$ lies on a line $L(d\uc)$ of a suitable height. Lemmas \ref{lem:Dio1dim} through \ref{lem:Diogen} establish the precise distribution of rational points on individual lines $L(d\uc)$. This essentially follows from the one dimensional Dirichlet approximation theorem. Later, Lemma \ref{lem:Dio3} establishes that the lines $L(d\uc)$ stay sufficiently far away from one another. Theorem \ref{thm:split} is proved by combining all these ingredients together. 

Throughout this section, just for the sake of convenience of the notation, we will treat the tuples $\ux\in K_\infty^2$ as column vectors (instead of the row vector notation used in Sec. \ref{sec:notation} and the rest of the paper). This choice makes little difference to the analysis in the remaining sections, where $\ux\in K_\infty^2$ can be purely viewed as a tuple (either a row vector or a column).

We start by recalling the definition of lines $L(d\uc)$:
\begin{equation*}
L(d\uc):=\{\ua/r\in \TT^2\cap L_1(d\uc,k):\textrm{ where } k\in \scrO, \gcd(a_1,a_2,r)=\gcd(d,k )=1\},
\end{equation*}
where $L_1(d\uc,k)$ denotes the affine line defined by the equation $d\uc\cdot\ux=k$. Note that
 \begin{equation}\label{eq:d|r}
 \ua/r\in L(d\uc)\Rightarrow d\uc\cdot\ua=kr, \textrm{ where }\gcd(k,d)=1\Rightarrow d\mid r.
\end{equation}

We first start by proving that every rational pair $\ua/r$ satisfying $\gcd(\ua,r)=1$ lies on one of the lines of suitable height.
\begin{lemma}
\label{lem:Diri}
Given any rational $\ua/r$ satisfying $\gcd(\ua,r)=1$, there exists a primitive $\uc=\cmatr{c_1}{c_2}\in \scrO^2$ and a monic $d\in \scrO$ satisfying $|dc_1|\leq |r|^{1/2}, |dc_2|<|r|^{1/2}$, such that $\ua/r\in L(d\uc)$.
\end{lemma}
\begin{proof}
Let $|r|=q^L$. We will start by proving the existence of a possibly non-primitive vector $\uc_1$ such that $r\mid\uc_1\cdot\ua $.  Using the fact that for any $N\in\NN$, $\#\{x\in\scrO:|x|<\hat{N}\}=\hat{N}$, we have $$q^L<\#\{(c_1,c_2):|c_1|\leq \hat{L/2}, |c_2|<\hat{L/2}\}=q^{L+1}.$$ Therefore, for any triple $(a_1,a_2,r) $, at least two distinct elements in $\{c_1a_1+c_2a_2:|c_1|\leq \hat{L/2},|c_2|<\hat{L/2}\}$ must have the same residue modulo $r$. This implies that $\ua\cdot\uc_1=kr$ for some $\underline{0}\neq \uc_1\in\scrO^2, k\in\scrO$, satisfying the required bound on the size of the co-ordinates of $\uc_1$. If $\uc_1$ is not primitive, let $d=\gcd(\uc_1,r)$. Let $d'=\gcd(\uc_1)/d$ and $\uc=\uc_1/\gcd(\uc_1)$, where upon possibly multiplying by a unit, we may ensure that $\uc$ is monic as well.  Note that $\gcd(d',r/d)=1$.   We then have $$\ua\cdot \uc_1\equiv 0\bmod{r}\Rightarrow \ua\cdot d'\uc\equiv 0\bmod{r/d}\Rightarrow \ua\cdot \uc\equiv 0\bmod{r/d}\Rightarrow \ua\cdot d\uc\equiv 0\bmod{r}. $$
 We have now proved that $d\uc\cdot \ua=k_1r$ for some $k_1\in \scrO$, where $\uc$ is primitive. If $\gcd(d,k_1)=1$, then we are done. Otherwise, if $d_2=\gcd(d,k_1)$, then note that $\ua/r\in L_1((d/d_2)\uc,k_1/d_2)$, which further implies that $\ua/r\in L((d/d_2)\uc)$. The required bound for the coordinates of $d\uc$ follows from further observing $|d/\gcd(\uc_1)|\leq 1$.
\end{proof}
 We next prove a refinement of the one dimensional Diohantine approximation \cite[Lemma 4.2]{Browning_Vishe15}:
\begin{lemma}
\label{lem:Dio1dim}
Given any $a,r\in\scrO$ such that $\gcd(a,r)=1$ and $|r|=\hat{M}$, there exists $a_1,r_1$, such that $|r_1|=\hat{M+1}$, $\gcd(a_1,r_1)=1$ and $|a/r-a_1/r_1|=\hat{-2M-1}=(|r||r_1|)^{-1}$.
\end{lemma}
\begin{proof}
The proof is a direct consequence of \cite[Lemma 4.1]{Browning_Vishe15}. Let $y=a/r+z$, where $z=t^{-2M-1}$. For any $a',r'$ such that $a'/r'\neq a/r$, $|r'|\leq |r|$, note that $|y-a'/r'|\geq \hat{M}^{-1}|r'|^{-1}$. However, a further application of \cite[Lemma 4.1]{Browning_Vishe15} produces $a_1,r_1$, satisfying $|r_1|\leq \hat{M+1}$ and $|a_1/r_1-y|<\hat{M+1}^{-1}|r_1|^{-1}$. Clearly, $|r_1|=\hat{M+1}$ by our earlier observation. This implies that $|a_1/r_1-y|<\hat{M+1}^{-2}$. A simple triangle inequality establishes the lemma.
\end{proof}
We now investigate the structure of the rational points on each individual line, starting with a line $L(\uc)$, where $\uc$ is primitive.
\begin{lemma}
\label{lem:Diomain}
Let $c_1,c_2,r\in \scrO $ satisfying $\gcd(c_1,c_2,r)=1$. Then we have the following equality of residues modulo $r$:
 $$\{\ua\bmod{r}: \gcd(\ua,r)=1,\uc\cdot\ua\equiv 0\bmod{r}, |\ua|<|r|\}=\{a\uc^\bot\bmod{r}: |a|<|r|, \gcd(a,r)=1\},$$
 where $\uc^\bot=(-c_2,c_1)^t$.
\end{lemma}
\begin{proof}
We will assume that $r=\vp^k$, for some prime $\vp$. Without loss of generality, we can assume that $\vp\nmid c_1$. Clearly, modulo $r$, the left hand side is equal to $$\{y(-c_1^{-1}c_2,1):\gcd(y,r)=1, y\bmod{r}\}=\{yc_1(-c_1^{-1}c_2,1):\gcd(y,r)=1, y\bmod{r}\}.$$ In general, if $r=\vp_1^{k_1}...\vp_m^{k_m}$ is a prime decomposition of $r$ into co-prime prime powers, then our previous analysis shows that $\uc\cdot \ua\equiv 0\bmod{r}$ would necessarily imply that for each $1\leq i\leq m$, there exists $b_i$ such that $\vp_i\nmid b_i$ and $\ua\equiv b_i\uc^\perp \bmod{\vp_i^{k_i}}$ . An application of the Chinese remainder theorem will finish the proof of the lemma.
\end{proof}
As a direct corollary of Lemma \ref{lem:Diomain}, we get
\begin{corollary}
\label{cor:11}
For every $\ua/r\in L(\uc)$, there exists a unique $|a|<|r|$, $\gcd(a,r)=1 $ and a unique $\ud\in\scrO^2$ satisfying $|\ud|<|\uc|$ and $\ua/r=a\uc^\bot/r+\ud$.
\end{corollary}
Similarly for any general line $L(d\uc)$, we have the following generalisation:
\begin{lemma}
\label{lem:Dio}
Let $\uc\in\scrO^2$ be primitive and $d\in\scrO$. Then, for every $\ua/r\in L(d\uc),$ there exists a unique $\ua'/(r/d)\in L(\uc)$ and a unique $\ud'\in \scrO^2$ satisfying $|\ud'|<|d|$ such that $\ua/r=\ua'/r+\ud'/d $, where $|\ud'|<|d| $. Consequently, $\ua/r=a\uc^\bot/r+\ud/d$, where $(a,r/d)=1,|a|<|r/d|,\gcd(\ud,d)=1,\ud\in\scrO^2$.
\end{lemma}
\begin{proof}
We begin by recalling that $\ua/r\in L(d\uc)$ implies that $d\uc\cdot\ua=kr$, where $\gcd(k,d)=1$. Thus, $\uc\cdot\ua\equiv 0\bmod{r/d}$. The first part of the lemma is established upon choosing $|\ua'|<|r/d| $ such that $\ua'\equiv \ua\bmod{r/d}$. This choice of $\ua'$ is also unique, since any representation $\ua/r=\ua_1/r+\ud_1/d$ must satisfy $\ua\equiv \ua_1\bmod{r/d}$.

Corollary \ref{cor:11} implies that $\ua'/(r/d)=a\uc^\bot/(r/d)+\ud''$, for some $\ud''\in \scrO^2$, $\gcd(a,r/d)=1$. Thus, $\ua/r=a\uc^\bot/r+\ud/d$, for some $\ud\in\scrO^2$. This implies that $d\uc\cdot (a\uc^\bot/r+\ud/d)=\uc\cdot\ud=k$. Since $(k,d)=1$, $\gcd(\ud,d)=1$.
\end{proof}
As a consequence of the previous lemmas, we are now set to establish results about the distribution of rational points on the generalised lines $L(d\uc) $. As before, we start by investigating the lines of the type $L(\uc)$. The following lemma is a consequence of the one dimensional Dirichlet approximation.
\begin{lemma}
\label{lem:Dio2}
Let $\uc$ be primitive and let $\ua_1/r_1\neq \ua_2/r_2 \in L(\uc)$ satisfying $|\uc|^2\leq|r_1||r_2|$, then $$|\ua_1/r_1-\ua_2/r_2|\geq \frac{|\uc|}{|r_1||r_2|}.$$ Moreover, given any $\ua/r\in L(\uc)$ satisfying $|\uc|^2\leq |r|$, there exist  $\ua_1/r_1\in L(\uc) $ satisfying $|r|<|r_1|$ and $$|\ua/r-\ua_1/r_1|=\frac{|\uc|}{|r||r_1|}.$$
We can further guarantee that $\ua/r$ and $\ua_1/r_1 $ both lie on the line $L_1(\uc,k)$, for some $k\in\scrO$.
\end{lemma}
\begin{proof}
We begin by proving the first part of the lemma. Since $\ua_1/r_1,\ua_2/r_2\in L(\uc)$, we have $(\ua_i/r_i)\cdot \uc=k_i$, for $k_1,k_2\in \scrO$. Thus, $(\ua_1/r_1-\ua_2/r_2)\cdot \uc=k_1-k_2$. If $k_1\neq k_2$, then this implies that $|\ua_1/r_1-\ua_2/r_2|\geq |\uc|^{-1} $. The first part now follows from the condition on $r_1,r_2$ and $\uc$. On the other hand, Corollary \ref{cor:11} implies that $\ua_1/r_1=a_1\uc^\bot/r_1+\ud_1$ and $\ua_2/r_2=a_2\uc^\bot/r_2+\ud_2$. As a result, if $k_1=k_2$, then this necessarily implies $(\ud_1-\ud_2)\cdot \uc=0$. Now we use the fact that $\uc$ is primitive, along with the fact that $|\ud_1|,|\ud_2|<|\uc|$ to get that $\ud_1=\ud_2$. The first part now follows from the observation $ |\ua_1/r_1-\ua_2/r_2|=|(a_1/r_1-a_2/r_2)(-c_2,c_1)|$.

To prove the second part, we appeal to Lemma \ref{lem:Dio1dim}. Suppose, $\ua/r=a(-c_2,c_1)/r+\ud$. Lemma \ref{lem:Dio1dim} provides  us $a_1/r_1$ such that $|r_1|=q|r|$ and $|a/r-a_1/r_1|=(|r||r_1|)^{-1}$. Now, let $\ua_1/r_1=a_1(-c_2,c_1)/r_1+\ud$. Clearly  $|\ua/r-\ua_1/r_1|=|(a/r-a_1/r_1)(-c_2,c_1)|=\frac{|\uc|}{|r||r_1|}<1$. Thus, $\ua_1/r_1\in\TT$. We must also have $\gcd(\ua_1,r_1)=1$, since $\ua_1\equiv a_1\uc^\perp\bmod{r_1}$. We thus have $\ua_1/r_1\in L(\uc)$. The final part of the lemma follows from choosing $k=\ud\cdot \uc$.
\end{proof}
We further extend this result to the lines of general type:
\begin{lemma}
\label{lem:Diogen}
Let $\uc$ be primitive, let $d\in \scrO$ and let $\ua_1/r_1\neq \ua_2/r_2 \in L(d\uc)$ satisfying $|d\uc|^{2}\leq |r_1||r_2|$, then $$|\ua_1/r_1-\ua_2/r_2|\geq \frac{|d\uc|}{|r_1||r_2|}.$$ Moreover, given any $\ua/r\in L(d\uc)\cap L_1(d\uc,k)$, where   $|d\uc|^2\leq |r|$, there exists $\ua_2/r_2\in L(d\uc)\cap L_1(d\uc,k) $ satisfying $|r|< |r_2|$ such that $$|\ua/r-\ua_2/r_2|=\frac{|d\uc|}{|r||r_2|}.$$

\end{lemma}
\begin{proof}
The first part is almost immediate from Lemmas \ref{lem:Dio} and \ref{lem:Dio2}. The first part of Lemma \ref{lem:Dio} implies that $\ua_i/r_i=\ua_i'/(dr_i/d)+\ud_i'/d$, where $\ua_i'/(r_i/d)\in L(\uc) $, for $i=1,2$. Thus, $$\frac{\ua_1}{r_1}-\frac{\ua_2}{r_2}=\frac1d\left(\frac{\ua_1'}{r_1/d}-\frac{\ua_2'}{r_2/d}\right)+\frac{\ud_1'-\ud_2'}{d}. $$
The second term is clearly bigger than the first one on the right side of the above expression, except when $\ud_1'=\ud'_2$, since $|d\uc|/(|r_1r_2|)\leq 1/|d|$, the bound $1/|d|$ is admissible. This leaves us with the case $\ud_1'=\ud'_2$. We use Lemma \ref{lem:Dio2} to get $\left|\frac{\ua_1'}{r_1/d}-\frac{\ua_2'}{r_2/d}\right|\geq \frac{|d^2\uc|}{|r_1||r_2|}$, which settles this part.

For the second part, we again begin by applying the first part of Lemma \ref{lem:Dio} to write $\ua/r=\ua'/r+\ud/d$, where $\ua'/(r/d)\in L(\uc) $. We next use the second part of Lemma \ref{lem:Dio2}, to obtain $\ua_1/r_1\in L(\uc)$ satisfying $|r/d|<|r_1|$, $|\ua'/(r/d)-\ua_1/r_1|=d/(|r||r_1|) $ and $\ua'\cdot\uc=\ua_1\cdot\uc $. Set $\ua_2/r_2=\ua_1/(r_1d)+\ud/d$. Clearly, $$(\ua_2/r_2)\cdot d\uc=\uc\cdot\ua_1/r_1+\ud\cdot \uc=\uc\cdot\ua'/(r/d)+\ud\cdot \uc=k.$$ Since $\gcd(d,k)=1$, it follows that $d\mid r_2 $. This implies that $\ua_2/r_2\in L(d\uc)\cap L_1(d\uc,k)$. Moreover,
\begin{align*}
\left|\frac{\ua}{r}-\frac{\ua_2}{r_2}\right|=\left|\frac{1}{d}\left(\frac{\ua'}{r/d}-\frac{\ua_1}{r_1}\right)\right|=\frac{|\uc|}{|r||r_1|}\leq\frac{|d\uc|}{|r||r_2|}.
\end{align*}
The last inequality comes from the fact here that $|dr_1|\geq |r_2|$. However, since $|d\uc|^2\leq |r|$, the first part of the lemma is applicable. This gives  $\left|\frac{\ua}{r}-\frac{\ua_2}{r_2}\right|\geq \frac{|d\uc|}{|r||r_2|}$, which implies the equality, and that $r_2=dr_1$.
\end{proof}

We are now almost ready to prove the fact that the lines $L(d\uc)$ stay sufficiently far away from one another, cf. Lemma \ref{lem:Dio3} below. We will start with proving an auxiliary result.

\begin{lemma}
\label{lem:det0}
Let $C\in M_{ 2}(\scrO)$ be a matrix satisfying that $\varpi\nmid C$ , for some prime $\varpi\in \scrO$. Let $\nu_\varpi(\det(C) )=k_0$, then for any $k\in \NN$, if $k>k_0$ we have $$\{\ua\bmod \varpi^k: \gcd(\ua,\varpi)=1, C\ua\equiv\underline{0}\bmod{\varpi^{k}}\}=\emptyset.$$
\end{lemma}
\begin{proof}
Let $C= TDS$ be a Smith normal form of $C$. The matrices $S,T\in \GL_2(\scrO)$ and $D=\smatr{d_1}{0}{0}{d_2}$ is a diagonal matrix. Clearly, $\nu_\vp(d_1d_2)=k_0$. Since $S$ and $T$ are invertible modulo $\varpi$, $\gcd(S\ua,\vp)=1\iff \gcd(\ua,\vp)=1$. We thus have the equality:
$$\#\{\ua\bmod \varpi^k: \gcd(\ua,\varpi)=1, C\ua\equiv\underline{0}\bmod{\varpi^{k}}\}=\#\{\ua\bmod \varpi^k: \gcd(\ua,\varpi)=1, D\ua\equiv\underline{0}\bmod{\varpi^{k}}\}.$$
The right hand side is empty, as $\nu_\varpi(d_1d_2)=k_0<k$.
\end{proof}

\begin{lemma}
\label{lem:Dio3}
Let $\uc_1=\cmatr{c_1}{c_2}, \uc_2=\cmatr{c_3}{c_4}\in\scrO^2$ be two primitive vectors, and let $d_1,d_2\in \scrO $ be monic such that there are points $\ua_1/r_1\in L(d_1\uc_1),\ua_2/r_2\in L(d_2\uc_2)$, satisfying $|d_1\uc_1|^2\leq |r_1|$ and $|d_2\uc_2|^2\leq |r_2| $, and $$\left|\ua_1/r_1-\ua_2/r_2\right|<\frac{\max\{|d_1\uc_1|,|d_2\uc_2|\}}{|r_1r_2|},$$ then $\ua_1/r_1=\ua_2/r_2$. 

Moreover, if $\ua/r\in L(d_1\uc_1)\cap L(d_2\uc_2)$, where $|d_1\uc_1|^2$ and $|d_2\uc_2|^2\leq |r| $, and $|c_1c_4|,|c_2c_3|<|r/d_1d_2|$, then we must have $d_1\uc_1=d_2\uc_2$.
\end{lemma}
\begin{proof}
We start by proving the second part of the lemma first. We begin by noting that if $\ua/r\in L(d_1\uc_1)\cap L(d_2\uc_2)$, then this implies $C\cmatr{a_1}{a_2}\equiv \cmatr{0}{0}\bmod{r/\ell},$ where $C=\matr{c_1}{c_2}{c_3}{c_4}$, and $\ell=\mathrm{lcm}(d_1,d_2) $. Since both $\uc_1,\uc_2$ are primitive, we can use Lemma \ref{lem:det0} to get that $r/\ell\mid\det(C)$. Since $|c_1c_4|< |r|/|d_1d_2| $ and $|c_2c_3|<|r|/|d_1d_2| $, we have $|\det(C)|<|r/\ell|$.  This must imply that $\det(C)=0$. This would then confirm that $ \uc_1=\uc_2=\uc$, since $\uc_1,\uc_2$ are primitive and therefore monic according to our definition in Sec \ref{sec:notation}. 

We now set $r'=\gcd(\ua\cdot\uc,r), d=r/r'$, where $d$ monic. Clearly, $d\uc\cdot\ua/r=\uc\cdot\ua/r'\in\scrO$. We also have $\gcd(d,\ua\cdot\uc/r')=1$, which implies $\ua/r\in L(d\uc)$. Moreover, since $d_1\ua\cdot\uc=k_1r$, where $d_1\mid r$, $\gcd(d_1,k_1)=1$, we then have $d_1(\ua\cdot\uc/r')=(d_1\ua\cdot\uc/r)d=k_1d$. Since $\gcd(d,\ua\cdot\uc/r')=1$, we must have $d\mid d_1$, but on the other hand, $\gcd(d_1,k_1)=1$ implies that $d_1\mid d$. Since both of them are monic, this must mean $d_1=d$. We can similarly prove $d_2=d$, settling the second part of the lemma.

For the first part, let $\ua_1/r_1\neq \ua_2/r_2$. Without loss of generality, we assume $|d_1\uc_1|\geq |d_2\uc_2|$ and let $|\ua_1/r_1-\ua_2/r_2|<|d_1\uc_1|/(|r_1r_2|)$. The first part of Lemma \ref{lem:Diogen} asserts $\ua_2/r_2\notin L(d_1\uc_1)$.  Using the second part of Lemma \ref{lem:Diogen}, we have $\ua'/r'\in L(d_1\uc_1)$ such that $|\ua_1/r_1-\ua'/r'|=\left|\frac{d_1\uc_1}{r_1r'}\right|$, and moreover, $\ua_1/r_1, \ua'/r'\in L_1(d_1\uc_1,k)$, where $\gcd(d_1,k)=1$. If $\ua_2/r_2\notin L_1(d\uc,k)$, the volume of the parallelepiped with vertices $\ua_1/r_1,\ua_2/r_2,\ua'/r' $ must be non-zero. This volume is also given by $\left| \det\left(\begin{matrix}
 \ua_1/r_1-\ua_2/r_2\\ \ua_1/r_1-\ua'/r'
\end{matrix}\right)\right| $. Clearly, this volume $\geq \frac{1}{|r_1||r_2||r'|}$. On the other hand, it is $<\frac{|d_1\uc_1|^2}{|r_1|^2|r'||r_2|}\leq \frac{1}{|r_1r_2r'|}  $, which is a contradiction.  

We are now reduced to the case $\ua_2/r_2\in L_1(d_1\uc_1,k)$. This implies that $(\ua_2/r_2)\cdot d_1\uc_1=k$. Since $\gcd(d_1,k)=1$, we must have $\ua_2/r_2\in L(d_1\uc_1)$ which is a contradiction, unless, $\ua_1/r_1=\ua_2/r_2$. 
\end{proof}
As an immediate corollary of the second part of Lemma \ref{lem:Dio3} we have:
\begin{corollary}\label{cor:2}
For any $r\in \scrO$, we have
$$\{\ua\in \scrO^2:|\ua|<|r|,\gcd(\ua,r)=1\}=\bigsqcup\limits_{\substack{ d \textrm{ {\em monic}, }\uc \textrm{ {\em primitive}}\\d\mid r\\|dc_1|\leq |r|^{1/2},|dc_2|<|r|^{1/2}}}\{\ua:\gcd(\ua,r)=1,\ua/r\in L(d\uc)\}.$$

\end{corollary}
\begin{proof}
The disjointness of the sets on the right hand side follows immediately from the second part of Lemma \ref{lem:Dio3}. The right hand side is obviously contained in the left hand side. Lemma \ref{lem:Diri} implies that the left hand side is contained in the right hand side.
\end{proof}
We are now ready to establish a refinement of \eqref{eq:Dirichlet}, our main objective in this section,namely, the proof of Theorem \ref{thm:count}:
\begin{proof}[ Proof of Theorem \ref{thm:count}] Throughout this argument, we assume that $ d,d_1,d_2,...\in\scrO$ are monic and $\uc,\uc_1,\uc_2\in\scrO^2$ are primitive. 
Lemma \ref{lem:Diri} implies that every $\ua/r\in L(d\uc) $, for some $d,\uc$ satisfying $|dc_1|\leq |r|^{1/2}, |dc_2|<|r|^{1/2}$. This also implies that $|d\uc|^2\leq|r| $. The proof will follow from an induction on $|r|$. We begin noting that proving Theorem \ref{thm:split} is equivalent to proving
\begin{equation}\label{eq:Dirichlet2'}
\TT^2=\bigsqcup_{0\leq Y\leq Q}\bigsqcup\limits_{\substack{r,d\textrm{ monic },\uc \textrm{ primitive}\\ |r|=\hat{Y}, d\mid r\\\hat{Y-Q/2}\leq |d\uc|\leq \hat{Y/2}\\  |dc_2|<\hat{Y/2}}}\,\,\,
\cupstar_{\substack{|\ua|<|r|\\ \ua/r\in L(d\uc)} }D(\ua,r,Q).
\end{equation}
Here $*$ beside $\sqcup$ denotes that the union is over $\ua\in\scrO^2$ such that $\gcd(\ua,r)=1$.
We begin by proving the disjointness of the intervals on the right hand of \eqref{eq:Dirichlet2'}. Let $\ua_1/r_1\in L(d_1\uc_1), \ua_2/r_2\in L(d_2\uc_2)$, where $d_i,r_i,\uc_i$ satisfy the constraints appearing on the right hand side of \eqref{eq:Dirichlet2'}.  Lemma \ref{lem:Dio3} then implies that either $\ua_1/r_1=\ua_2/r_2$ or $$|\ua_1/r_1-\ua_2/r_2|\geq \frac{\max\{|d_1\uc_1|,|d_2\uc_2|\}}{|r_1||r_2|}\geq \frac{q^{-Q/2}\max\{|r_1|,|r_2|\}}{|r_1||r_2|}\geq q^{-Q/2}\max\{|r_1|^{-1},|r_2|^{-1}\}.$$
On the other hand, if $\ua/r\in L(d_1\uc_1)\cap L(d_2\uc_2)$ then the second part of Lemma \ref{lem:Dio3} forces $d_1\uc_1=d_2\uc_2$, implying disjointness of the right hand side of \eqref{eq:Dirichlet2'}.

Clearly, the right side of \eqref{eq:Dirichlet2'} is contained in the left. To prove the other way around, we proceed with induction. We intend to prove that for any $0\leq M\leq Q $,
\begin{equation}\label{eq:Minduct}
\bigcup_{|r|\leq \hat{M}}\,\,\,\cupstar_{|\ua|<|r|}D(\ua,r,Q)\subseteq \bigsqcup_{0\leq Y\leq M}\bigsqcup\limits_{\substack{r,d\textrm{ monic },\uc \textrm{ primitive}\\ |r|=\hat{Y}, d\mid r\\\hat{Y-Q/2}\leq |d\uc|\leq \hat{Y/2}\\  |dc_2|<\hat{Y/2}}}\,\,\,
\cupstar_{\substack{|\ua|<|r|\\ \ua/r\in L(d\uc)} }D(\ua,r,Q).
\end{equation}

The base case $M=0$ is obvious, since we only have one term on the left hand side, namely, $D(\underline{0},1,Q)$. Clearly, it is contained in $L(\ue_1)$, where $\ue_1=\cmatr{1}{0}$. Let us assume the validity of \eqref{eq:Minduct} for all $M\leq M_0<Q$. Note that $d\mid r$ is forced upon us from \eqref{eq:d|r}. Now, let us choose $\ua/r$, $|\ua|<r,\gcd(\ua,r)=1$, such that $|r|=\hat{M_0+1} $. Lemma \ref{lem:Diri} implies that $\ua/r\in L(d\uc)$, where $|dc_1|\leq |r|^{1/2}, |dc_2|<|r|^{1/2}$. This forces $|d\uc|^2\leq |r|$. If $|r|\leq |d\uc|\hat{Q/2}$, we are done. Otherwise, using Lemma \ref{lem:Dio}, we write $\ua/r=a\uc^\perp/r+\ud/d$, where $|a|<|r/d|, \gcd(a,r/d)=1$. A further application of \cite[Lemma 4.3]{Browning_Vishe15} gives us $a'/r'$ satisfying $|r'|\leq |\uc|\hat{Q/2}$ such that $|a/(r/d)-a'/r'|<(|r'||\uc|\hat{Q/2})^{-1}$. We now set $\frac{\ua_1}{r_1}=\frac{a'\uc^\bot}{r'd}+\frac{\ud}{d}$. If $d\uc\cdot \ua/r=k$, for some $(k,d)=1$, then clearly, $d\uc\cdot \ua_1/r_1=\uc\cdot \ud=k$, as well. Moreover, $d\mid r_1$, and 
\begin{align*}
|\ua/r-\ua_1/r_1|=|d|^{-1}|a/(r/d)-a'/r'||\uc|<(|dr'|\hat{Q/2})^{-1}\leq (|r_1|\hat{Q/2})^{-1}
\end{align*}
We use here that $|r_1|\leq |dr'|$. However, since $|r'|<|r/d|$, we have $|r_1|<|r|$.
Thus, we have found an $\ua_1/r_1\in L(d\uc)$ satisfying $|r_1|< |r|$, such that $\ua/r\in D(\ua_1,r_1,Q)$, which further implies that $D(\ua,r,Q)\subseteq D(\ua_1,r_1,Q)$. We are now through using induction.
\end{proof} 

\begin{remark}
For any $|r|\leq \hat{Q/2}$ and $|\ua|<r$, $\gcd(\ua,r)=1$, by Lemma \ref{lem:Diri}, $\ua/r\in L(d\uc)$, where $|dc_1|\leq |r|^{1/2}, |dc_2|<|r|^{1/2}$. Moreover, since $ |r|\leq \hat{Q/2}$, $D(\ua,r,Q)$ appears exactly once on the right hand side of \eqref{eq:Dirichlet2}. Since $\ua/r$ is was chosen to be arbitrary, this shows that 

\begin{equation}\label{eq:Dirichlet3}
\TT^2=\bigsqcup_{\substack{|r|\leq \hat{Q/2}\\ r\textrm{ monic }}}\,\,\,
\cupstar_{\substack{|\ua|<|r|} }D(\ua,r,Q)\bigsqcup\limits_{\substack{r,d\textrm{ monic, }\uc \textrm{ primitive}\\ \hat{Q/2}<|r|\leq \hat{Q}\\ |r|\leq |d\uc|\hat{Q/2}, d\mid r\\ |dc_1|\leq |r|^{1/2}, |dc_2|<|r|^{1/2}}}\,\,\,
\cupstar_{\substack{|\ua|<|r|\\ \ua/r\in L(d\uc)} }D(\ua,r,Q).
\end{equation}
This is the same idea that handed us Corollary \ref{cor:2}. This is expected, since if $r_1$ and $r_2$ are small, then we do not expect any overlaps in the intervals $D(\ua_1,r_1,Q)$ and $D(\ua_2,r_2,Q)$. \eqref{eq:Dirichlet3} could be used to estimate contribution from low values of $r$ more effectively. More explicitly, we may be able to save a factor of size $O(|r|^{3/2})$ from all square-free values of $|r|\leq \hat{Q/2}$. This saving is not required in this work, but it may be useful in further applications.
\end{remark}

\section{Background on a pair of quadrics}
\label{sec:background}
In this section, we will collect some relevant facts regarding smooth complete intersections of two absolutely irreducible quadratic forms. Let $F_1,F_2\in \scrO[x_1,...,x_n]$ be absolutely irreducible quadratic forms defining a smooth complete intersection $X$.  Throughout, we  will assume that Char$(K)>2$. Let $M_1,M_2$ be symmetric matrices with $\scrO$ entries defining $F_1$ and $F_2$ respectively, i.e., $F_i(\x)=\x^tM_i\x$, for $i=1,2$. Since we are interested in obtaining an asymptotic formula for the counting function $N(P)$ defined in \eqref{eq:Count}, throughout the paper, we will also fix $N\in \scrO$ and $\b\in \scrO^n$ such that $F_1(\b)\equiv F_2(\b)\equiv 0\bmod{N}$. The geometry of $X$ is well-understood, see \cite{Reid72} and \cite{HeathBrown_Pierce17} for example. Most of the geometric properties derived there are valid for any smooth complete intersection of two quadrics over any field of odd characteristic, most of which we will just state here without any further explanation.

We begin with defining some notation.  For any pair $\ux=(x,y)\in K_\nu^2$, let 
\begin{equation}
\label{eq:FuxMuxdef}
F_\ux=-yF_1+xF_2\,\,\,\,\,\textrm{ and }\,\,\,\,\, M_{\ux}=-yM_1+xM_2
\end{equation} denote the matrix defining the quadratic form $F_\ux$.  As per \cite[Proposition 2.1]{Reid72}, we can assume that $M_1$ is of full rank. \cite[Proposition 2.1]{Reid72} also implies that the matrices $M_1$ and $M_2$ are simultaneously diagonalisable over an algebraic closure $\overline{K}$.  \cite[Condition 4]{HeathBrown_Pierce17} implies that for any primitive $\uc\in\scrO^2$, $\rank(M_\uc)\geq k-1$. Moreover, when $c_1\neq 0$, $\rank(M_\uc)=k-1$ precisely when $c_2/c_1$ is an eigenvalue of $M_1^{-1}M_2$. However, since $M_1^{-1}M_2 $ has at most $n$ distinct eigenvalues and each primitive vector $\uc$, produces a unique ratio $c_2/c_1$, there are at most $n$ distinct primitive vectors $\uc$'s for which $\rank(-c_2F_1+c_1F_2)=k-1$. We call such $\uc$'s as ``bad''.

\subsection{The determinant form $F(x,y)$}Given any $x,y\in \overline{K}_\infty$, let
\begin{equation}
\label{eq:Fxydef}
F(x,y)=\det(-yM_1+xM_2)
\end{equation}
be a homogeneous binary form of degree $n$.
   \cite[Condition 2]{HeathBrown_Pierce17} implies that $F(x,y)$ has distinct linear factors over $\bar{K}$. Let $K_1$ denote the splitting field of the polynomial $F$ over $K$. Thus, we can factor 
\begin{equation}\label{eq:Fsplit}
F(x,y)=h^{-1}\prod_{i=1}^n(\lambda_i x-\mu_iy),
\end{equation}
where $h\in \scrO$, $\lambda_i,\mu_i\in \scrO_{K_1}$. Let $\rho_i=\lambda_i/\mu_i$ denote the eigenvalues of $M=M_1^{-1}M_2$. $\rho_i$'s must be pairwise distinct and therefore, at most one of them could be $0$. Throughout, we will assume that $\rho_i\neq 0$ for any $1\leq i\leq n-1$. Without loss of generality, let $0\leq n_1\leq n$ is such that $\rho_i\notin K_\infty$ if $ i\leq n_1$ and $\rho_i\in K_\infty$ if $i>n_1$. The norm on $K_\infty$ could be suitably extended to $K_1$. Note that when $K=\QQ$, since $M$ is symmetric, $n_1=0$, and therefore $M$ can be diagonalised over $\RR$. In the function field setting however, this might not hold. However, we may still be able to obtain the following result, which will be necessary in obtaining a satisfactory bound for our singular integral (see Lemma \ref{lem:I-hard}):
\begin{lemma}
\label{lem:eigenvalue}
We can find a matrix $U\in \GL_n(K_\infty)$ satisfying $$U^{-1}MU=\matr{M'_{n_1\times n_1}}{M''_{n_1\times (n-n_1)}}{\vecnull_{(n-n_1)\times n_1}}{D(\rho_{n_1+1},...,\rho_{n})},$$ where $D$ is a $(n-n_1)\times (n-n_1)$ diagonal matrix with the prescribed diagonal entries. Moreover, the eigenvalues of $M'$ are precisely given by $\rho_1,...,\rho_{n_1}$ and therefore, they do not belong to $K_\infty$. 

Moreover, we can also find a constant $0<C_1\leq 1$ such that $|\rho_i|\leq C_1^{-1}$ for any $i$, $C_1\leq |\rho_i|$ for any $i\neq n$, $C_1\leq|\rho_i-\rho_j|$ for any $i\neq j$, and for any $z\in K_\infty$ and for any $1\leq i\leq n_1$, we have $C_1\leq |z-\rho_i|$ and $C_1\leq |z-\rho_i^{-1}|$. If $\rho_n\neq 0$, then we can also make sure that $C_1\leq |\rho_n|\leq C_1^{-1}$.
\end{lemma}
\begin{proof}
Let $i$ be any integer satisfying $n_1+1\leq i\leq n$. We have $\det(\rho_iI_n-M)=0$. Let $\rho_iI_n-M=TDS$ be a Smith normal form for the matrix $\rho_iI_n-M$  over $K_\infty$. Therefore $T,S\in\GL_n(K_\infty)$ and $D$ is a diagonal matrix with entries in $K_\infty$. We may also assume that only the last diagonal entry of $D$ is $0$. Let $\e_n$ be the vector which contains $1$ at the $n$-th place and $0$'s everywhere else. The vector $\v_i=S^{-1}\e_n\neq \vecnull$ satisfies $M\v_i=\rho_i\v_i$. Moreover, since $M$ is symmetric, we must have $\v_i\cdot \v_j=0$, for $i\neq j$. We thus have an orthogonal system of eigenvectors $\v_{n_1+1},...,\v_{n}\in K_\infty^n$. Upon extending the basis and changing the standard basis to this new one, we are now guaranteed a matrix $U_1\in \GL_n(K_\infty)$ such that $U_1MU_1^{-1}$ is in the form of the transpose of the required form. We may now use the symmetry of $M$ and choose $U=U_1^t$ to get the required expression.

To prove the second part, we begin by observing that $\rho_1,...,\rho_{n_1}$ have to be the eigenvalues of $M'$. For any $1\leq i\leq n_1$, we must have $\sup_{x\in K_\infty}|\rho_i-x|>0$, since otherwise, using the completeness of $K_\infty$, $\rho_i\in K_\infty$. The existence of a suitable constant $C_1$ now follows from this fact and due to the fact that $\rho_i$'s are all distinct.
\end{proof}

\subsection{Good and bad primes}\label{sec:goodbad}Let $\uc$ be a primitive pair and let $M_\uc=TDS$ denote a smith normal form over $\scrO$. Here, $T$ and $S$ are in $\GL_n(\scrO)$ satisfying $\det(T),\det(S)\in\FF_q^\times$ and $D=\diag(\mu_1,...,\mu_n)$ is diagonal. Moreover, $\mu_1\mid \mu_2\mid\mu_3...\mid\mu_n$. Therefore, $\mu_i\neq 0$ if $i\neq n$, and $\mu_n=0\iff \uc$ is a bad pair. Let $\e_j$ denote the $j$-th vector in the standard basis for $\cO^n$. Let $\y_j=S^{-1}\e_j$ be another basis of $\cO^n$. The quadratic form
\begin{equation}\label{eq:Qucdef}
Q_\uc(x_1,...,x_{n-1}):=F_\uc(x_1\y_1+...+x_{n-1}\y_{n-1}),
\end{equation} 
 in $n-1$ variables will feature prominently in our bounds for exponential sums. When $\uc$ is bad, $\mu_n=0$. Therefore, $M_\uc\y_n=\vecnull$. Moreover, since $M_\uc$ is symmetric, $\y_n^tM_\uc=\vecnull^t$. Therefore, 
 $$F_\uc(x_1\y_1+...x_n\y_n)=F_\uc(x_1\y_1+...+x_{n-1}\y_{n-1})=Q_\uc(x_1,...,x_{n-1}).$$
$Q_\uc$ has to be non-singular, since the set $\{M_\uc\y_1,...,M_\uc\y_{n-1}\}$ is linearly independent and since the rank of $M_\uc$ is $\geq n-1$.

Let 
\begin{equation}
\label{eq:deltadef}
D_\F=Nh\Delta_\F\prod_{\uc \textrm{ primitive and bad}}\Delta(Q_\uc)\prod_{\sigma\in \mathrm{Gal}(K_1/K)}\prod_{1\leq i<j\leq n}\sigma(\lambda_j\mu_i-\mu_i\lambda_j),
\end{equation}
where $h,\lambda_i,\mu_i$ as in \eqref{eq:Fsplit}, $\Delta_\F$ denotes the discriminant of the binary form $F(x,y)$ and $\Delta(Q_\uc)$ denote the discriminant of the quadratic form $Q_\uc$. Here, $\mathrm{Gal}(K_1/K)$ denotes the Galois group of the splitting field $K_1$ of the polynomial $F(x,y)$ over $K$.
We say that a prime $\vp$ is {\em bad} if $\vp\mid D_\F$, and the rest of the primes will be called {\em good} primes. For any {\em good, primitive} $\uc$, if a good prime $\vp$ is such that $\vp\nmid \det(M_\uc)$, then we say that $\vp$ is of type I for $\uc$, otherwise, we say that $\vp$ is of type II. Note that for a bad pair $\uc$, every good prime $\vp$ will be of type I, since $\vp\nmid \Delta(Q_{\uc})$ for any good prime $\vp$. Note that our definition of bad primes differs slightly from that in \cite{HeathBrown_Pierce17}. For convenience, we have added the primes dividing $N$ as well as the ``type II'' primes for bad pairs $\uc$ to this list.

\subsection{The dual variety} In our analysis, an important role will be played by the following family of dual forms defined by $$F^*(x,y,\v)=\v^t\det(-yM_1+xM_2)(-yM_1+xM_2)^{-1}\v.$$ For a fixed value of $\v$, we may consider $F^*(x,y,\v)$ as a binary, homogeneous polynomial of degree $n-1$. The discriminant of this polynomial, denoted by $\scrF^*(\v)$, is a polynomial of degree $4(n-2)$. This polynomial has an albeit more familiar interpretation:
\begin{lemma} $\scrF^*(\v)$ is the polynomial defining the dual variety $X^*$ of the complete intersection $X$. 
\end{lemma}
\begin{proof}
$F^*(x,y,\v)=0$ if and only if the quadratic variety $\{-yF_1(\x)+xF_2(\x)=\v\cdot \x=0\}$ is singular, since $F^*(x,y,\v)$ is a non-zero multiple of  the determinant of the matrix defining the corresponding quadratic form. On the other hand, if $\scrF^*(\v)=0$, then the polynomial  $F^*(x,y,\v)=0$ must have a double root $(x_0,y_0)\in \overline{K}^2$. Without loss of generality, let $x_0\neq 0$. Let $X_1=\{-y_0F_1(\x)+x_0F_2(\x)=\v\cdot \x=0\}$ and let $M'$ be a $(n-1)\times (n-1)$ matrix defining $X_1$. 

If the singular locus of this variety is of projective dimension $\geq 1$, then it must intersect $F_1(\x)=0$, thus producing a singular point in the complete intersection of $\v\cdot\x=0$ and $X$. On the other hand, if $X_1$ only has one singular point, it means that the matrix $M'$ defining $X_1$ (up to scalar multiplication) has only one zero eigenvector, say $\x_0$. Since, $\scrF^*(\v)=0$, $(\partial/\partial t)|_{t=0}(\det(M'+tM''))=0$, where $M''$ is the matrix defining $F_1(\x)=\v\cdot\x=0$. $M'$ can be diagonalised over $\bar{K}$. An easy calculation shows that $(\partial/\partial t)|_{t=0}(\det(M'+tM''))$ is proportional to $\x_0^tM''\x_0$. Thus we must have $\x_0^tM''\x_0=0$, which means $F_1(\x_0)=0$. This implies that $\x_0$ belongs to $X\cap \{\v\cdot\x=0\}$. This further implies that $X\cap\{\v\cdot\x=0\} $ is singular. Thus, $\scrF^*(\v)=0$ implies that $\v$ belongs to $X^*$. Moreover, according to \cite[Theorem 3]{Aznar}, the polynomial defining the dual $X^*$ is an irreducible polynomial of degree $4(n-2)$. Therefore, $\scrF^*(\v)$ must be a polynomial defining the dual variety $X^*$.

\end{proof}

\section{Activation of the circle method}
\label{sec:circle}
Let $w$ denote the characteristic function of $\TT^n\subset K_\infty^n$, and let $\x_0\in K_\infty^n$ be a fixed point satisfying $F_1(\x_0)=F_2(\x_0)=0$. Since both forms are homogeneous, we may also assume $|\x_0|<1/H_\F$. Let $\omega(\x)=w(t^{L}(\x-\x_0))$, where $L\geq 0$ be a suitable integer to be chosen later. The extra conditions $|\x_0|<1/H_\F$ and $L\geq 0$ are only used to make the constants a bit more explicit. Recall that for any $P\in \scrO$, we consider the counting function
\begin{equation*}
N(P)=\sum_{\substack{\x\in\scrO^n\\ F_1(\x)=F_2(\x)=0\\ \x\equiv \b\bmod{N}}}\omega(\x/P).
\end{equation*}
We intend to establish an asymptotic formula as $|P|\rightarrow \infty$. We may write
\begin{align}\label{eq:intT2}
N(P)=\int_{\TT^2}S(\alpha_1,\alpha_2)d\ualf,
\end{align}
where
\begin{align*}
S(\ualf)=\sum_{\substack{\x\in\scrO^n\\ \x\equiv \b\bmod{N}}}\omega(\x/P)\psi(\alpha_1 F_1(\x)+\alpha_2 F_2(\x)).
\end{align*}
We will apply Theorem \ref{thm:split} (version \eqref{eq:Dirichlet2'}) with $Q$ satisfying
\begin{equation}\label{eq:Qchoice}
|P|^{4/3}\leq \hat{Q}\leq |P|^{4/3}q
\end{equation} to replace the integral over $\TT^2$ in \eqref{eq:intT2} to get
\begin{align}\label{eq:D-Refine}
\int_{\TT^2}S(\ualf)d\ualf=\sum_{Y=0}^Q\,\,\,\,\,\,\sum_{\substack{r,d\textrm{ monic, }\uc \textrm{ primitive}\\  \hat{Y-Q/2}\leq |d\uc|\leq\hat{Y/2}\\ |dc_2|<\hat{Y/2}\\|r|=\hat{Y}, d\mid r}}\,\,\,
 \int_{|\uz|<\hat{Y}^{-1}q^{-Q/2}}S(d\uc,r,\uz)d\uz,
\end{align} 
where
\begin{equation}
S(d\uc,r,\uz)=\sumstar\limits_{\substack{\ua\in \cO^2\\ \ua/r\in L(d\uc)}} S(\ua/r+\uz).
\end{equation}
 This choice of $Q$ is standard for a system of two quadrics. It is chosen in such a way that when $Y=Q$ and $r$ is such that $|r|=\hat{Y}$, then for any $\gcd(\ua,r)=1$, the measure of the set $D(\ua,r,Q)$ in \eqref{eq:Dirichlet} is $\leq \hat{Q}^{-3}\ll |P|^{-4}$, aiding us to prove the right asymptotic in Theorem \ref{thm:count}.

For each $L(d\uc)$, we are going to consider the contribution from $\ua/r\in L(d\uc)$. 
Let $r_N=rN/\gcd(r,N)$, the least common multiple of $ r$ and $N$. 
 We next use a standard Poisson summation argument as in \cite[Section 4]{Browning_Vishe15} applied to \eqref{eq:D-Refine} to establish the following result:
\begin{lemma}\label{lem:NP1} Given any $\ve>0$, we have
\begin{equation*}
N(P)=|P|^n\sum_{0\leq Y\leq Q}\,\,\,\,\,\sum_{\substack{r,d\textrm{ {\em monic,} }\uc \textrm{ {\em primitive}}\\ \hat{Y-Q/2}\leq|d\uc|\leq \hat{Y/2}\\ |dc_2|<\hat{Y/2}\\ |r|=\hat{Y},d\mid r}}|r_N|^{-n}
 \int\limits_{\substack{|\uz|<\hat{Y}^{-1}q^{-Q/2}}}\sum_{\v\in\scrO^n}S_{d\uc,r,\b,N}(\v)I_{r_N}(\uz;\v),d\z,
\end{equation*}
where
\begin{equation}\label{eq:Sexpsumdef}S_{d\uc,r,\b,N}(\v)=\sum_{\ua/r\in L(d\uc)}\sum_{\substack{\x\in\scrO^n\\|\x|<|r_N|\\\x\equiv\b\bmod{N}}}\psi\left(\frac{a_1F_1(\x)+a_2F_2(\x)}{r}\right)\psi\left(\frac{-\v\cdot\x}{r_N}\right),\end{equation}
$$I_{s}(\uz;\v)=\int_{K_\infty^n}\omega(\x)\psi\left((z_1P^2F_1(\x)+z_2P^2F(\x)) +P\v\cdot\x/s\right)d\x, $$
and $r_N=rN/\gcd(r,N)$. 
\end{lemma}
We begin by establishing the following multiplicativity relation for the exponential sums:
\begin{lemma}
\label{lem:Multipli}
Let $d\mid r$ and let $r=r_1r_2$, where $\gcd(r_1,r_2)=1$, then there exist $\b_1,\b_2,\b_3\in(\scrO/N\scrO)^n$ such that
\begin{equation}\label{eq:multi1}
S_{d\uc,r,\b,N}(\v)=S_{d_1\uc,r_1,\b_1,N_1}(\v)S_{d_2\uc,r_2,\b_2,N_2}(\v)\psi\left(\frac{-\v\cdot\b_3}{N_3}\right),
\end{equation}
where $d=d_1d_2$ such that $d_i\mid r_i$ for $i=1,2$, and $N=N_1N_2N_3$, where $N_1\mid r_1^\infty$, $N_2\mid r_2^\infty$, $\gcd(N_3,r)=1$.

\end{lemma}
\begin{proof}
Recall that $\ua/r\in L(d\uc)\iff d\uc\cdot\ua=rk, \mathrm{ where }\,\,\, \gcd(\ua,r)=\gcd(k,d)=1$. We start by rewriting $\ua=r_2\ua_1+r_1\ua_2$, where $|\ua_i|<|r_i|, \gcd(\ua_i,r_i)=1$. Firstly, since $\uc\cdot\ua\equiv 0\bmod{r/d} $, this forces $\uc\cdot\ua_i\equiv{0}\bmod{r_i/d_i}$ for $i=1,2$. i.e. $\uc\cdot\ua_i/(r_i/d_i)\in\scrO$. Next, since $d\uc\cdot\ua/r= d\uc\cdot\ua_1/r_1+d\uc\cdot\ua_2/r_2=d_2\uc\cdot\ua_1/(r_1/d_1)+d_1\uc\cdot\ua_2/(r_2/d_2)$, $\gcd(d\uc\cdot\ua/r, d)=1$ if and only if $\gcd(d_i\uc\cdot\ua_i/r_i,d_i)=1$, for $i=1,2$, which implies that 
 $$\ua/r\in L(d\uc)\iff\ua_i/r_i\in L(d_i\uc),\,\,\,\textrm{for}\,\,\, i=1,2.$$
\eqref{eq:multi1} now follows from exactly following the argument in \cite[Lemma 4.5]{Browning_Vishe15}. 
\end{proof}
This multiplicativity relation will be used to obtain finer bounds for the exponential sums, which will be the focus of Section \ref{sec:expsum}. We now consider bounds for the exponential integral.
\subsection{Bounds for the exponential integral}\label{s:honk}

We proceed to study $I_{r_N}(\uz;\v)$ for a given  $r\in \cO$. We have 
\begin{align}\notag
I_{r_N}(\uz;\v)
&=
\int_{K_\infty^n} 
w\left(t^L(\x-\x_0)\right)
\psi \left( z_1P^2F_1(\x)+z_2P^2F_{2}(\x)+P\v.\x/r_N\right)d \x\\
&=\frac{1}{\hat{L}^n}\psi\left(\frac{P\v\cdot\x_0}{r_N}\right)
J_{\G}\left((z_1 P^2,z_2P^2);
\frac{Pt^{-L} \v}{r_N}\right),\label{eq:small}
\end{align}
in the notation of \eqref{eq:J},
where $\G(\y)=
(G_1(\y),G_2(\y))$, $G_i(\y)=F_i(\x_0+t^{-L}\y)$ for $i=1,2$.

According to Lemma 
\ref{lem:J-easy}, $
J_\G((P^3z_1,P^2z_2);P\v/r_N)=0$ if 
$$
\frac{|P||\v|}{|r_N|}> \max\{1,| P|^2 |z_1 | H_{F_1},  |P|^2 |z_2|H_{F_2}\}.
$$
Hence we may truncate the sum over $\v$ in Lemma \ref{lem:NP1} to arrive at the following result.

\begin{lemma}\label{lem:N2}
$$
N(P)=|P|^n\sum_{0\leq Y\leq Q}\,\,\,\,\,\sum_{\substack{r,d\textrm{ {\em monic,} }\uc \textrm{ {\em primitive}}\\ \hat{Y-Q/2}\leq|d\uc|\leq \hat{Y/2}\\ |dc_2|<\hat{Y/2}\\ |r|=\hat{Y},d\mid r}}|r_N|^{-n}
 \int\limits_{\substack{|\uz|<\hat{Y}^{-1}q^{-Q/2}}}\sum_{\substack{\v\in\scrO^n,|\v|\leq \hat{V}}}S_{d\c,r,\b,N}(\v)I_{r_N}(\uz;\v)d\uz,$$
where 
\begin{equation}
\label{eq:Vdef}
\hat V= H_{\F} |r_N||P|^{-1}\max\{1, |z_1|| P|^2, |z_2||P|^2\}.
\end{equation}
\end{lemma}

We will need a good upper bound for 
$I_{r}(\uz;\v)$, for $r, \uz, \v$ appearing in the expression for $N(P)$ in this lemma.  This need is met by the following lemma. A key result in proving it is a decomposition of the matrix $M=M_1^{-1}M_2$ obtained in Lemma \ref{lem:eigenvalue}. 

In the following lemma, we borrow the notation from Lemma \ref{lem:eigenvalue}, i.e, the eigenvalues $\rho_j$, the matrix $U$ and the constant $C_1$ are as in the statement of Lemma \ref{lem:eigenvalue}.
\begin{lemma}\label{lem:I-hard} 
Let $Z\in \ZZ$ and let $\uz$ be such that $|\uz|=\hat{Z}$.  Let $|\v|\leq \hat{V}$, where $\hat{V}$ as in \eqref{eq:Vdef}. Then
$$|I_{r_N}(\uz;\v)|\leq  \hat{L}^{-n}\meas(\Omega_\uz),$$
where
$$\Omega_\uz=\left\{\x\in \TT^n: |\x-\x_0|<\hat{-L}, ~ |P^2z_1\nabla F_1(\x)+P^2z_2\nabla F_2(\x)+ P\v/r_N|\leq H_\F
J(Z)^{1/2}\right\},$$
where 
\begin{equation}\label{eq:Juzdef}
J(Z)=1+|P|^2\hat{Z}.
\end{equation}

Moreover,
\begin{equation*}
\int_{|\uz|=\hat{Z}}\meas(\Omega_\uz)d\uz \ll C_\F J(Z)^{-n/2+1}\log(|P|^2\hat{Z})\hat{Z+1}\sum_{j=n_1+1}^{n}\left(\frac{1+|\rho_j|}{\min\{1,|\rho_j|\}}\right)\min\left\{\hat{Z+1}, |P|^{-2}\right\}.
\end{equation*}
where $C_\F=(H_{U^{-1}}H_{M_1^{-1}}H_\F)^{n}H_{M'}^{n_1(n_1-1)}C_1^{-2n^2-2(n-n_1)} $.
\end{lemma}

\begin{proof}
Let $G_1$ and $G_2$ be as in \eqref{eq:small}. Let 
$\gamma_i= z_iP^{2}$ and $\w=Pt^{-L}\v/r_N$, for convenience. Let $\hat{Z_i}=|z_i|$, and therefore $Z=\max\{Z_1,Z_2\}$. Since $F_1(\x_0)=F_2(\x_0)=0$, $|\x_0|< 1$ and $L\geq 0$, 
\begin{equation}\label{eq:HGbound}
H_\G< \hat{L}^{-1}H_\F.
\end{equation}
In particular, when $|\v|\leq \hat{V}$, we have 
$$
 |\w|\leq H_\F\max\{1,|\gamma_1|,|\gamma_2|\}=H_\F J(Z).
$$
Lemma \ref{lem:small} in conjunction with \eqref{eq:small} implies that 
\begin{align*}
|I_{r_N}(\uz;\v)| &\leq \frac{1}{\hat{L}^n}
\left|J_\G((\gamma_1,\gamma_2);\w)\right| \\
&
\leq \frac{1}{\hat{L}^n}\meas\left\{\y\in \TT^n: |\gamma_1\nabla G_1(\y)+\gamma_2\nabla G_2(\y)+ \w|\leq
H_\G \max\{1,
|\gamma_1|,|\gamma_2|\}^{1/2}\right\}\\
&\leq \meas\left\{\x\in \TT^n: |\x-\x_0|<\hat{-L}, ~ |\gamma_1\nabla F_1(\x)+\gamma_2\nabla F_2(\x)+ t^L\w|\leq H_\F
\max\{1,
|\gamma_1|^{1/2},|\gamma_2|^{1/2}\}\right\}\\
&=\meas(\Omega_\uz).
\end{align*}
This settles the first part of the lemma. We can further bound 
\begin{align*}
\meas(\Omega_\uz)&\leq \meas\left\{\x\in \TT^n: |\x|<1, ~ |M_1(\gamma_1I_n+\gamma_2M)\x+ t^L\w|\leq H_\F
\max\{1,
|\gamma_1|^{1/2},|\gamma_2|^{1/2}\}\right\}\\
&\leq \meas(\mathcal{R})
\end{align*}
where
$$
\mathcal{R}=
\left\{\x\in \TT^n: |\x|<1, ~ |(\gamma_1I_n+\gamma_2M)\x+ t^{L}\w|\leq H_{M_1^{-1}}H_\F
\max\{1,
|\gamma_1|^{1/2},|\gamma_2|^{1/2}\}\right\}.
$$
If $\x$ and $\x+\x'\in\mathcal{R}$, then $|(\gamma_1I_n+\gamma_2M)\x'|\leq H_{M_1^{-1}}H_\F
\max\{1,
|\gamma_1|^{1/2},|\gamma_2|^{1/2}\}$.
If $|\gamma_1|, |\gamma_2|\leq  1$, then the trivial bound $1$ will suffice here. Hence from now on, we assume the contrary, i.e.  $1<\max\{|\gamma_1|,|\gamma_2|$\}.

At this point, we change the variables to place $\y=U^{-1}\x$, where $U$ is as in Lemma \ref{lem:eigenvalue}. Thus, it is enough to estimate the measure of the set 
\begin{equation}\label{eq:One}\left\{|\y|<H_{U^{-1}}:\left|\left(\gamma_1 I+\gamma_2\matr{M'_{n_1\times n_1}}{M''_{n_1\times (n-n_1)}}{\vecnull_{(n-n_1)\times n_1}}{D(\rho_{n_1+1},...,\rho_{n})}\right)\y\right|\leq H_0\right\},\end{equation}
where $H_0=H_{U^{-1}}H_{M_1^{-1}}H_\F
\max\{1,
|\gamma_1|^{1/2},|\gamma_2|^{1/2}\}$.

First, we turn our attention to $y_{n_1+1},...,y_{n}$. If $|\gamma_1+\gamma_2\rho_{i_0}|< C_1^2\max\{|\gamma_1|,|\gamma_2|\}$ for some $n_1+1\leq i_0\leq n-1$, then since $C_1\leq |\rho_{i_0}|\leq C_1^{-1} $, this forces that  $|\gamma_1|=|\rho_{i_0}\gamma_2|$ which gives $|\gamma_2|\geq C_1|\gamma_1| $. Moreover, for any $i\neq i_0$ we have,
\begin{align*}
|\gamma_1+\gamma_2\rho_i|=|\gamma_1+\gamma_2\rho_{i_0}+\gamma_2(\rho_i-\rho_{i_0})|\geq |\gamma_2(\rho_i-\rho_{i_0})|\geq C_1|\gamma_2|\geq C_1^2\max\{1,|\gamma_1|,|\gamma_2|\}.
\end{align*}
If $i_0=n$ and $\rho_n\neq 0$, then the argument outlined above goes through verbatim. On the other hand, if $\rho_n=0$, then this forces $|\gamma_1|< C_1^2\max\{|\gamma_1|,|\gamma_2|\}$ which implies that $|\gamma_2|>C_1^{-2}|\gamma_1|$ and hence $|\gamma_2|>1$, and thus for any $i\neq n$ we get
\begin{align*}
|\gamma_1+\gamma_2\rho_i|=|\gamma_2\rho_i|\geq C_1|\gamma_2|.
\end{align*}
Combining these bounds, the measure of $y_{n_1+1},...,y_n$ appearing in \eqref{eq:One} is bounded by \begin{equation}\label{eq:1}
(H_{U^{-1}}H_{M_1^{-1}}H_\F)^{(n-n_1)}C_1^{-2(n-n_1)}\max\{1,|\gamma_1|,|\gamma_2|\}^{-(n-n_1-2)/2}(1+\min_{n_1+1\leq j\leq n}|\gamma_1+\rho_j\gamma_2|)^{-1}.
\end{equation}
To bound the size of the first $n_1$ co-ordinates $\y_1=(y_1,..,y_{n_1})$ appearing in \eqref{eq:One}, note that for a fixed choice of $y_{n_1+1},...,y_n$, the two different values of $\y_1$ must differ by an element in the set
$$\{|\y_1|<H_{U^{-1}}:|(\gamma_1I+\gamma_2M')\y_1|\leq H_0\}.$$
Therefore, it is enough to bound the measure of this set. Suppose, $|\gamma_1|\geq |\gamma_2|$, then the eigenvalues of $\gamma_1I+\gamma_2M'$ are $$|\gamma_1+\gamma_2\rho_i|=|\gamma_1||\rho_i||\gamma_2/\gamma_1+\rho_i^{-1}|\geq|\gamma_1|C_1^{2}.$$
We can prove a similar statement when $|\gamma_2|>|\gamma_1|$. This gives us that $|\det(\gamma_1I+\gamma_2M')|\geq C_1^{2n_1}\max\{1,|\gamma_1|,|\gamma_2|\}^{n_1}$. Thus, $(\gamma_1I_n+\gamma_2M')^{-1}$ has entries bounded by $H_{M'}^{n-1}C_1^{-2n_1}\max\{1,|\gamma_1|,|\gamma_2|\}^{-1}$. Thus, the condition on $\y_1$ transforms to bounding
\begin{equation}\label{eq:2}\begin{split}\meas\{|\y_1|\leq &H_{U^{-1}}H_{M_1^{-1}}H_{M'}^{n_1-1}C_1^{-2n_1}\max\{1,|\gamma_1|,|\gamma_2|\}^{-1/2}\}\\ &\leq (H_{U^{-1}}H_{M_1^{-1}}H_\F)^{n_1}H_{M'}^{n_1(n_1-1)}C_1^{-2n_1^2}\max\{1,|\gamma_1|,|\gamma_2|\}^{-n_1/2},
\end{split}
\end{equation}
\eqref{eq:1} and \eqref{eq:2} give us that \begin{equation}\label{eq:omuzB}
\meas(\Omega_\uz)\leq C_\F 
J(Z)^{-(n-2)/2}(1+|P|^2\min_{n_1+1\leq j\leq n}|z_1+\rho_jz_2|)^{-1}.
\end{equation}
This readily gives us the bound
\begin{equation}\label{eq:omuzB1}
\int_{|\uz|\leq \hat{Z}}\meas(\Omega_\uz)d\uz\leq C_\F 
J(Z)^{-(n-2)/2}\hat{Z+1}^2.
\end{equation}
To obtain the other bound, note that to bound $\int_{|\uz|=\hat{Z}}\meas(\Omega_\uz)d\uz$, it is clearly enough to bound the integral $\int_{|\uz|\leq \hat{Z}}\meas(\Omega_\uz)d\uz$. For every $n_1+1\leq j\leq n$, let $$I_j=\{|\uz|\leq \hat{Z}: |z_1+\rho_j z_2|<|P|^{-2}\}.$$
Measure of $I_j$ is clearly $\leq |P|^{-2}\hat{Z}$. We may now bound the required integral by:
\begin{align*}
\int_{|\uz|= \hat{Z}}(1+|P|^2\min_{n_1+1\leq j\leq n}|z_1+\rho_jz_2|)^{-1} d\uz&\leq (n-n_1)|P|^{-2}\hat{Z}\\&+|P|^{-4}\sum_{j=n_1+1}^n\int_{\{|\uz|\leq |P|^2\hat{Z}\}\setminus I_j'}(1+ |z_1+\rho_jz_2|)^{-1}d\uz,
\end{align*}
where
$$I_j'=\{|\uz|\leq |P|^2\hat{Z}: |z_1+\rho_j z_2|<1\}.$$
If $\rho_n=0$, $\{|\uz|\leq |P|^2\hat{Z}\setminus I_n'\}=\{|\uz|\leq \hat{Z}, |z_1|\geq 1\}$.
Thus,
\begin{align*}
\int_{\{|\uz|\leq |P|^2\hat{Z}\}\setminus I_n'}|z_1|^{-1}dz_1dz_2=q|P|^2\hat{Z}\int_{1\leq |z_1|\leq |P|^2\hat{Z}}|z_1|^{-1}dz_1=|P|^{2}\hat{Z+1}\log(|P|^{2}\hat{Z})
\end{align*}
which is clearly admissible. When $\rho_j\neq 0$, we may change the variables to put $s_1=z_1,s_2=z_1+\rho_jz_2$ to get
\begin{align*}
\int_{\{|\uz|\leq |P|^2\hat{Z}\}\setminus I_j'}|z_1+\rho_jz_2|^{-1}d\uz\leq |\rho_j|^{-1}\int_{\substack{|s_1|,|s_2|\leq |P|^2\hat{Z}|(1+|\rho_j|)\\ 1\leq |s_2|}}|s_2|^{-1}ds_1ds_2\leq q \frac{1+|\rho_j|}{|\rho_j|}\log(|P|^{2}\hat{Z})\hat{Z}.
\end{align*}
Combining the above bound with \eqref{eq:omuzB}, proves the final part of the lemma.
\end{proof}
\subsection{Preparation of the error term}\label{sec:Prep}
We now come back to our main counting function $N(P)$. Lemma \ref{lem:N2} implies
$$
N(P)=|P|^n\sum_{0\leq Y\leq Q}\,\,\,\,\,\sum_{\substack{r,d\textrm{ { monic,} }\uc \textrm{ { primitive}}\\ \hat{Y-Q/2}\leq|d\uc|\leq \hat{Y/2}\\ |dc_2|<\hat{Y/2}\\ |r|=\hat{Y},d\mid r}}|r_N|^{-n}
 \int\limits_{\substack{|\uz|<\hat{Y}^{-1}q^{-Q/2}}}\sum_{|\v|\leq \hat{V}}S_{d\uc,r,\b,N}(\v)I_{r_N}(\uz;\v)d\uz,$$
where $S_{d\uc,r,\b,N}(\v)$, $I_{r}(\uz,\v)$ and $\hat{V}$ are as in the statements of Lemmas \ref{lem:NP1} and \ref{lem:N2} respectively. The main contribution would arise from the $\v=\vecnull$ terms when $|r|\leq \hat{Q}^\Delta$, where $0<\Delta<1/2$ be a constant to be decided later, which we fix throughout this argument. i.e. Our main term, the major arcs regime, will correspond to
\begin{equation}\label{eq:M(P)def}
N_0(P):=|P|^n\sum_{0\leq Y\leq \Delta Q}\,\,\,\,\,\sum_{\substack{r,d\textrm{ { monic,} }\uc \textrm{ { primitive}}\\ |d\uc|\leq \hat{Y/2},\, |dc_2|<\hat{Y/2}\\ |r|=\hat{Y},d\mid r}}|r_N|^{-n}
 \int\limits_{\substack{|\uz|<\hat{Y}^{-1}q^{-Q/2}}}S_{d\c,r,\b,N}(\vecnull)I_{r_N}(\uz;\vecnull)d\uz.
\end{equation}
The rest of the terms will contribute to the error, which we denote by $E(P)$, the minor arcs contribution. Here the dependence of both the terms on $\Delta$ is implicit.

We first observe that using the trivial bound $|S(\ua/r+\uz)|\ll |P|^n$ we may satisfactorily bound the contribution from the regions $$|\uz|< |P|^{-5},$$ to \eqref{eq:D-Refine}  directly. For any $Y\leq Q$, where $Q$ is as in \eqref{eq:Qchoice}, the measure $$\textrm{meas}(|\uz|< |P|^{-5})\ll |P|^{-10+\ve}.$$ Using this fact, for any $\ve>0$, the total contribution from this region to \eqref{eq:D-Refine} is at most
\begin{align*}
\sum_{Y=0}^Q\sum_{\substack{|r|=\hat{Y}\\ r \textrm{ monic }}}\sum_{\substack{|\ua|<\hat{Y}\\ \gcd(\ua,r)=1}}\int\limits_{|\uz|< |P|^{-5}}|S(\ua/r+\uz)|\leq \sum_{Y=0}^Q \sum_{\substack{|r|=\hat{Y}\\ r \textrm{ monic }}}\sum_{\substack{|\ua|<\hat{Y}\\ \gcd(\ua,r)=1}}|P|^{n-10+\ve}\ll |P|^{n-6+\ve},
\end{align*}
using the fact that $\hat{Q}^3\ll |P|^{4}$.
In the light of this bound, we may ignore the contribution from the region corresponding to the integrals over $ |\uz|< |P|^{-5}$ in our error term $E(P)$. Incorporating this observation, we will further split the error term in two major parts:
\begin{align}\label{eq:D-Refine1}
\int_{\TT^2}S(\ualf)d\ualf=N_0(P)+E_1(P)+E_2(P)+O_\ve(|P|^{n-6+\ve}),
\end{align}  
where 
\begin{equation}
\label{eq:E1def}
E_1(P):=|P|^n\sum_{0\leq Y\leq Q}\,\,\,\,\sum_{\substack{r,d\textrm{ { monic,} }\uc \textrm{ { primitive}}\\ \hat{Y-Q/2}\leq|d\uc|\leq \hat{Y/2}\\ |dc_2|<\hat{Y/2}\\ |r|=\hat{Y},d\mid r}}|r_N|^{-n}
 \int\limits_{\substack{|P|^{-5}\leq |\uz|<\hat{Y}^{-1}q^{-Q/2}}}\sum_{\substack{\v\in\scrO^n\setminus\vecnull,\\|\v|\leq \hat{V}}}S_{d\uc,r,\b,N}(\v)I_{r_N}(\uz;\v)d\uz,
\end{equation}
and
\begin{equation}
\label{eq:E3def}
E_2(P):=|P|^{n}\sum_{Q\Delta< Y\leq Q}\sum_{\substack{d\textrm{ monic, }\uc \textrm{ primitive}\\ \hat{Y-Q/2}\leq|d\uc|\leq \hat{Y/2}\\ |dc_2|<\hat{Y/2}}}|r_N|^{-n}
  \int\limits_{\substack{|P|^{-5}\leq |\uz|<\hat{Y}^{-1}q^{-Q/2}}}\sum\limits_{\substack{|r|=\hat{Y}\\ r \textrm{ monic, }d\mid r}}S_{d\uc,r,\b,N}(\vecnull)I_{r_N}(\uz;\v)d\uz.
\end{equation}
\subsection{The main term}\label{sec:main term}
We begin by establishing the required asymptotic formula for our main term $N_0(P)$. Throughout, we will treat $q$ as fixed and the implied constants may depend on it. When $\v=\vecnull$, the exponential integral $I_r(\uz,\vecnull)$ is independent of $r$, which we denote by $I(\uz)$ for simplicity, i.e. set
$$I(\uz)=\int_{K_\infty^n}\omega(\x)\psi\left(z_1P^2F_1(\x)+z_2P^2F(\x)\right)d\x. $$
Thus,
\begin{equation*}
N_0(P)=|P|^n  \sum_{\substack{ r \textrm{ monic }\\ |r|\leq Q^\Delta}}|r_N|^{-n}S_r\int_{|\uz|<|r|^{-1}q^{-Q/2}}I(\uz)d\uz.
\end{equation*}
 where
\begin{equation}\label{eq:S(r)def}
S_r=\sum\limits_{\substack{d\textrm{ monic, } \uc \textrm{ primitive }\\ |d\uc|\leq |r|^{1/2}\\ |dc_2|<|r|^{1/2}\\ d\mid r}}S_{d\uc,r,\b,N}(\vecnull)=\sumstar_{|\ua|<|r|}\sum_{\substack{|\x|<|r_N|\\\x\equiv \b\bmod{N}}}\psi\left(\frac{a_1F_1(\x)+a_2F_2(\x)}{r}\right).
\end{equation}
Here, the second equality is obtained from 
using Corollary \ref{cor:2}. We begin by proving the convergence of the singular series assuming the validity of the bound in Lemma \ref{lem:ExpsumsComb}, which will be proved in the following section:
\begin{lemma}\label{lem:singul}
For any $Y\geq 1$, and for any $\ve>0$,
\begin{align*}
\sum_{\substack{r\in\scrO\\ r\textrm{ {\em monic} }\\ |r|=\hat{Y}}}|r_N|^{-n}|S_r|\ll \hat{Y}^{(7-n)/2+\ve}.
\end{align*}
\end{lemma}
\begin{proof}
We assume the bound \eqref{eq:baddd}, which gives us:
\begin{align}\label{eq:singul}
\sum_{|r|=\hat{Y}}|S_r|\ll \sum_{|r|=\hat{Y}}\sum\limits_{\substack{d\textrm{ monic, } \uc \textrm{ primitive }\\ |d\uc|\leq |r|^{1/2}\\ |dc_2|<|r|^{1/2}\\ d\mid r}}|S_{d\uc,r,\b,N}(\vecnull)|\ll \sum_{|r|=\hat{Y}}\hat{Y}^{n/2+3/2}\sum\limits_{\substack{d\textrm{ monic, } \uc \textrm{ primitive }\\ |d\uc|\leq |r|^{1/2}\\ |dc_2|<|r|^{1/2}\\ d\mid r}}|d|^{1/2}\ll \hat{Y}^{n/2+7/2+\ve},
\end{align}
 establishing the bound.
\end{proof}
We next deal with the integral over $\uz$. We split it over $ \{|\uz|<\hat{C}|P|^{-2}\}$ and $\{\hat{C}|P|^{-2}\leq |\uz|< |r|^{-1}q^{-Q/2}\}$, where $ C>0$ is a fixed positive integer to be decided later. To bound the contribution of the second term, we use Lemma \ref{lem:I-hard}. Thus, for any $Z\geq |P|^{-2}$, we have
\begin{align*}
\int_{|\uz|=\hat{Z}}|I(\uz)|d\uz\ll \hat{L}^{-n}|P|^{-2}\log(|P|^2\hat{Z})\hat{Z}(1+|P|^2\hat{Z})^{1-n/2}\ll_\ve \hat{L}^{-n}|P|^{-2}\hat{Z}(1+|P|^2\hat{Z})^{1-n/2+\ve}.
\end{align*}
After summing over $Z$ and replacing $Z_1=|P|^2\hat{Z}$ for $n\geq 7$,
\begin{align*}
\int_{\hat{C}|P|^{-2}\leq |\uz|}|I(\uz)|d\uz\leq |P|^{-4}\sum_{\hat{C}\leq Z_1}(1+Z_1)^{-3/2+\ve}\ll |P|^{-4}\hat{C}^{-1/2+\ve}.
\end{align*}
This bound, in conjunction with Lemma \ref{lem:singul} assert that for $n\geq 8$ we have
\begin{equation}
\begin{split}
N_0(P)&=|P|^{n}\mathfrak{S}(\hat{Q}^\Delta)\int_{|\uz|<\hat{C}|P|^{-2}}\int\omega(\x)\psi(P^2z_1F_1(\x)+P^2z_2F_2(\x))d\x d\uz+O(|P|^{n-4}\hat{L}^{-n}\hat{C}^{-1/2+\ve})\\
&=|P|^{n-4}\mathfrak{S}(\hat{Q}^\Delta)\int_{|\uz|<\hat{C}}\int\omega(\x)\psi(z_1F_1(\x)+z_2F_2(\x))d\x d\uz+O(|P|^{n-4}\hat{L}^{-n}\hat{C}^{-1/2+\ve}).
\end{split}
\end{equation}
Here, given $Y\in \RR_{\geq 0}$,
$$\mathfrak{S}(\hat{Y})=\sum_{\substack{r\in\scrO, r \textrm{ monic }\\ |r|\leq \hat{Y}}}|r_N|^{-n} S_r,$$
is a truncated singular series.
We now switch the order of integrals over $\x$ and over $\uz$ and employ Lemma \cite[Lemma 2.2]{Browning_Vishe15} to obtain:
\begin{align*}
\int\omega(\x)\int_{|\uz|<\hat{C}}\psi(z_1F_1(\x)+z_2F_2(\x)) d z_1d z_2d\x=\hat{C}^2\meas\{|\x-\x_0|<\hat{L}^{-1}: |F_1(\x)|<\hat{C}^{-1}, |F_2(\x)|<\hat{C}^{-1}\}.
\end{align*}
Let us investigate the measure of the above set. Upon a change of variable, this is bounded by
\begin{equation}\label{eq:mE}
\hat{L}^{-n}\meas\{|\x|<1: |F_1(t^{-L}\x+\x_0)|<\hat{C}^{-1}, |F_2(t^{-L}\x+\x_0)|<\hat{C}^{-1}\}.
\end{equation}
For $i=1,2$, from \eqref{eq:HGbound} we get
\begin{align*}|F_i(t^{-L}\x+\x_0)|<H_{\F}\hat{L}^{-1}.
\end{align*}
 We may now choose $L$ to be an even integer $2\leq L$ such that $H_{\F}\leq \hat{L/2}$ for $i=1,2$, and choose $C=L/2$. Thus, for such a choice of $L$ and $C$, we get
\begin{align*}
\int_{|\uz|<\hat{C}}\int\omega(\x)\psi(z_1F_1(\x)+z_2F_2(\x))d\x d\uz=\hat{C}^2\hat{L}^{-n}=\hat{L}^{-n+1}.
\end{align*}
Finally, as a consequence of Lemma \ref{lem:singul}, we have also established the convergence of the singular series, namely
\begin{align*}
|\mathfrak{S}(\hat{Q}^\Delta)-\mathfrak{S}|\ll \hat{Q}^{-\Delta/2+\ve}\ll |P|^{-2\Delta/3+\ve},
\end{align*}
where
\begin{align*}
\mathfrak{S}=\sum_{r\in\scrO, r \textrm{ monic }}|r_N|^{-n}S_r,
\end{align*}
be the usual singular series. If $X(\AA_K)\neq\emptyset$, \cite[Cor. 7.7]{Lee11} establishes that $\mathfrak{S}>0$. The argument in \cite[Cor. 7.7]{Lee11} is obtained for $\b=\vecnull, N=1$, however, adapting it to deal with a fixed and general $\b,N$ is a routine exercise, which we skip here.

To summarise, we have established the following asymptotic formula:
\begin{lemma}\label{lem:Majorfinal}
For $n\geq 8$, for any even integer $L$ satisfying $H_\F\leq \hat{L/2}$, and any $0<\Delta<1/2$, we have
\begin{align*}
N_0(P)=\mathfrak{S}|P|^{n-4}\hat{L}^{-n+1}+O(|P|^{n-4}\hat{L}^{-n-1/4+\ve})+O(\hat{L}^{-n+1}|P|^{n-4-2\Delta/3+\ve}),
\end{align*} 
where $\mathfrak{S}>0$ if $X(\AA_K)\neq\emptyset$.
\end{lemma}

\section{Complete exponential sums bounds}\label{sec:expsum}
In this section, we will focus on getting satisfactory bounds for the exponential sums $S_{d\uc,r,\b,N}(\v)$. The notation and the results in Sec. \ref{sec:background} will be used throughout this section. Throughout, let $\v\in \scrO^n $, let $d\in\scrO$ be monic and $\uc\in\scrO^2$ be primitive. Recall that given any $r\in \scrO$, we consider the exponential sums
$$S_{d\uc,r,\b,N}(\v)=\sum_{\ua/r\in L(d\uc)}\sum_{\substack{\x\in\scrO^n\\|\x|<|r_N|\\\x\equiv \b\bmod{N}}}\psi\left(\frac{a_1F_1(\x)+a_2F_2(\x)}{r}\right)\psi\left(\frac{-\v\cdot\x}{r_N}\right).$$

The multiplicativity relation in Lemma \ref{lem:Multipli} will allow us to consider exponential sums modulo powers of primes $\vp^k$. Note that as per our definition, our set of bad primes, defined in section \ref{sec:goodbad}, includes all primes dividing $N$. We will begin by obtaining bounds for the exponential sums modulo $\vp^k$, where $\vp$ is a type I prime, which does not divide $d$. These translate to traditional quadratic exponential sums corresponding to the quadratic form $F_\uc=-c_2F_1+c_1F_2$, which have been considered in Lemma \ref{lem:Expsum'}. The treatment of type II primes will be similar to that of bad $\uc$'s.

\subsection{Exponential sum bounds I}
This part will be devoted to obtaining bounds for $S_{\uc,r,\vecnull, 1}(\v)$, i.e., when $d=1$ and $\vp$ is not a bad prime. When $d=1$, Lemma \ref{lem:Diomain} implies that the exponential sums $S_{\uc,r,\vecnull,1}(\v)$ are equal to a familiar quadratic exponential sums:
\begin{equation}\label{eq:S,,}
S_{\uc,r,\vecnull,1}(\v)=\sumstar_{|a|<|r|}\sum_{\substack{\x\in\scrO^n\\|\x|<|r|}}\psi\left(\frac{a(-c_2F_1(\x)+c_1F_2(\x))-\v\cdot\x}{r}\right).
\end{equation} 
Throughout this section, let \begin{equation}\label{eq:fdef}f(\x):=F_\uc(\x)=-c_2F_1(\x)+c_1F_2(\x).\end{equation} As before $M_\uc=-c_2M_1+c_1M_2$ is the defining matrix for $f$. If we want to give up on the cancellations arising from the extra average over $a$, then using Lemma \ref{lem:pointwise} in the {\em generic} case, it is expected to be able to obtain square-root cancellations in the inner sum over $\x$ in \eqref{eq:S,,}, which would hand us the following generic bound: 
\begin{equation}\label{eq:generic}
|S_{\uc,r,\vecnull,1} (\v)|\ll |r|^{n/2+1}.
\end{equation}
We will use this bound only as a reference for comparing with various bounds showing up in this section.

\subsubsection{$\uc$ good case} Let us assume that $\uc$ is good. \cite[Lemma 2.1]{HeathBrown_Pierce17} implies that when $\vp$ is not a bad prime, $\rank_{\vp}(f(\x))\geq n-1$. Since $\uc$ is good, $\det(M_\uc)\neq 0 $. Therefore, the set of primes of type I consists of all good primes which do not divide $\det(M_\uc)$, and the set of primes of type II consists of good primes which divide $\det(M_\uc) $. Thus the cardinality of the set of type II primes is at most $O(\log|\uc|)$. We simplify our notation and define
\begin{equation}\label{eq:Sr}
S_{r}(\v)=\sumstar_{|a|<|r|}\sum_{|\x|<|r|}\psi\left(\frac{af(\x)-\v\cdot\x}{r}\right),
\end{equation}
where $f$ as in \eqref{eq:fdef}. Since $f$ is a quadratic form, we can explicitly evaluate $S_{\vp^k}(\v)$ when $\vp$ is a type I prime using Lemma \ref{lem:Expsum'}:
\begin{lemma}\label{lem:Expsum}
Let $\uc$ be a good pair and let $\vp$ be a prime of type I. Let $|\vp|=q^L$, and $q=p^{\ell_0}$. Then
\begin{align*}
|S_{\vp^k}(\v)|\leq |\vp|^{(n+1)k/2}\gcd(f^*(\v),\vp^k)^{1/2},
\end{align*}
where $f^*(\v)=\det(M_\uc)\v^tM_\uc^{-1}\v$ is the dual form. More explicitly, we have:
\begin{align*}
S_{\vp^k}(\v)=\begin{cases}
|\vp|^{nk/2}(|\vp|^k\delta_{\vp^k\mid f^*(\v)}-|\vp|^{k-1}\delta_{\vp^{k-1}\mid f^*(\v)}), & \text{if }2\mid k,\\
\left(\frac{\det(M_\uc)}{\vp}\right)|\vp|^{kn/2}i_p^{L\ell_0n}(|\vp|^k\delta_{\vp^k\mid f^*(\v)}-|\vp|^{k-1}\delta_{\vp^{k-1}\mid f^*(\v)}), &\textrm{ if } 2\mid n,2\nmid k,\\
\left(\frac{-f^*(\v)}{\vp}\right)|\vp|^{k(n+1)/2}i_p^{L\ell_0(n+1)}, &\textrm{ if } 2\nmid n,2\nmid k,
\end{cases}
\end{align*}
with $i_p$ as in \eqref{eq:ipdef}. 
\end{lemma}
\begin{remark}\label{rem:1}
Let us consider various implications of the bounds in Lemma \ref{lem:Expsum}. The bounds depend on the parities of $n$ and $k$. When $r$ is generic, i.e., when $\gcd(r,f^*(\v))=1$, we may always save a factor of size $|r|^{1/2}$ as compared with \eqref{eq:generic}. We will save another factor of size $O(|r|^{1/2})$ from an average over the square-free values of $r$. As a result, we are able to adequately bound $E(P)$ as long as $n\geq 9$. When $n=8$, and $r$ is square-free and generic, Lemma \ref{lem:Expsum} hands us a $O(|r|^{n/2})$ bound instead of \eqref{eq:generic}, effectively saving a factor of size $O(|r|)$ without even utilising the average over $r$. In theory, this should lead us to settle this case. However, when $r\mid f^*(\v)$, we are handed back the bound in \eqref{eq:generic}. Moreover, $f^*(\v)$ depends both on $\v$ as well as on $\uc$, and this is the primary reason why we are unable to deal this contribution in a satisfactory manner. 
\end{remark}

When $\vp$ is a prime of type II, our bounds will not be as good as those in Lemma \ref{lem:Expsum}. Let $M_\uc=TDS$, where $T,S$ are invertible matrices as in Sec. \ref{sec:goodbad} with entries in $\scrO$ and $D=\diag(\mu_1,...,\mu_n)$ is a diagonal matrix satisfying $\mu_i\mid \mu_{i+1}$. Let $\{\y_j=S^{-1}\e_j\}$ be a basis for $\scrO^n$, and recall that the quadratic form $$Q_\uc(x_1,...,x_{n-1})=f(x_1\y_1+...+x_{n-1}\y_{n-1})$$ defined in \eqref{eq:Qucdef} is non-singular modulo $\vp$. Clearly, $\vp\mid M_\uc\y_n=\mu_nT\e_n$. We will therefore end up giving up on an extra factor of size $\gcd(\vp^k,\det(M_\uc))^{1/2}=\gcd(\vp^k,\mu_n)^{1/2}$, as compared with the bound in \eqref{eq:generic}. However, we will salvage this loss somewhat by obtaining a congruence condition on the vector $\v$:
\begin{lemma}
\label{lem:type II}
Let $\vp$ be a prime of type II, and let $k_1=\min\{k,\nu_\vp(\mu_n)\}$. Then,
\begin{equation}
|S_{\vp^k}(\v)|\leq |\vp|^{k(n/2+1)}\delta_{\vp^{k_1}\mid ((S^{-1})^t\v)_n}\gcd(\vp^{k_1},Q_\uc^*(\v'))^{1/2},
\end{equation}
where $Q_\uc^*$ denotes the dual of the quadratic form $Q_\uc$, and $((S^{-1})^t\v)_n $ denotes the $n$-th entry of the vector $(S^{-1})^t\v$, and $\v'$ denotes the $n-1$ dimensional vector obtained by deleting the $n$-th entry of $(S^{-1})^t\v$. As a consequence,
\begin{equation}
|S_{\vp^k}(\v)|\leq |\vp|^{k(n/2+1)}\gcd(\vp^{k_1},Q_\uc^*(\v'),((S^{-1})^t\v)_n)^{1/2},
\end{equation}
\end{lemma} 
\begin{proof}
Recall that $M_\uc=TDS$ where $T,S$ are in $\GL_n(\scrO)$ with $\det(T),\det(S)\in \FF_q^\times$. Since $\vp$ is a prime of type II, $\vp\mid \mu_n$, and $\vp\nmid \mu_j$ for any $1\leq j\leq n-1$. Let $Q(\x)=f(S^{-1}\x)$.
\begin{align}\label{eq:type2}
S_{\vp^k}(\v)&=\sumstar_{|a|<|\vp|^k}\sum_{|\x|<|\vp|^k}\psi\left(\frac{af(\x)-\v\cdot\x}{\vp^k}\right)=\det(S)^{-1}\sumstar_{|a|<|\vp|^k}\sum_{|\x|<|\vp|^k}\psi\left(\frac{aQ(\x)-((S^{-1})^t\v)\cdot \x}{\vp^k}\right).
\end{align}
We now change the variables to write $x_n=x_{n,1}+\vp^{k-k_1}x_{n,2}$, and $\x=\x_1+\x_2$ where $\x_2=(0,...,0,\vp^{k-k_1}x_{n,2})^t$. Note that $Q(\x)=\x^t(S^{-1})^t (T DS)S^{-1}\x$. Moreover, $M_\uc S^{-1}\x_2\equiv \vecnull\bmod{\vp^k}$, and therefore, using the symmetry of $M_\uc$, we must have $\x_2^t(S^{-1})^tM_\uc\equiv \vecnull^t\bmod{\vp^k}$, as well. Therefore, the value of $Q(\x)\bmod{\vp^k}$ is independent of $x_{n,2}$. We thus get:
\begin{align*}
S_{\vp^k}(\v)&=\det(S)^{-1}\sumstar_{|a|<|\vp|^k}\sum_{\x_1}\psi\left(\frac{aQ(\x_1)-((S^{-1})^t\v)\cdot \x_1}{\vp^k}\right)\sum_{|x_{n,2}|<|\vp|^{k_1}}\psi\left(\frac{((S^{-1})^t\v)_nx_{n,2}}{\vp^{k_1}}\right).
\end{align*}
The inner sum vanishes unless $\vp^{k_1}\mid ((S^{-1})^t\v)_n$.
On the other hand, Lemma \ref{lem:pointwise} gives
\begin{align*}
|S_{\vp^k}(\v)|^2\leq |\vp|^{k(2+n)}\#\{\x\bmod{\vp^k}:\vp^k\mid M_\uc\x\}.
\end{align*}
Using the Smith normal form again, $$\#\{\x\bmod{\vp^k}:\vp^k\mid M_\uc\x\}=\#\{\x\bmod{\vp^k}:\vp^k\mid D\x\}=|\vp|^{k_1}, $$
using the fact that $S$ and $T$ are invertible. This provides us with our first bound:
\begin{equation}\label{eq:Bound1}
|S_{\vp^k}(\v)|\leq \vp^{k_1/2}|\vp|^{k(n/2+1)}\delta_{\vp^{k_1}\mid ((S^{-1})^t\v)_n}.
\end{equation}
Unfortunately, this bound is not enough for us. Therefore we go back to \eqref{eq:type2}, and evaluate the sum in a different way. This time we write $\x=\x'+x_n\e_n$, where $\e_n=(0,...,0,1)$ as before, to get:
 \begin{align*}
|S_{\vp^k}(\v)|&=\left|\sumstar_{|a|<|\vp|^k}\sum_{|\x|<|\vp|^k}\psi\left(\frac{aQ(\x)-((S^{-1})^t\v)\cdot \x}{\vp^k}\right)\right|\\&\leq \sum_{|x_n|<|\vp|^k}\left|\sumstar_{|a|<|\vp|^k}\sum_{|\x'|<|\vp|^k}\psi\left(\frac{aQ(\x'+x_n\e_n)-\v'\cdot \x'}{\vp^k}\right)\right|.
\end{align*}
We now invoke our general bound \eqref{eq:firstB} in Lemma \ref{lem:Expsum'} by  applying it to the inner exponential sums with the quadratic polynomial $g(\x')=Q(\x'+x_n\e_n)$. Note that following the above notation, $Q(\x')=Q_\uc(\x')$ is the leading quadratic part of $g(\x')$. We are thus left with:
\begin{align*}
|S_{\vp^k}(\v)|\leq |\vp|^{k(n/2+1)}\gcd(Q_\uc^*(\v'),\vp^k)^{1/2}.
\end{align*}
The lemma now follows upon taking the minimum of this bound and the one in \eqref{eq:Bound1}. 
\end{proof}

\subsubsection{$\uc$ bad case ($f$ singular)}
The strategy for dealing with the bad values of $\uc$ will emulate that of type II primes. Note that $\vp\mid \mu_{n-1}$ if and only if $\vp$ is a bad prime. After using the change of variables as in \eqref{eq:type2}, we have
\begin{align*}
S_{r}(\v)&=\det(S)^{-1}\sumstar_{|a|<|r|}\sum_{|\x|<|r|}\psi\left(\frac{aQ(\x)-((S^{-1})^t\v)\cdot \x}{r}\right)\\&=\det(S)^{-1}|r|\delta_{r\mid ((S^{-1})^t\v)_n}\sumstar_{|a|<|r|}\sum_{|\x_1|<|r|}\psi\left(\frac{aQ_\uc(\x_1)-\v'\cdot \x_1}{r}\right),
\end{align*}
where $((S^{-1})^t\v)_n$ is the $n$-th entry of the vector $(S^{-1})^t\v$, $\v'$ denotes the $n-1$ dimensional vector obtained after deleting the $n$-the entry in $(S^{-1})^t\v$, and $\x_1=(x_1,...,x_{n-1}) $. The last exponential sum can again be evaluated using Lemma \ref{lem:Expsum'} to obtain:
\begin{lemma}
\label{lem:expsumsingular}
Let $f$ be singular, and let $\vp$ is not a bad prime. Let $|\vp|=q^L$, and $q=p^{\ell_0}$. Then we have:
\begin{equation*}
\begin{split}
S_{\vp^k}(\v)=|\vp|^k&\delta_{\vp^k\mid ((S^{-1})^t\v)_n}\\&\times\begin{cases}
|\vp|^{(n-1)k/2}(|\vp|^k\delta_{\vp^k\mid Q_\uc^*(\v')}-|\vp|^{k-1}\delta_{\vp^{k-1}\mid Q_\uc^*(\v')}), & \text{if }2\mid k,\\
\left(\frac{\det(M_\uc')}{\vp}\right)|\vp|^{k(n-1)/2}i_p^{L\ell_0(n-1)}(|\vp|^k\delta_{\vp^k\mid Q_\uc^*(\v')}-|\vp|^{k-1}\delta_{\vp^{k-1}\mid Q_\uc^*(\v')}), &\textrm{ if } 2\nmid n,2\nmid k,\\
\left(\frac{-Q_\uc^*(\v')}{\vp}\right)|\vp|^{kn/2}i_p^{L\ell_0n}, &\textrm{ if } 2\mid n,2\nmid k,
\end{cases}
\end{split}
\end{equation*}
where  $M_\uc'$ is the matrix defining $Q_\uc$, $\v'$ as in Lemma \ref{lem:type II}, and $i_p$ as in \eqref{eq:ipdef}. 
\end{lemma}

\subsection{A general bound}
So far, the above bounds suffice as long as $\vp\nmid dD_\F$. We first shift the focus to $\vp\mid d$. Using the multiplicativity of the exponential sums in Lemma \ref{lem:Multipli}, it is enough to look at the sums of type $S_{\vp^{m}\uc,\vp^{k}, \b,\vp^{\ell}}(\v)$, where $m\leq k$. As before, let us first assume that $\vp\nmid N$. First, we begin by investigating the structure of points $\ua/\vp^k\in L(\vp^{m}\uc)$. From our definition \eqref{eq:ducdef}, when $m\geq 1$, $\ua/\vp^k\in L(\vp^{m}\uc)$ if and only if the conditions $\gcd(\ua,\vp)=1, \vp^{k-m}\mid\ua\cdot\uc$ and $\vp\nmid (\ua\cdot\uc/\vp^{k-m})$ simultaneously hold. Lemma \ref{lem:Dio} implies that
\begin{align*}
\{\ua\bmod{\vp^k}:\ua/\vp^k\in L(\vp^{m}\uc)\}\subseteq\{a\uc^\perp+\vp^{k-m}\ud\bmod{\vp^k}:|a|<|\vp|^{k-m},\gcd(a,\vp)=1,\gcd(\ud,\vp)=1\}.
\end{align*}
$\ud$ also needs to satisfy an extra condition that $\vp\nmid\uc\cdot\ud$, which forces that $\ud$ itself can not be of the form $a'\uc^\perp+\vp\ud_1$, where $0\leq |a'|<|\vp|$. Therefore, this concludes that when $m<k$, we have the following equality of the sets modulo $\vp^k$:
\begin{align}
\nonumber\{\ua:\ua/\vp^k\in L(\vp^{m}\uc)\}&=\{a\uc^\perp+\vp^{k-m}\ud:|a|<|\vp|^{k-m},\gcd(a,\vp)=1,|\ud|<|\vp|^m\}\setminus\\
&\{a\uc^\perp+\vp^{k-m+1}\ud:\gcd(a,\vp)=1,|a|<|\vp|^{k-m+1},|\ud|<|\vp|^{m-1}\},\label{eq:kneqm}
\end{align}
while when $k=m$, we get
\begin{align}
\nonumber\{\ua:\ua/\vp^k\in L(\vp^{k}\uc)\}&=\{\ud:\gcd(\ud,\vp)=1,|\ud|<|\vp|^k\}\setminus\\
&\{a\uc^\perp+\vp^{k-1}\ud:\gcd(a,\vp)=1,|a|<|\vp|,|\ud|<|\vp|^{k-1}\}.\label{eq:k=m}
\end{align}
Using the above structure, it is easy to obtain the bound:
\begin{equation}
\label{eq:Lbound}
\#\{\ua/\vp^k\in L(\vp^m\uc)\}\leq |\vp|^{k-m+2m}=|\vp|^{k+m}.
\end{equation}
When $\uc$ is a bad pair, we will need to obtain some saving from the primes which divide the square-free part of $d$. It will be enough to obtain the following bound:
\begin{lemma}
\label{lem:TypeIgen}
Let $\uc$ be a bad pair and let $\vp$ not be a bad prime further satisfying $\gcd(\vp, \scrF^*(\v))=\gcd(\vp,Q_\uc^*(\v'))=1$, then
\begin{align*}
|S_{\vp,\uc,\vp,\vecnull,1}(\v)|\leq |\vp|^{n/2+1}.
\end{align*}
\end{lemma}
\begin{proof}
\eqref{eq:k=m} implies
\begin{align*}
S_{\vp\uc,\vp,\vecnull,1}(\v)=\sum_{\substack{|\ud|<|\vp|\\ \gcd(\ud,\vp)=1}}\sum_{|\x|<|\vp|}\psi\left(\frac{d_1F_1(\x)+d_2F_2(\x)-\v\cdot \x}{\vp}\right)-S_{\uc,\vp,\vecnull,1}(\v).
\end{align*}
The bound here follows for a standard Deligne bound (see \cite[Lem. 14]{Katz} for example) for the complete exponential sums, and our bounds in Lemma \ref{lem:expsumsingular}.
\end{proof}
Note that the method of the above lemma could be generalised to obtain further savings from $S_{\vp^m\uc,\vp^k,\vecnull,1}(\v)$, when $k,m\neq 1$, however this is not needed in this work.

When $\vp$ is a bad prime, we know that $|\vp|$ is absolutely bounded. The argument of  \cite[Lemma 5.5]{HeathBrown_Pierce17} holds here as well, as it only depends on the fact that $f(\x)$ has dimension at least $ k-1$ over $K_\vp$. \cite[Lemma 5.5]{HeathBrown_Pierce17} thus provides us:
\begin{lemma}
For each bad prime $\vp$, there is a constant $c_\vp$ such that
$$\nu_\vp(\mu_{n-1})\leq c_\vp.$$
\end{lemma}

We now turn our attention to a more general bound which can be seen as a combination of methods in Lemma \ref{lem:type II} and \cite[Lemma 5.4]{HeathBrown_Pierce17}. In the light of \eqref{eq:Lbound}, the bound obtained in the following Lemma, upto a factor of $|\vp|^{k_1/2} $, is a direct analogue of \eqref{eq:generic} in this case. The loss of the factor $|\vp|^{k_1/2}$ essentially arises from $\gcd(\vp^{k-m},F(\uc))$, where $F(\uc)=\det(M_\uc)$. Akin to Lemma \ref{lem:type II}, we compensate the loss of this factor by obtaining a congruence condition on $\v$.
\begin{lemma}
\label{lem:Expweak}
Let $\uc$ be any primitive pair. Then for any good prime $\vp$, and for any $1\leq m\leq k$ we have
\begin{align}\label{eq:enough}
|S_{\vp^m\uc,\vp^k,\vecnull,1}(\v)|\leq |\vp|^{k(n/2+1)+m+k_1/2}\delta_{\vp^{k_1}\mid ((S^{-1})^t\v)_n}.
\end{align} 
When $\vp$ is a bad prime, then 
\begin{align}\label{eq:enough1}
|S_{\vp^m\uc,\vp^k,\b,\vp^\ell}(\v)|\leq C_{\vp,\ell} |\vp|^{k(n/2+1)+m+k_1/2}\delta_{\vp^{k_2}\mid ((S^{-1})^t\v)_n},
\end{align}
where $C_{\vp,\ell}$ is a constant which only depends on $|\vp|$ and $\ell$. Here $k_1=\min\{k-m,\nu_\vp(\mu_n)\}$, and $k_2=\min\{k-m,\nu_\vp(\mu_n),k-\ell\} $. 
\end{lemma}
\begin{proof}
Since the set of bad primes is bounded and $N$ is fixed, without loss of generality, we may assume that $\ell\leq k/3$. This dependence may be absorbed in the constant.  Recall that
\begin{align*}
S_{\vp^m\uc,\vp^k,\b,\vp^\ell}(\v)&=\sum\limits_{\substack{|\ud|<|\vp|^m\\ \vp\nmid\uc\cdot\ud}}\sumstar_{|a|<|\vp|^{k-m}}\sum\limits_{\substack{|\x|<|\vp|^k\\ \x\equiv\b\bmod{\vp^\ell} }}\psi\left(\frac{(a\uc^\perp+\vp^{k-m}\ud)\cdot(F_1(\x),F_2(\x))-\v\cdot\x}{\vp^k}\right)\\
&=\sum\limits_{\substack{|\ud|<|\vp|^m\\ \vp\nmid\uc\cdot\ud}}\sumstar_{|a|<|\vp|^{k-m}}\sum\limits_{\substack{|\x|<|\vp|^k\\ \x\equiv\b\bmod{\vp^\ell} }}\psi\left(\frac{af(\x)+\vp^{k-m}(d_1F_1(\x)+d_2F_2(\x))-\v\cdot\x}{\vp^k}\right).
\end{align*}
We follow the recipe of Lemma \ref{lem:type II} to first change the variables and write $\y=S\x$, and then write $\y=\y_1+\vp^{k-k_2}\y_2$, where $\y_2=(0,...,0,y_2)$. It is easy to see that $f(S^{-1}\y)=f(S^{-1}\y_1)$. Moreover, the congruence condition is converted to $\y_1\equiv S\b\bmod{\vp^\ell}$. As a result, akin to the argument in Lemma \ref{lem:type II}, the sum over $\y_2$ hands us the condition $\vp^{k_2}\mid ((S^{-1})^t\v)_n$. 

On the other hand, we substitute $\x=\b+\vp^\ell\y$ and apply the bound in Lemma \ref{lem:pointwise} to get
\begin{align}\notag
|S_{\vp^m\uc,\vp^k,\b,\vp^\ell}(\v)|&\leq C_{\vp,\ell}' |\vp|^{n(k-\ell)/2}\sum\limits_{|\ud|<|\vp|^m,\gcd(\ud,\vp)=1}\,\,\,\sumstar_{|a|<|\vp|^{k-m}}N(a\uc^\perp+\vp^{k-m}\ud,\vp^{k})^{1/2}\\
&\leq  C_{\vp,\ell}' |\vp|^{n(k-\ell)/2+k-m}\sum\limits_{|\ud|<|\vp|^m,\gcd(\ud,\vp)=1}\gcd(F(\uc+\vp^{k-m}\ud),\vp^k)^{1/2}\label{eq:no}
\end{align}
where $$N(\ua,\vp^k)=\#\{\x\bmod{\vp^{k}}:\vp^{k}\mid(a_1M_1+a_2M_2)\x\},$$
and $F(x,y)$ is the determinant form defined in \eqref{eq:Fxydef}. 

If $\vp$ is not a bad prime then $\nu_{\vp}(F(\uc))=\nu_{\vp}(\mu_n)$. Therefore, if $\nu_{\vp}(F(\uc))<k-m$, then $k_1=\nu_{\vp}(F(\uc))$. Moreover, for any choice of $\ud$ and $\gcd(a,\vp)=1$, $\nu_{\vp}(F(a\uc+\vp^{k-m}\ud))=k_1<k-m$ as well.  \eqref{eq:enough} further follows from \eqref{eq:no}.
When $\vp$ is a bad prime, then $\nu_{\vp}(\mu_n)\leq \nu_{\vp}(F(\uc))\leq \nu_{\vp}(\mu_n)+(n-1)c_\vp $. If we further have that $m_1=\nu_{\vp}(F(\uc))<k-m$, then \eqref{eq:enough1} follows from a minor modification of the argument above after observing that $|\vp|^{m_1}\leq |\vp|^{k_1+(n-1)c_\vp}$.

 It is therefore enough to assume that $\vp^{k-m}\mid F(\uc)$, which we do for the rest of the proof. This in turn implies that $k-m\leq k_1$ if $\vp$ is not bad and $k-m\leq k_1+(n-1)c_\vp$ otherwise. The rest of the argument will follow from minor modifications of the proof of \cite[Lemma 5.4]{HeathBrown_Pierce17}, which we reproduce below.

We start by rewriting \eqref{eq:no} as
\begin{align*}
&|S_{\vp^m\uc,\vp^k,\b,\vp^\ell}(\v)|\\&\leq C_{\vp,\ell} |\vp|^{nk/2+k-m}\sum_{g=0}^{m}|\vp|^{(g+(k-m))/2}\#\{|\ud|<|\vp|^m,\vp\nmid \ud,\gcd(\vp^{k},F(\uc+\vp^{k-m}\ud))=\vp^{g+k-m}\}.\\
&\leq C_{\vp,\ell}' |\vp|^{nk/2+k-m+k_1/2}(|\vp|^{2m}+\sum_{g=1}^{m}|\vp|^{g/2}\#\{|\ud|<|\vp|^m,\vp\nmid \ud,\gcd(\vp^{k},F(\uc+\vp^{k-m}\ud))=\vp^{g+k-m}\})\\
&\leq C_{\vp,\ell}' |\vp|^{nk/2+k-m+k_1/2+2m}(1+\sum_{g=1}^{m}|\vp|^{-3g/2}\#\{|\ud|<|\vp|^g,\vp\nmid \ud,\gcd(\vp^{k},F(\uc+\vp^{k-m}\ud))=\vp^{g+k-m}\})
\end{align*}
The number of $|\ud|<|\vp|^g$ such that the second co-ordinate of $\uc+\vp^{k-m}\ud$ is co-prime to $\vp$ is 
$$\leq |\vp|^{g}\#\{|u|<|\vp|^{k-m+g}: \vp^{k-m+g}\mid F(u,1) \}. $$
The main result in \cite{GGI} applied to the polynomial $F(u,1)$ and its derivative, implies that for any root $u_0\in\scrO$, satisfying $\vp\mid F(u_0,1)$, we must have $\nu_\vp(F'(u_0,1))\leq\nu_{\vp}(D_\F)$, where $D_\F$ is as in \eqref{eq:deltadef}. We may now further use Hensel's Lemma to obtain $$\#\{|u|<|\vp|^{k-m+g}: \vp^{k-m+g}\mid F(u,1) \}\leq n|D_\F|.$$ 
This bound is clearly enough. We can similarly bound the number of terms where the first co-ordinate is co-prime to $\vp$ to finish the proof.
\end{proof}
As an immediate corollary of the above Lemma, we get the following weak bound, which holds for any $r$ and any primitive $\uc$:
\begin{lemma}\label{lem:ExpsumsComb}Let $d,N\in\scrO$, $\uc\in\scrO^2$ be any primitive pair, and let $\b\in\scrO^n$. Given any $r$, we have
\begin{align}\label{eq:bdd1}
|S_{d\uc,r,\b,N}(\v)|&\ll_{D_\F}|d||r|^{n/2+1}\gcd(r/d,((S^{-1})^t\v)_n,\det(M_\uc))^{1/2}\\
& \ll_{D_\F}|d|^{1/2}|r|^{n/2+3/2}.\label{eq:baddd}
\end{align}
\end{lemma}

Observe that our bounds throughout this section are independent of the choice of $\b$ and depend only on $|D_\F|$. This will make their application rather convenient. 

\section{Square-free moduli contribution}\label{sec:sqfree}
Our rest of the effort will be spent in proving that $|E_i(P))|\ll |P|^{n-4-\ve}$, for $i=1,2$. We will begin by considering the term $E_1(P)$ as defined in \eqref{eq:E1def}. Let $|r|=\hat{Y}$ and let $|\uz|=\hat{Z}$. Let $J(Z)=1+|P|^2\hat{Z}$.
 Since $\v\neq \vecnull$, this forces,
\begin{equation}
\label{eq:YBound}
\hat{Y}\gg \frac{|P|}{J(Z)}.
\end{equation}

From now on, we fix  $0\leq Y\leq Q$, and $d\uc$ satisfying  $|d\uc|\leq \hat{Y/2}, |dc_2|<\hat{Y/2}$ and $Z\in\ZZ$ such that $-5\log_q |P|\leq Z<-Y-Q/2$. Note that there are only $O(|P|^\ve)$ choices for $Y$ and $Z$.
 Let $E_i(d\uc,Y,Z)$ denote the contribution to the term $E_i$ from this specific choice of $d\uc$ after summing over all monic $|r|=\hat{Y} $, and integrating over $|\uz|=\hat{Z}$. For example:
\begin{equation}
\label{eq:Edcdef}
E_1(d\uc, Y,Z):=|P|^n\sum\limits_{\substack{|r|=\hat{Y}\\d\mid r}}|r_N|^{-n}
 \int_{|\uz|=\hat{Z}}\sum_{\substack{\v\in\scrO^n\setminus\vecnull,\\|\v|\leq \hat{V}}}S_{d\uc,r,\b,N}(\v)I_{r_N}(\uz;\v)d\uz,
\end{equation}
 
 Let $\scrP$ denote a set of primes to be specified later, containing at least all primes dividing $dD_\F$. Next, we write $r=br_1$, where $b$ denotes the square free part of $r$ satisfying a further constraint: $\gcd(b,\scrP)=1$. 
Recalling the factorisation of the exponential sum in Lemma \ref{lem:Multipli}, there exist $\b_1\in (\scrO/N\scrO)^n, b_0\in (\scrO/N\scrO)^*$ satisfying
\begin{equation}\label{eq:E(P)-1}
E_1(d\uc,Y,Z)
\ll \frac{|P|^n}{\hat Y^{n}}
\sum_{\substack{ \v\in \cO^n\\
\v\neq \vecnull\\
|\v|\ll \hat{V}
}} 
\sum_{\substack{
r_1\in \cO, d\mid r_1\\ |r_1|\leq \hat{Y}}} 
|S_{d\uc,r_1,\b_1,N_1}(\v)|
|\Sigma(Z,r_1,\hat{Y}/|r_1|)| ,
\end{equation}
where
\begin{equation}\label{eq:access}
\Sigma(Z,y,\hat{B})
=
\int_{|\uz|=\hat{Z}}\sum_{\substack{
b\in \cO^\sharp
\\
(b,\scrP)=1\\
|b|=\hat{B} 
\\ b\equiv b_0\bmod{N}} }S_{\uc,b,\vecnull,1}(\v)
I_{by_N}(\uz;\v)d\uz.
\end{equation}
Let $f(\x)=-c_2F_1+c_1F_2$ as in the previous section. Using notation \eqref{eq:Sr}, we may rewrite this as
\begin{equation}\label{eq:access1}
\Sigma(Z,y,\hat{B})
=
\int_{|\uz|=\hat{Z}}\sum_{\substack{
b\in \cO^\sharp
\\
(b,\scrP)=1\\
|b|=\hat{B} 
\\ b\equiv b_0\bmod{N}} }S_{b}(\v)
I_{by_N}(\uz;\v)d\uz.
\end{equation}
In this section, we will derive a good bound for $\Sigma(Z,y,\hat{B})$, and eventually apply it with $y=r_1,\hat{B}=\hat{Y}/|r_1|$. Since our bounds for exponential sums differ with the parity of $n$, so will our treatment. We begin by noting a weaker bound which is a direct consequence of Lemmas \ref{lem:I-hard} and \ref{lem:Expsum}:
\begin{lemma}
\label{lem:sqfreeeasy}
Let $\ve>0$, let $\uc$ be monic primitive and good, and $\scrP$ be the set of primes dividing $d\det(M_\uc)f^*(\v)$ and $D_\F$ if $n$ is even, and the set of primes dividing $d\det(M_\uc)D_\F$ if $n$ is odd. Then
\begin{align*}
|\Sigma(Z,y,\hat{B})|\ll J(Z)^{-n/2+1}(\log|P|)\hat{Z}\min\{\hat{Z}, |P|^{-2}\}\hat{B}^{n/2+1}\begin{cases}1&\textrm{ if }2\mid n\\
\hat{B}^{1/2}&\textrm{ if }2\nmid n.
\end{cases}
\end{align*}
Let $\uc$ be bad, then let $\scrP$ denote the set of primes dividing $dD_\F Q_\uc^*(\v')$, then
\begin{align*}
|\Sigma(Z,y,\hat{B})|\ll J(Z)^{-n/2+1}(\log|P|)\hat{Z}\min\{\hat{Z}, |P|^{-2}\}\hat{B}^{n/2+3/2}\begin{cases}1&\textrm{ if }2\nmid n\\
\hat{B}^{1/2}&\textrm{ if }2\mid n.
\end{cases}
\end{align*}
\end{lemma}

In obtaining the above lemma, we are giving up on some extra cancellations we may be able to obtain from the sum over $b$. In order to exploit this, we need to look at this contribution more closely. We begin by noting that 
\begin{align*}
\Sigma(Z,y,\hat{B})&=\sum_{\substack{
b\in \cO^\sharp
\\
(b,\scrP)=1\\
|b|=\hat{B} 
\\ b\equiv b_0\bmod{N}} }S_{b}(\v)
\int_{|\uz|=\hat{Z}}I_{by_N}(\uz;\v)d\uz.
\end{align*}
We begin by focusing on the average value of the exponential integral:
\begin{align*}
\int_{|\uz|=\hat{Z}}I_{by_N}(\uz;\v)d\uz&=\int_{|\uz|=\hat{Z}}\int \omega(\x)\psi(z_1 P^2F_1(\x)+z_2P^2F_2(\x))\psi\left(\frac{P\v\cdot\x}{by_N}\right)d\x d\uz\\
&=\int_{|\uz|=\hat{Z}}\int w(t^{L}(\x-\x_0))\psi(z_1 P^2F_1(\x)+z_2P^2F_2(\x))\psi\left(\frac{P\v\cdot\x}{by_N}\right)d\x d\uz\\
&=\hat{L}^{-n}\int_{|\uz|=\hat{Z}}\int_{\TT^n}\psi(z_1P^2F_1(t^{-L}\x+\x_0))+z_2F_2(t^{-L}\x+\x_0))\psi\left(\frac{P\v\cdot(t^{-L}\x+\x_0)}{by_N}\right)d\x d\uz\\
&=\hat{L}^{-n}\psi\left(\frac{P\v\cdot\x_0}{by_N}\right)\int_{|\uz|=\hat{\uZ}}J_\G(z_1P^2,z_2P^2,t^{-L}P\v/(by_N))d\uz,
\end{align*}
where $G_i=F_i(t^{-L}\x+\x_0)$. Note $H_\G< \hat{L}^{-1}H_\F$ as noted in \eqref{eq:HGbound}. Using Lemma \ref{lem:small}, for any $\w$, 
\begin{align*}
\int_{|\uz|=\hat{\uZ}}J_\G(z_1P^2,z_2P^2,\w)d\uz=
\int_{|\uz|=\hat{Z}} \int_{\Omega} \psi\left(z_1P^2 G_1(\x)+z_2P^2G_2(\x)
+\w.\x \right) d\x d\uz,
\end{align*}
where
\begin{equation}
\label{eq:Omegadef}
\Omega=\left\{\x\in \TT^n: |z_1P^2G_1(\x)|,|z_2P^2G_2(\x)|\leq \max\{1,H_\G\} J(Z)^{1/2}, |P^2\uz\cdot\nabla \G(\x)+ \w|\leq H_\G J(Z)^{1/2}\right\}.
\end{equation}
Here, $\uz\cdot \nabla \G(\x):=z_1\nabla G_1(\x)+z_2\nabla G_2(\x) $.
We now replace $\w=t^{-L}P\v/(by_N)$ and $G_i(\x)=F_i(t^{-L}\x+\x_0)$. Thus, after noting that $|t^{-L}\x+\x_0|\leq 1$ for all $\x\in \TT^n$, we have
\begin{align*}
&|P^2t^{-L}(z_1\nabla F_1(t^{-L}\x+\x_0)+z_2\nabla F_2(t^{-L}\x+\x_0))+\w|\leq H_\G J(Z)^{1/2}\\
\Rightarrow&|P^2t^{-L}(t^{-L}\x+\x_0)\cdot(z_1\nabla F_1(t^{-L}\x+\x_0)+z_2\nabla F_2(t^{-L}\x+\x_0))+\w\cdot (t^{-L}\x+\x_0)|\leq H_\G J(Z)^{1/2}\\
\Rightarrow& |P^2t^{-L}(z_1G_1(\x)+z_2G_2(\x))+t^{-L}P\v\cdot(t^{-L}\x+\x_0)/(by_N)|\leq H_\G J(Z)^{1/2}\\
\Rightarrow & |P^2(z_1G_1(\x)+z_2G_2(\x))+P\v\cdot(t^{-L}\x+\x_0)/(by_N)|\leq H_\G \hat{L} J(Z)^{1/2}\leq H_\F J(Z)^{1/2}.
\end{align*}
However, we also have $|P^2(z_1G_1(\x)+z_2G_2(\x))|\leq \max\{1,H_\G\} J(Z)^{1/2}\leq H_\F J(Z)^{1/2}$. Thus, we must have 
\begin{equation}\label{eq:no1}
|P\v\cdot(t^{-L}\x+\x_0)/(by_N)|\leq H_\F J(Z)^{1/2},\,\,\,\forall \x\in \Omega.
\end{equation}
Our findings therefore give:
\begin{equation}\label{eq:access2}
\Sigma(Z,y,\hat{B})=\sum_{\substack{
b\in \cO^\sharp
\\
(b,\scrP)=1\\
|b|=\hat B 
\\ b\equiv b_0\bmod{N}} }S_{b}(\v)
\int_{|\uz|=\hat{Z}}\int_{\Omega_1}\omega\left(\x\right)
\psi \left(z_1 P^2F_1(\x)+z_2P^2F_2(\x)\right)\psi\left(\frac{P\v\cdot\x}{by_n}\right)d\x d\uz,
\end{equation}
where 
\begin{equation*}
\Omega_1=\{\x\in K_\infty^n:|P\v\cdot\x/(by_N)|\leq H_\F J(Z)^{1/2} \}\cap \Omega',
\end{equation*}
where
\begin{equation*}
\Omega'=\left\{\x\in \TT^n: |\x-\x_0|<\hat{-L}, ~ |P^2z_1\nabla F_1(\x)+P^2z_2\nabla F_2(\x)+ P\v/(by_N)|\leq H_\F
J(Z)^{1/2}\right\}
\end{equation*}
Note that for a fixed value of $y$, the set $\{\x\in K_\infty^n:|P\v\cdot\x/(by_N)|\leq H_\F J(Z)^{1/2} \}$ only depends on the absolute value $|b|$. Let $J$ be the smallest integer such that
\begin{equation}
\label{eq:Jdef}
H_\F J(Z)^{1/2}\leq q^J\leq qH_\F J(Z)^{1/2}.
\end{equation}
If $J\leq B$, then, since $b$ is  monic, there exist
$c_1,\dots,c_K\in \FF_q$ such that 
$$
b=
\underbrace{t^B+c_1t^{B-1}+\dots+c_{J-1}t^{B-J+1}}_{=t^Ba}
+\underbrace{c_Jt^{B-J}+\dots+c_B}_{=t^{B-J}b'},
$$
where  $a\in (A/x^J A)^*$ and 
 $b'\in A$, where $x=t^{-1}$. If $B<J$, the treatment above still formally works upon choosing $c_{B+1}=...=c_{J-1} =0$ and $b'=0 $.  Since $|P\v/(by_N)|\leq H_\F J(Z)$, we have
 \begin{align*}
 \left|\frac{P\v}{t^By_N}\left(\frac{1}{a+x^Jb'}-\frac{1}{a}\right)\right|\leq H_\F J(Z)\hat{J}^{-1}\leq  J(Z)^{1/2}.
\end{align*}
Therefore, the set $\Omega'$ only depends on the value of $b/t^B\bmod{x^J}$, i.e. on $a$. Moreover, an analogous calculation shows that since $|P\v\cdot\x|\leq q^{J}|by_N|$, the value of $\psi\left(\frac{P \v.\x/y_N}{b }\right)$ also only depends on the value of $b/t^B\bmod{x^J}$. We pick up this condition by introducing Dirichlet characters modulo $x^J$. Moreover, we pick up the condition $b\equiv b_0\bmod{N}$ by introducing characters modulo $N$. 
  Letting $D_1=(\cO/N\cO)^*$, and $D_2=\FF_q[x]/x^J\FF_q[x]$, we establish the identity
 \begin{equation}\label{eq:cross}
 \begin{split}
 \Sigma(Z,y,\hat{B})
=
&\int_{|\uz|=\hat{Z}}\int_{\Omega_1}\omega\left(\x\right)
\psi \left(z_1 P^2F_1(\x)+z_2P^2F_2(\x)\right)
\\&\times\frac{1}{\#D_1\#D_2}\sum_{\eta_1\bmod{N}}
\sum_{\chi \bmod{x^J}}\sum_{a\in D_2}
\psi\left(\frac{P \v.\x/y_N}{t^Ba}\right)
\bar{\eta_1(b_0)
\chi(a)}
\Sigma_0(\eta_1,\chi,\hat{B})d\x\, d\uz,
\end{split}
\end{equation}
where
$$
\Sigma_0(\eta_1,\chi,\hat{B})
=\sum_{\substack{
b\in \cO^\sharp
\\
(b,\scrP)=1\\
|b|=\hat B\\
} }\eta_1(b)\chi(t^{-B}b)S_{b}(\v).
$$
The strategy will follow closely with that of the proof of \cite[Lemma 8.2]{Browning_Vishe15}. There are two main estimates we would need. Firstly, we would like to bound the inner sum over $a$. Note that trivially we can obtain the bound:
\begin{align}\label{eq:no3}
\frac{1}{\#D_2}\sum_{\chi \bmod{x^J}}\left|\sum_{a\in 
D_2}\bar{\chi(a)}\psi\left(\frac{P \v.\x/y_N}{t^Ba}\right)\right|\ll \hat{J}\ll qJ(Z)^{1/2}.
\end{align}
Note that \eqref{eq:no3} already hands us a saving of an extra factor of $J(Z)^{1/2}$ as compared with \cite[(8.4)]{Browning_Vishe15}. This saving is obtained from our refined bounds in Lemma \ref{lem:small}, which handed us \eqref{eq:no1}. As in \cite[Lemma 8.3]{Browning_Vishe15}, this can be further improved by utilising the sum over $a$. This will be our next focus. The argument here is almost identical to that of \cite[Lemma 8.3]{Browning_Vishe15}.
\begin{lemma} For any $\x$ satisfying $|P\v\cdot\x|\leq \hat{J+B}|y_N|$,
\label{lem:red}
\begin{align*}
\frac{1}{\#D_2}\sum_{\chi \bmod{x^J}}\left|\sum_{a\in 
D_2}\bar{\chi(a)}\psi\left(\frac{P \v.\x/y_N}{t^Ba}\right)\right|
&\leq \hat{\lceil J/2\rceil}\leq qJ(Z)^{1/4}.
\end{align*}
\end{lemma}
\begin{proof}
Let $\chi \bmod{x^J}$ be a Dirichlet character. 
Let $\ve>0$ and choose $J_0\in \ZZ$ such that $ J_0=\lceil J/2\rceil$.
Clearly, $J/2\leq J_0\leq J$.
Recall that $x=t^{-1}$ and suppose that  $a\equiv a'\bmod{x^{J_0}}$, for $a,a'\in D_2$. Then for $\x$ as in the hypothesis of this lemma,
\begin{align*}
\left| \frac{P \v\cdot\x}{t^Bay_N}-\frac{P \v\cdot\x}{t^Ba'y_N}\right|\leq
\hJ\left|\frac{a-a'}{aa'}\right|\leq \frac{\hat J}{\hat J_0}\leq \hat{J-J_0}.
\end{align*}

Let us write $a=a_0+x^{J_0}a_1$, where 
$a_0\in (A/x^{J_0}A)^*$ and $a_1\in A/x^{J-J_0}A $. 
Then 
\begin{align*}
 \sum_{a\in D_2}\chi(a)\psi\left(\frac{P \v.\x/y_N}{t^Ba}\right)
 =~&
 \sum_{a_0\in (A/x^{J_0}A)^*}\sum_{a_1\in A/x^{J-J_0}A}\chi(a_0+x^{J_0}a_1)
\psi\left(\frac{P \v.\x}{t^B(a_0+x^{J_0}a_1)y_N}\right).
\end{align*}
For fixed   $a_0\in (A/x^{J_0} A)^*$ and $\x$, we proceed to examine the sum
\begin{align*}
S(\x)= \sum_{a_1\in A/x^{J-J_0}A}\psi\left(\frac{P \v.\x}{t^B(a_0+x^{J_0}a_1)y_N}\right)\chi(1+x^{J_0}a_1\bar{a_0}),
\end{align*}
where
$\bar{a_0}$ denotes the multiplicative inverse of $a_0\bmod{x^{J-J_0}}$. As seen in the proof of \cite[Lemma 8.3]{Browning_Vishe15}, the function $\phi_\chi(a)=\chi(1+x^{J_0}a)$ must be a twist of a standard additive character 
$$\phi_\chi(a)=\psi\left(\frac{a_\chi a}{x^{J-J_0}}\right). $$

Similarly, since $|\v\cdot\x|\leq \hat{J+B}|y_N|/|P|$, 
\begin{align*}
\psi\left(\frac{P \v.\x}{t^B(a_0+x^{J_0}a_1)y_N}\right)=\psi\left(\frac{P \v.\x}{t^Ba_0(1+x^{J_0}\overline{a_0}a_1)y_N}\right)=\psi\left(\frac{P \v.\x (1-x^{J_0}\overline{a_0}a_1)}{t^Ba_0y_N}\right)=\psi\left(\frac{P \v.\x}{t^Ba_0y_N}+\frac{a_1\bar{a_0}^2a''}{x^{J-J_0}}\right),
\end{align*} where $a''$ is independent of the choices of $\chi$, $a_0$ and $a_1$. Therefore, 
\begin{align*}
S(\x)= 
\psi\left(\frac{P 
\v.\x}{t^Ba_0y_N}\right)
\sum_{a_1\in A/x^{J-J_0}A}\psi\left(\frac{a_1\bar{a_0}(a_\chi+a''\bar{a_0})}{x^{J-J_0}}\right).
\end{align*}
For a fixed $a_0$, we deduce that $S(\x)=0$ unless  
$a_\chi\equiv a'''\bmod{x^{J-J_0}} 
$, where $a'''=-a''\bar{a_0}\bmod{x^{J-J_0}}$,  in which case $|S(\x)|\leq \hat J/\hat J_0$.  
 However, for a fixed $a'''\in A/x^{J-J_0}A$ we have 
$
\#\{\chi:a_\chi\equiv a'''\bmod{x^{J-J_0}}\}\leq \hat{J_0}.
$ 
Thus
\begin{align*}
\frac{1}{\#D_2}\sum_{\chi \bmod{x^J}}\left|\sum_{a\in 
D_2}\bar{\chi(a)}\psi\left(\frac{P \v.\x/y_N}{t^Ba}\right)\right|
\leq~& \frac{1}{\hat J} 
\sum_{\chi \bmod{x^J}}
\sum_{a_0\in (A/x^{J_0}A)^*}|S_\x|
\leq~ \hat J_0.
\end{align*}
This completes the proof of the lemma.
\end{proof}
We now turn our attention to the term $\Sigma_0(\eta_1,\chi,\hat{Y}/|y|)$. Let $\eta_2:\scrO\rightarrow\CC^*$ be the Hecke character given by $\eta_2(r)=\chi(r/t^{\deg{r}}) $. We thus focus on the following twisted averages:
\begin{align}\label{eq:sun}
\sum_{\substack{
b\in \cO^\sharp
\\
(b,\scrP)=1\\
|b|=\hat B\\
} }\eta_1(b)\eta_2(b)S_{b}(\v).
\end{align}
We next replace the exponential sums $S_{b}(\v)$ by their explicit values obtained in Lemma \ref{lem:Expsum}. This will transform the sum  \eqref{eq:sun} to a character sum. Let
\begin{equation}
\alpha=\begin{cases}
(i_p^{\ell_0})^{n}& \textrm{ if } 2\mid n,\\
(i_p^{\ell_0})^{(n+1)}&\textrm{ if }2\nmid n.
\end{cases}
\end{equation}
where $q=p^{\ell_0}$ and $i_p$ as in \eqref{eq:ipdef}. Moreover, let 
\begin{equation}
\beta=\begin{cases}-1&\textrm{ if }2\mid n\\
1&\textrm{ if }\textrm{ if }2\nmid n.
\end{cases}
\end{equation}
Let us finally define a Dirichlet character $\eta_3$
\begin{equation}
\label{eq:eta3def}
\eta_3(b)=\begin{cases}\left(\frac{\det(M_\uc)}{b}\right)&\textrm{ if }2\mid n\\\left(\frac{-f^*(\v)}{b}\right)&\textrm{ if }2\nmid n.
\end{cases}
\end{equation}
Using Lemma \ref{lem:Expsum}, we get
\begin{align*}
\sum_{\substack{
b\in \cO^\sharp
\\
(b,\scrP)=1\\
|b|=\hat B\\
} }\eta_1(b)\eta_2(b)S_{b}(\v)=\alpha^{B}\sum_{\substack{
b\in \cO^\sharp
\\
(b,\scrP)=1\\
|b|=\hat Z\\
} }\beta^{\Omega(b)}\eta_1(b)\eta_2(b)\eta_3(b)\times\begin{cases}\hat{B}^{n/2}&\textrm{ if }2\mid n\\
\hat{B}^{n/2+1/2}&\textrm{ if }2\nmid n.
\end{cases}
\end{align*}
At this point, we wish to invoke Lemma \ref{lem:Charsum} to bound the character sum satisfactorily. In order to achieve the extra square-root cancellations in the $b$ sum, we need to make sure that $\eta_1(b)\eta_2(b)\eta_3(b)$ is not a character of type $|b|^{ix}$ for any $x\in\RR$. $\eta_3$ can be viewed as a Dirichlet character of order $2$, modulo $\det(M_\uc)$ if $n$ is even and modulo $-f^*(\v)$ if $n$ is odd. This is non-trivial if $\det(M_\uc)$ is not a perfect square when $n$ is even and when $-f^*(\v)$ is not a perfect square when $n$ is odd. However, this is not enough. We also need to guarantee that the character $\eta_1\eta_3$ is non-trivial. Since $\eta_3$ is a quadratic character, it is enough to make sure that for any $b'\in \cO$ satisfying $b'\mid N$, $b'\det(M_\uc)$ is not a perfect square if $n$ is even, and that $b'f^*(\v)$ is not a perfect square if $n$ is odd. This is due to the fact that these conditions would guarantee that $f^*(\v)$ (or $\det(M_\uc)$) will contain an odd power of a prime not dividing $N$. We now apply Lemma \ref{lem:Charsum} to obtain the desired square-root cancellations in the character sum \eqref{eq:sun}:

\begin{lemma}
\label{lem:good1}
For any good pair $\uc$, as long as
\begin{align}
\label{eq:case2}&\forall b'\mid N,\begin{cases}
b'\det(M_\uc)\textrm{ is not a perfect square} & \textrm{ if }2\mid n,\\
b'f^*(\v) \textrm{ is not a perfect square} & \textrm{ if }2\nmid n,
\end{cases}
\end{align}
we have that given any $\ve>0$, any $B\in\NN$ we have:
\begin{equation}
\label{eq:sigma0bound}
\Sigma_0(\eta_1,\chi;\hat{B})\ll_{n,\|\F\|}|P|^{\ve}\begin{cases}\hat{B}^{n/2+1/2}&\textrm{ if }2\mid n,\\
\hat{B}^{n/2+1}&\textrm{ if }2\nmid n.\end{cases}
\end{equation}
\end{lemma}
We will summarize our findings into the following lemma:

\begin{lemma}\label{lem:SigB}Let $\uc$ be a good primitive pair in $\scrO^2$, let $\ve>0$, and $\scrP$ be a set of primes dividing $d$, $f^*(\v),\det(M_\uc)$ and $D_\F$ when $2\mid n$ and a set of primes dividing $d$, $\det(M_\uc)$ and $D_\F$ when $2\nmid n$. If \eqref{eq:case2} is true then we have
\begin{align*}
|\Sigma(Z,y,\hat{B})|\ll_{q,\F} J(Z)^{-n/2+5/4}(\log|P|)\hat{Z}\min\{\hat{Z}, |P|^{-2}\}\hat{B}^{(n+1)/2}\begin{cases}1&\textrm{ if }2\mid n\\
\hat{B}^{1/2}&\textrm{ if }2\nmid n.
\end{cases}
\end{align*}
Combining this bound with the weaker one in Lemma \ref{lem:sqfreeeasy}, we get that for any $0\leq \gamma\leq 1/2$, we must have
\begin{align*}
|\Sigma(Z,y,\hat{B})|\ll_{q,\F} J(Z)^{-n/2+5/4-\gamma/2}(\log|P|)\hat{Z}\min\{\hat{Z}, |P|^{-2}\}\hat{B}^{(n+1)/2+\gamma}\begin{cases}1&\textrm{ if }2\mid n\\
\hat{B}^{1/2}&\textrm{ if }2\nmid n.
\end{cases}
\end{align*}
\end{lemma}

\section{Minor arcs bound and  proof of Theorem \ref{thm:count}}
\label{sec:minor}
We continue our analysis from the last section. In the light of our results in Section \ref{sec:Prep}, Theorem \ref{thm:count} will be established upon proving that the minor arcs contribution $|E_i(P))|\ll |P|^{n-4-\ve}$, for $i=1,2$. This will be our main focus here. Our treatment in the $n$ odd and even cases will be slightly different, due to the nature of our exponential sum bounds. $n=9$ will be the hardest case for us, $n\geq 10$ being relatively easier, aided by the fact that Lemma \ref{lem:sqfreeeasy} will be enough for these. In many cases, the bounds for the $n=9$ case will subsume those for the even $n$'s. Therefore, here we shall mostly concentrate on the $2\nmid n$ case.  In each case, we will deal with the contributions from the good 
 and bad pairs $\uc$'s separately.  
 
 Throughout this section, we will assume that $q$ is fixed, and our constants may implicitly depend on it. We recall that $\scrF^*(\v)$ denotes the dual variety of the complete intersection of $F_1$ and $F_2$. 
\subsection{Good $\uc$ contribution: $n$ odd case}
Recall the definition of $E_1(d\uc,Y,Z)$ from \eqref{eq:Edcdef}. When $n$ is odd and when $\uc$ is good, we will split the sum over $\v$ in $E_1(d\uc,Y,Z)$ into two subsums:
\begin{align*}
\sum_{\substack{ \v\in \cO^n\\
\v\neq \vecnull\\
|\v|\ll \hat Y |P|^{-1}J(Z)
}} =\sum_{\substack{ \v\in \cO^n\\
\scrF^*(\v)\neq 0\\
|\v|\ll \hat Y |P|^{-1}J(Z)
}}+\sum_{\substack{ \v\neq\vecnull\in \cO^n\\
\scrF^*(\v)=0\\
|\v|\ll \hat Y |P|^{-1}J(Z)
}}.
\end{align*}
We call the corresponding contributions $E_{1,1}$ and $E_{1,2}$ respectively. The reason behind doing so is that we can obtain square-root cancellations in Lemma \ref{lem:good1} as long as $b'f^*(\v)$ is not a perfect square for any $b'\mid N$. For a fixed value of $\v$ satisfying $\scrF^*(\v)\neq 0$, we are able to employ Lemma \ref{lem:Bro} to bound the number of $\uc$'s satisfying $b'f^*(\v)=y^2$, for some $y\in\scrO$ and $b'\in \scrO$. The condition that $\scrF^*(\v)\neq 0$ is crucial here as it would imply that $b'f^*(\v)$ is a square-free polynomial in $\uc$. On the other hand, when $\scrF^*(\v)=0$, we gain by sparseness of such $\v$'s using a Serre type bound. 

We now turn to our main optimisation process. First and foremost, we write $r=br_1$, where $b$ denotes the square-free part of $r$ which is co-prime to $d\det(M_\uc)D_\F$ if $\uc$ is good and is co-prime to $Q_\uc^*(\v)dD_\F$ if $\uc$ is bad. Due to our separate bounds for good, type II and bad primes, we will write 
$$r_1=r_2r_3r_4,$$
 where $r_j$'s are all pairwise co-prime. $r_2r_3$ is free of any fifth power and furthermore, $\gcd(r_2,d\det(M_\uc)D_\F)=1$ and $r_3$ is a $5$-free number satisfying $r_3\mid \det(M_\uc)^\infty $ according to our notation \eqref{eq:a||b}, but $\gcd(r_3,dD_\F)=1$, i.e., $r_3$ consists of type II primes which are co-prime to $dD_\F$.  Lastly, $r_4$ consist of the rest, i.e., it is composed of the primes dividing $dD_\F$ and all $5$-full numbers.
We will now split our sum into the following dyadic sums 
\begin{equation}\label{eq:udsum}
\ud:=(|d|,|\uc|,|r_2|,|r_3|,|r_4|)=\d:=(\hat{D},\hat{C},\hat{R_2},\hat{R_3},\hat{R_4}),
\end{equation} where $|b|=\hat{B}$, such that, $B+R_2+R_3+R_4=Y$ and $2(D+C)\leq Y\leq D+C+Q/2$, with an extra condition that $\uc$ is good. We also define $R_1=R_2+R_3+R_4$.

As noted in the previous section, since $0\leq Y\leq Q$, and $Z\in\ZZ$ such that $-5\log_q |P|\leq Z<-Y-Q/2$, there are only $O(|P|^\ve)$ different choices for vectors $\d$. Therefore, it will be enough to focus on the contribution from $\ud=\d$ to $E_1$ and $E_2$ for any arbitrary, permissible choice of $\d$.

 Let $E_{1,1}(\d,Y,Z)$ denote the contribution to $E_{1,1}$ by the sum over $\ud=\d$. Throughout, we will adopt the notation $f^*(\v)\neq\msquare$  to denote that $b'f^*(\v)\neq y^2$ for any $y\in \scrO$ and any $b'\mid N$. We may analogously define $f^*(\v)=\msquare$. When $f^*(\v)\neq\msquare$, we apply Lemma \ref{lem:SigB} with $\gamma=1/5$,  and when $f^*(\v)=\msquare$, Lemma \ref{lem:sqfreeeasy} to \eqref{eq:E(P)-1}, to obtain that there exist $\b_1\in\scrO^n $ and $N_1\mid N$ such that
\begin{equation*}
\begin{split}
E_{1,1}(\d,Y,Z)
\ll~& \frac{|P|^n}{\hat Y^{n}}
\sum_{\substack{ \v\in \cO^n\\
\scrF^*(\v)\neq 0\\
|\v|\ll \hat{V}
}} 
\sum\limits_{\substack{\ud=\d\\\uc \textrm{ good }}}
|S_{d\uc,r_1,\b_1,N_1}(\v)|J(Z)^{-n/2+1}\hat{B}^{n/2+1}\times\\
&\left(J(Z)^{1/4-1/10}\hat{B}^{1/5}\delta_{f^*(\v)\neq\msquare}+\hat{B}^{1/2}\delta_{f^*(\v)=\msquare}\right)
 \min\{\hat{Z_1+Z_2}, |P|^{-2}\log|P|\max\{\hat{Z_1},\hat{Z_2}\}\},
\end{split}
\end{equation*} 
Let $\eone^1$ denote the contribution from $f^*(\v)\neq\msquare$ and $\eone^2$ from $f^*(\v)=\msquare$. Thus, 
\begin{equation}\label{eq:E11Bound}
\begin{split}
E_{1,1}^1(\d,Y,Z)
\ll~& \frac{|P|^n}{\hat Y^{n}}J(Z)^{-n/2+5/4-1/10} \hat{Z}\min\{\hat{Z}, |P|^{-2}\}\sum_{\substack{ \v\in \cO^n\\
\scrF^*(\v)\neq 0\\
|\v|\ll \hat{V}
}} 
\sum_{\substack{\ud=\d\\f^*(\v)\neq\msquare\\ \uc \textrm{ good }}}
|S_{d\uc,r_1,\b_1,N_1}(\v)|\hat{B}^{n/2+6/5},
\end{split}
\end{equation} 
where $\hat{V}=\hat Y |P|^{-1}J(Z)$.

Our choice of the decomposition of $r_1$ arises from different bounds in Section \ref{sec:expsum}. Lemma \ref{lem:Expsum} provides a satisfactory bound for the exponential sums modulo $r_2$. Lemma \ref{lem:type II} bounds the sums modulo $r_3$.  Lastly, for a fixed $d$, the number of permissible $r_4$ is at most $O(\hat{R_4}^{1/5})$. We make our bounds in Lemma \ref{lem:ExpsumsComb} work for the sums modulo $r_4$. More explicitly, we write
$$|S_{d\uc,r_1,\b_1,N_1}(\v)|=|S_{r_2}(\v)S_{r_3}(\v)S_{d\uc,r_4,\b_2,N}(\v)|,$$
and obtain
\begin{equation}
\label{eq:Splitbounds}
\begin{split}
|S_{r_2}(\v)|&\ll \hat{R_2}^{(n+1)/2}\gcd(r_2,f^*(\v))^{1/2},\\
|S_{r_3}(\v)|&\ll \hat{R_3}^{n/2+1}\gcd(r_3,((S^{-1})^t\v)_n, Q_\uc^*(\v'))^{1/2}.\\
|S_{d\c,r_4,\b_2,N}(\v)|&\ll \hat{D}\hat{R_4}^{n/2+1}\gcd(r_4/d,\det(M_\uc),((S^{-1})^t\v)_n))^{1/2}.
\end{split}
\end{equation}
We now arrange the various sums in the following order and evaluate them using our previous bounds:
\begin{align*}
\sum_{d}\sum_{r_4}\sum_{\uc}\sum_{r_3}\sum_\v\sum_{r_2}.
\end{align*}
Note that $r_2$ only consists of square-full numbers and the condition $f^*(\v)\neq \msquare$ guarantees that $f^*(\v)\neq 0$. Therefore, for a fixed value of $\v$, we must have
\begin{equation}\label{eq:r2sum}
\sum_{|r_2|=\hat{R_2}}|S_{r_2}(\v)|\ll \hat{R_2}^{n/2+1+\ve}.
\end{equation}
Combining our bounds, we obtain the following result:
\begin{align}\notag
&\sum_{\substack{ \v\in \cO^n\\
\scrF^*(\v)\neq 0\\
|\v|\ll \hat{V}
}} 
\sum_{\substack{\ud=\d\\f^*(\v)\neq\msquare\\ \uc \textrm{ good }}}|S_{d\uc,r_1,\b_1,N_1}(\v)|
\ll |P|^\ve\widehat{R_1}^{n/2+1}\hat{D} \\ &\sum_{d}\sum_{r_4}\sum_{\uc}\sum_{r_3}\sum_{\v}\gcd(r_4/d,\det(M_\uc),((S^{-1})^t\v)_n)^{1/2}\gcd(r_3,((S^{-1})^t\v)_n,Q_\uc^*(\v'))^{1/2}\notag\\
&\ll |P|^\ve\widehat{R_1}^{n/2+1}\hat{D}\sum_{|d|=\hat{D}}\sum_{|r_4|=\hat{R_4}}\sum_{x_1\mid r_4/d}\sum_{\substack{|\uc|=\hat{C}\\\uc \textrm{ good }\\x_1\mid \det(M_\uc)}}\sum_{x_2\mid r_3}\sum_{\substack{0\neq |\v|\leq \hat{V}\\\scrF^*(\v)\neq 0\\ f^*(\v)\neq \msquare\\ x_1x_2\mid ((S^{-1})^t\v)_n\\ x_2\mid  Q_\uc^*(\v')}}|x_1x_2|^{1/2}.
\label{eq:sumfinal1}
\end{align}
Before we start our final computation, we will need an estimate for $$\#\{\uc:x\mid \det(M_{\uc})\}\textrm{  and  }\#\{|\v|\leq \hat{V}:x\mid ((S^{-1})^t\v)_n,y\mid Q^*(\v')\}.$$ Here $\v'$ denotes the vector obtained from the first $n-1$ entries of $(S^{-1})^t\v$. This will be our next goal.
\begin{lemma}
\label{lem:VCsum}
Given any $n\geq 2$, any $x,y\in\scrO$ such that $y\mid x$, any primitive $\uc\in\scrO^2$,  any $V\in \NN_{\geq 0}$, we have
\begin{equation}
\label{eq:Vsumbound}
\#\{|\v|\leq\hat{V}:x\mid ((S^{-1})^t\v)_n,y\mid Q_\uc^*(\v')\}\ll (\hat{V})^{n-2}\min\left\{\hat{V}\left(1+\frac{\hat{V}}{|x|}\right),\left(1+\frac{\hat{V}}{\prod_{\vp\mid y}|\vp|}\right)\left(1+\frac{\hat{V}}{\prod_{\vp\mid x}|\vp|}\right)\right\}.
\end{equation}  
Similarly, given any $n\geq 2$, $x\in\scrO$, any $\uc$, any $\ve>0$ and any $C\in \NN_{\geq 0}$ we have
\begin{equation}
\label{eq:csumbound}
\#\{\uc\in\scrO^2 \textrm{ {\em primitive} }: |\uc|< \hat{C},x\mid \det(M_{\uc})\}\ll_{\Delta_\F,\ve} |x|^\ve\hat{C}\left(1+\frac{\hat{C}}{|x|^{1/2}}\right).
\end{equation}
\end{lemma}
\begin{proof}
We start by proving \eqref{eq:Vsumbound}. Let $(S^{-1})^t=(s_{i,j})_{1\leq i,j\leq n}$. For each $i\geq 1$, let $$x_i=\gcd(x,s_{n,1},...,s_{n,i-1})/\gcd(x,s_{n,1},...,s_{n,i}).$$ Here, by convention, $s_{n,0}=1$. Since $\det(S)\in \FF_q^\times$, each row and column of $S^{-1}$ should be primitive. Thus, $x_n=\gcd(x,s_{n,1},...,s_{n,n-1})$.  If $x\mid s_{n,1}v_1+...+s_{n,n}v_n$, then note that we must have $x_n\mid v_n$. For a fixed choice of such $v_n$, if $(v_1,...,v_{n-1},v_n)$ and $(v_1',...,v_{n-1}',v_n) $ both satisfy $x\mid ((S^{-1})^t\v)_n$, then we must have $x\mid \sum_{i=1}^{n-1}s_{n,i}(v_i-v_i')$. This forces that $x_{n-1}\mid v_{n-1}-v_{n-1}'$. Continuing inductively, for a fixed choice of $v_j,...,v_n$, if two vectors $(v_1,...,v_n)$ and $(v_1',...,v_{j-1}',v_j,...,v_n)$ are both solutions of $x\mid ((S^{-1})^t\v)_n$, then this must imply $x_{j-1}\mid v_{j-1}-v'_{j-1}$. Thus, the quantity in \eqref{eq:Vsumbound} is $\ll \prod_{i=1}^n (1+\hat{V}/|x_i|).$ This is clearly enough to obtain 
\begin{equation}
\label{eq:Vsumbound1}
\#\{|\v|\leq\hat{V}:x\mid ((S^{-1})^t\v)_n\}\ll (\hat{V})^{n-1}\left(1+\frac{\hat{V}}{|x|}\right),
\end{equation} since $|x|=|x_1...x_n|$. 

The second bound in \eqref{eq:Vsumbound} is obtained by realising this as a counting problem modulo primes. Let $\scrV_1$ be the variety defined by $s_{n,1}v_1+...+s_{n,n}v_n=0$, of affine dimension $n-1$ and let $\scrV_2$ denote the complete intersection of $s_{n,1}v_1+...+s_{n,n}v_n=Q_\uc^*(\v')=0$, an variety of affine dimension $n-2$. We may clearly bound the left hand side in \eqref{eq:Vsumbound} by 
\begin{equation*}
\#\{|\v|\leq\hat{V}:(\v\bmod{\vp})\in \scrV_1, \forall ~ \vp\mid x, (\v\bmod{\vp})\in \scrV_2, \forall ~ \vp\mid y\}.
\end{equation*}   The second bound on the right hand side of \eqref{eq:Vsumbound} is then an easy consequence of \cite[Lemma 4]{Browning-Heath-Brown09}, which holds in the function field setting analogously, since it only uses bounds for number of points on varieties over finite fields.

We now focus on obtaining \eqref{eq:csumbound}. For any decomposition $x=x_1x_2$, where $\gcd(x_1,x_2)=1$, let $$C_\ux=\{|\uc|< \hat{C}:\gcd(x_1,c_1)=\gcd(x_2,c_2)=1\}.$$
From now on we fix a decomposition $x=x_1x_2$ as above, we will establish the bound $$\#\{\uc\in C_\ux,x\mid \det(M_\uc)\}\ll (\log|x|)^{n-1}\hat{C}\left(1+\frac{\hat{C}}{\max\{|x_1|,|x_2|\}}\right),$$
This bound will clearly suffice for us. Without loss of generality, let us assume that $|x_2|\leq |x_1|$. For a fixed value of $c_1$, we will bound
$$\#\{|c_2|<\hat{C}:x_1\mid g(\bar{c_1}c_2)\}, $$
where $g(T)=\det(-TM_1+M_2)$, a polynomial of degree at most $n$. Here, $\bar{c_1}$ denotes a multiplicative inverse of $c_1$ modulo $x_1$. If $\vp\mid x_1$ is not a bad prime, then we know that $\vp$ does not divide the discriminant of the polynomial $g(T)$, and therefore $g(T)$ does not have multiple roots modulo $\vp$. Thus the number of roots of $g(T)$ modulo $\vp$ is at most $n$ and each root is necessarily simple. Hensel's lemma then implies that there are at most $n$ roots of $g(T)$ modulo $\vp^k$ for any $k$ and that each root is simple. On the other hand, if $\vp$ is a bad prime, then $\vp$ is bounded. Moreover, we know that 
$$g(T)=\prod_{i=1}^{n-1} (T-\gamma_i),$$
where $\gamma_i$ are distinct elements of $\bar{K_\vp}$. Let $\gamma_1,...,\gamma_i\in \scrO_\vp$ and $\gamma_{i+1},...,\gamma_{n-1}\notin \scrO_\vp$. Since $\scrO_\vp$ is compact, clearly $\sup_{T\in\scrO_\vp}|\prod_{j=i+1}^{n-1}(T-\gamma_j)|_\vp\gg 1$, (since we can't have a sequence of elements of $\scrO_\vp$ converging to $\gamma_j$ for any $i<j$). Thus, $|g(T)|_\vp\gg \prod_{j=1}^i|T-\gamma_j|_\vp$. Moreover, $\gamma_1,...,\gamma_i$ are all distinct elements of $\scrO_\vp $ and therefore are sufficiently separated from one another. Thus, $\vp^k\mid g(T)$ must necessarily imply that $T\equiv \gamma_j\bmod{\vp^{k-k_0}}$ for some $k_0\ll_\vp 1$ and $1\leq j\leq i$. We have thus proved that for any $\vp^k|| x_1$, the equation $\vp^k\mid g(\bar{c_1}c_2)$ must imply that $\bar{c_1}c_2\bmod{\vp^{k-k(\vp)}}$ has at most $n$ distinct choices. Here, $k(\vp)=0$ if $\vp$ is not bad, and $k(\vp)\ll_{\vp} 1$ for when $\vp$ is bad. Thus, $\bar{c_1}c_2\mod{x_1}$ has at most $O(n^{\log(|x|)/\log\log(|x|)})=O_\ve(|x|^\ve)$ different choices modulo $x_1$. This leads to \eqref{eq:csumbound}.
\end{proof}
\subsubsection{Final optimisation for good $\uc$'s}\label{subsec:ucgood}
We are now set to establish the contribution of all good $\uc$'s to $ E_1$. Let us give an overview of how the optimisation process will work. Note that, for a fixed $d$, the number of $r_4=O(\hat{R_4}^{1/5})$ and for a fixed $\uc$, there are only $O(|P|^\ve)$ choices for $r_3$.
We may trivially bound \eqref{eq:sumfinal1} by
\begin{equation}\label{eq:trivialB}
|P|^\ve \widehat{2C+2D}\hat{R_1}^{n/2+3/2}\hat{R_4}^{1/5}\hat{D}^{-1/2}.
\end{equation}
This bound is only enough to obtain $n\geq 11$ unfortunately. 

Let us get back to \eqref{eq:sumfinal1}. A critical case for us is when $Y\asymp Q,D+C\asymp Q/2$. In this case, $V\asymp Q/4$. In the worst case, $|x_1x_2|^{1/2}\asymp\hat{Q}^{1/2}$. When $C$ and $|x_1|$ are large, we may simultaneously save from the sum over $\uc$ by utilising the condition $x_1\mid \det(M_\uc)$ in conjunction with the linear constraint $x_1x_2\mid ((S^{-1})^t\v)_n$ by applying \eqref{eq:csumbound} and \eqref{eq:Vsumbound1} respectively. When $C$ is relatively large, but $x_1$ is small, we need to resort to the second bound in \eqref{eq:Vsumbound}. Note that $x_2$ is free of any fifth power, so the factor $\prod_{\vp\mid x_2}\vp$ is of size at least $|x_2|^{1/4}$, making the second bound in \eqref{eq:Vsumbound} powerful here. When $C$ is very small, the saving by the factor $\hat{D}^{1/2}\asymp \hat{Q/4}$ appearing on the right hand side of \eqref{eq:trivialB} and the saving of size $\hat{V}$ from the linear constraint $x_1x_2\mid ((S^{-1})^t\v)_n$ together are enough. Lastly, when $f^*(\v)=\msquare$ and $\scrF^*(\v)\neq 0$, we may employ our counting estimates in Lemmas \ref{lem:B2} and \ref{lem:Bro}, the former being more useful when $C$ is small. 

Let us start with bounding $\eone$. We first apply the estimate \eqref{eq:Vsumbound} to the inner sum over $\v$ in \eqref{eq:sumfinal1}, along with the observation that $x_2$ appearing there is free of any fifth power and therefore $|\prod_{\vp\mid x_2}\vp|\leq |x_2|^{1/4}$ to obtain
\begin{align*}
&\#\{ |\v|\leq \hat{V}: x_1x_2\mid ((S^{-1})^t\v)_n, x_2\mid  Q_\uc^*(\v')\}\ll \min\{\hat{V}^{n-1}+\hat{V}^n/|x_1x_2|,\hat{V}^{n-2}+\hat{V}^n/|x_2|^{1/2}\}\\&\leq \hat{V}^{n-2}+\hat{V}^n/|x_1x_2|^{1/2}+\min\{\hat{V}^{n-1},\hat{V}^n/|x_2|^{1/2}\}
=\hat{V}^{n-2}+\hat{V}^n/|x_1x_2|^{1/2}+\hat{V}^n/|x_2|^{1/2}\min\{1,|x_2|^{1/2}/\hat{V}\}.
\end{align*}
Note that in principle, these bounds only work when $V\geq 0$. However, since we are summing over $\v\neq 0$, we may assume their validity for all $V\in \RR$. Next, we apply \eqref{eq:csumbound} to obtain
\begin{align}\label{eq:chelp}
\#\{|\uc|=\hat{C}:x_1\mid \det(M_\uc)\}\ll |x_1|^\ve \hat{C}(1+\hat{C}/|x_1|^{1/2}).
\end{align}
Applying these bounds to \eqref{eq:sumfinal1}, and as before noting that for a fixed $\uc$ there are only $O(|P|^\ve)$ choices for $r_3$, and for a fixed $d$, only $O(\hat{R_4}^{1/5})$ choices for $R_4$, we get
\begin{align}\notag
&\sum_{\substack{ \v\in \cO^n\\
\scrF^*(\v)\neq 0\\
|\v|\ll \hat{V}
}} 
\sum_{\substack{\ud=\d\\f^*(\v)\neq\msquare\\ \uc \textrm{ good }}}|S_{d\uc,r_1,\b_1,N_1}(\v)|
\\ 
\notag&\ll \widehat{R_1}^{n/2+1+\ve}\hat{D} \sum_{d}\sum_{r_4}\sum_{x_1\mid (r_4/d)}\sum_{\uc}\delta_{x_1\mid \det(M_\uc)}\sum_{r_3}\sum_{\substack{x_2\mid r_3}}|x_1x_2|^{1/2}\times\\&\left(\hat{V}^n/|x_1x_2|^{1/2}+\hat{V}^{n-2}+\hat{V}^n|x_2|^{-1/2}\min\{1,|x_2|^{1/2}/\hat{V}\}\right)\label{eq:Final3}\\
&\ll \hat{R_1}^{n/2+1+\ve}\hat{R_4}^{1/5}\hat{2D+2C}\hat{V}^n+\hat{R_1}^{n/2+1+\ve}\hat{R_4}^{1/5}\hat{2D+2C}(\hat{R_1}/\hat{D})^{1/2}\hat{V}^{n-2}\label{eq:Final1}\\
&+ \hat{R_1}^{n/2+1+\ve}\hat{R_4}^{1/5}\hat{2D+C}(\hat{R_1}/\hat{D})^{1/2}\hat{V}^{n-1}.\label{eq:Final2}
\end{align}
The bounds in \eqref{eq:Final1} are obtained from the first two terms in \eqref{eq:Final3}. It is not important to save from the sum over $\uc$ in these bounds. Therefore, we will sum over $\uc$ trivially here. While dealing with the third term in \eqref{eq:Final3}, we substitute our bound \eqref{eq:chelp}. The second term in \eqref{eq:chelp} hands us back the first bound in \eqref{eq:Final1} and the remaining term in \eqref{eq:chelp} hands us \eqref{eq:Final2}.

We are finally ready to analyse the term $E_{1,1}^1$. Inserting the bounds in \eqref{eq:Final1} and \eqref{eq:Final2} to \eqref{eq:E11Bound} we get
\begin{equation}\label{eq:E11Bound2}
\begin{split}
E_{1,1}^1&\ll |P|^{n+\ve}\hat{Y}^{6/5-n/2}(1+|P|^2\hat{Z})^{-n/2+5/4-1/10}\hat{Z}\min(\hat{Z},|P|^{-2})\\ &\hat{2D+C}\left(\hat{C}\hat{V}^n+\hat{C}(\hat{R_1}/\hat{D})^{1/2}\hat{V}^{n-2}+\hat{V}^{n}+(\hat{R_1}/\hat{D})^{1/2}\hat{V}^{n-1}\right).
\end{split}
\end{equation}
We will bound the different terms on the right hand side of \eqref{eq:E11Bound2} separately. Let us start with the term $\hat{C}\hat{V}^n$.  This term corresponds to obtaining perfect square root cancellations. Clearly, this term is at its maximum when $C+D=Y/2$. The total contribution is then
\begin{align*}
&\ll |P|^{n+\ve}\hat{Y}^{11/5-n/2}(1+|P|^2\hat{Z})^{-n/2+5/4-1/10}\hat{Z}\min(\hat{Z},|P|^{-2})\hat{V}^n\\
&\ll |P|^\ve\hat{Y}^{11/5+n/2}(1+|P|^2\hat{Z})^{n/2+5/4-1/10}\hat{Z}\min(\hat{Z},|P|^{-2}). 
\end{align*}
This expression is maximum when $\hat{Z}=\hat{-Y}|P|^{-2/3}=|P|^{-2}(\hat{Q}/\hat{Y})\geq |P|^{-2}$. Thus, $P^2\hat{Z}=|P|^{4/3}/\hat{Y}=\hat{Q}/\hat{Y}$. We thus have that this term is 
\begin{align*}
&\ll  |P|^\ve\hat{Y}^{11/5+n/2}(\hat{Q}/\hat{Y})^{n/2+5/4-1/10}|P|^{-4}\hat{Q}/\hat{Y}\ll |P|^\ve\hat{Y}^{1/20}|P|^{-4}\hat{Q}^{n/2+9/4-1/10}\\
&\ll |P|^\ve|P|^{-4}\hat{Q}^{n/2+9/4-1/20}=|P|^{n-4}|P|^{-n/3+3-1/15+\ve}.
\end{align*}
 This is enough as long as $n\geq 9$ and $\ve\leq 1/30$

We now move to the $\hat{C}(\hat{R_1}/\hat{D})^{1/2}\hat{V}^{n-2}$ term. This term is maximum when $Y=R_1$, $C=Y/2$ and $D=0$. Thus, the total contribution is
\begin{align*}
&\ll |P|^{n+\ve}\hat{Y}^{27/10-n/2}(1+|P|^2\hat{Z})^{-n/2+5/4-1/10}\hat{Z}\min(\hat{Z},|P|^{-2})\hat{V}^{n-2}\\
&\ll |P|^{2+\ve}\hat{Y}^{7/10+n/2}(1+|P|^2\hat{Z})^{n/2-17/20}\hat{Z}\min(\hat{Z},|P|^{-2}). 
\end{align*}
The maximum is again achieved when $\hat{Z}=(\hat{Q}/\hat{Y})|P|^{-2}\geq |P|^{-2} $. Thus, this contribution is
\begin{align*}
\ll |P|^{-2+\ve}\hat{Y}^{n/2+7/10}(\hat{Q}/\hat{Y})^{n/2+3/20}\ll |P|^{-2+\ve}\hat{Q}^{n/2+7/10}\ll |P|^{n-4}|P|^{-n/3+44/15+\ve}.
\end{align*}
which is enough when $n\geq 9$ and $\ve\leq 1/30$.

Now we move on to the last term in \eqref{eq:E11Bound2}. The maximum value is taken when $R_1=Y,D=Y/2,C=0$. Thus, this contribution is
\begin{align*}
&\ll |P|^{n+\ve}\hat{Y}^{49/20-n/2}(1+|P|^2\hat{Z})^{-n/2+5/4-1/10}\hat{Z}\min(\hat{Z},|P|^{-2})\hat{V}^{n-1}\\&\ll |P|^{1+\ve}\hat{Y}^{n/2+29/20}(1+|P|^2\hat{Z})^{n/2+3/20}\hat{Z}\min(\hat{Z},|P|^{-2}).
\end{align*}
The maximum is again achieved when $\hat{Y}=\hat{Q},\hat{Z}=\hat{-Y-Q/2}=|P|^{-2}(\hat{Q}/\hat{Y})$. Thus, this contribution is 
\begin{align*}
\ll |P|^{-3+\ve}\hat{Q}^{n/2+29/20}\ll |P|^{n-4}|P|^{-n/3+44/15+\ve}\ll |P|^{n-4-1/15+\ve},
\end{align*}
for all $n\geq 9$. We thus effectively bound all contributions for $\eone^1$, as long as, $n\geq 9$ and $\ve\leq 1/30$.

We now consider the term $\eone^2$ which corresponds to the validity of the conditions $f^*(\v)=\msquare$ and $\scrF^*(\v)\neq 0$. As noted in Sec. \ref{sec:background}, $\scrF^*(\v)$ is the discriminant of the polynomial $f^*(\v)$, seen as a polynomial in $\uc$. Thus, this would imply that $f^*(\v)$ has distinct roots  in $\PP^1_{\bar{K}}$, and therefore, this polynomial is necessarily square-free. We may now apply Lemma \ref{lem:Bro} to count the number of $\uc$'s for which $f^*(\v)=\msquare$. This bound will be effective when $C$ is large. Alternatively, for a fixed good $\uc$, $ f^*(\v)$ is a smooth quadratic form and therefore we may be able to bound the number of possible choices of $\v$'s for which  $f^*(\v)=\msquare$, using Lemma \ref{lem:B1}. We summarize these bounds into:
\begin{align*}
\#\{|\uc|\leq\hat{C},|\v|\leq \hat{V}: \uc\textrm{ primitive, }f^*(\v)=\msquare\}\ll \hat{V+C}^\ve\min\{\hat{V}^n\hat{C}, \hat{C}^2\hat{V}^{n-1}\}.
\end{align*}  Recall that \eqref{eq:E11Bound} hands us:
\begin{equation*}
\begin{split}
E_{1,1}^2(\d,Y,Z)
\ll &\frac{|P|^{n+\ve}}{\hat Y^{n}}
J(Z)^{-n/2+1}\hat{B}^{n/2+3/2}
\hat{Z} \min\{\hat{Z}, |P|^{-2}\}\sum_{\ud=\d}\sum_{\substack{ \v\in \cO^n\\
\scrF^*(\v)\neq 0\\
|\v|\ll \hat Y |P|^{-1}J(Z)
\\ f^*(\v)=\msquare}} |S_{d\uc,r_1,\b_1,N_1}(\v)|.
 \end{split}
\end{equation*} 
In this case, the extra saving obtained from the condition $f^*(\v)=\msquare$ will be enough. Using \eqref{eq:baddd}, we will use a weaker bound $|S_{d\uc,r_3r_4,\b_1,N_1}(\v)|\ll \hat{D}^{1/2}\hat{R_3R_4}^{n/2+3/2}$ to bound the sums modulo $r_3r_4$. This simplifies our process and we order our sums the following way:
\begin{equation}\label{eq:way}
\sum_d \sum_{\uc,\v}\sum_{r_2,r_3,r_4}.
\end{equation}
For a fixed value of $d,\uc$ and $\v$, there are at most $O(|P|^\ve)$ different choices for $r_3$ and $r_4$. Moreover, our bound \eqref{eq:r2sum} dealing with the exponential sums modulo $r_2$ still holds. Thus,
\begin{equation*}
\begin{split}
E_{1,1}^2\ll &\frac{|P|^{n+\ve}}{\hat Y^{n}}
(1+|P|^2\hat{Z})^{-n/2+1}\hat{Y}^{n/2+3/2}
 \hat{Z}\min\{\hat{Z}, |P|^{-2}\hat{Z}\}\hat{C}\hat{D}^{3/2}\hat{V}^{n-1}\min\{\hat{C},\hat{V}\}\\
 &\ll \frac{|P|^{n+\ve}}{\hat Y^{n}}
(1+|P|^2\hat{Z})^{-n/2+1}\hat{Y}^{n/2+3/2}
 \hat{Z}\min\{\hat{Z}, |P|^{-2}\hat{Z}\}\hat{(C+D)}^{3/2}\hat{V}^{n-1/2}\\
  &\ll |P|^{1/2+\ve}
(1+|P|^2\hat{Z})^{n/2+1/2}\hat{Y}^{n/2+1}
 \hat{Z}\min\{\hat{Z}, |P|^{-2}\hat{Z}\}\hat{(C+D)}^{3/2}.
 \end{split}
\end{equation*} 
Again, the maximum is achieved when $Z=-Y-Q/2$, and when $C+D=Y/2$. Thus, this contribution is
\begin{align*}
 &\ll |P|^{-3/2+\ve}\hat{Y}^{n/2+7/4}(\hat{Q}/\hat{Y})^{n/2+1/2}\hat{-Y-Q/2}.
\end{align*}
After comparing the powers of $\hat{Y}$, the above expression is maximum, when $Y=Q$ and therefore the contribution is 
\begin{align*}
\ll |P|^{-3/2+\ve}\hat{Q}^{n/2+1/4}\ll |P|^{-3/2+2n/3+1/3+\ve}\ll |P|^{n-4+\ve-(2n-17)/6} \ll |P|^{n-4+\ve-1/6},
\end{align*}
as long as $n\geq 9$ and $\ve\leq 1/12$.

Next, let us deal with the term $\etwo$. The main saving will be obtained here from a Serre type bound \cite[Lemma 2.9]{Browning_Vishe15}, which gives us:
$$\#\{|\v|\leq \hat{V}:\scrF^*(\v)=0\}\ll \hat{V}^{n-3/2}.$$
Our strategy will emulate closely that of bounding $\eone^2$. We again use the decomposition $r=br_2r_3r_4$ as before and use the bound in \eqref{eq:baddd} to bound the sums modulo $r_3r_4$, and use \eqref{eq:r2sum} to bound the averages modulo $r_2$. We also arrange the sums in a simplified way as in \eqref{eq:way}, to get:
\begin{equation}\label{eq:M}
\begin{split}
\etwo(\d,Y,Z)
\ll &\frac{|P|^{n+\ve}}{\hat Y^{n}}
(1+|P|^2\hat{Z})^{-n/2+1}\hat{Y}^{n/2+3/2}\hat{C}^2\hat{D}^{3/2}\hat{Z}\min\{\hat{Z},|P|^{-2}\}\hat{V}^{n-3/2}\\
&\ll \frac{|P|^{n+\ve}}{\hat Y^{n}}
(1+|P|^2\hat{Z})^{-n/2+1}\hat{Y}^{n/2+5/2}\hat{Z}\min\{\hat{Z},|P|^{-2}\}\hat{V}^{n-3/2}\\
&\ll |P|^{3/2+\ve}
(1+|P|^2\hat{Z})^{n/2-1/2}\hat{Y}^{n/2+1}\hat{Z}\min\{\hat{Z},|P|^{-2}\}.
 \end{split}
\end{equation}
We may again assume $Z=\hat{-Y-Q/2}$ to get,
\begin{align*}
\etwo&
\ll |P|^{-5/2+\ve}
(\hat{Q}/\hat{Y})^{n/2+1/2}\hat{Y}^{n/2+1}
\ll |P|^{-5/2+\ve}\hat{Q}^{n/2+1}\ll |P|^{n-4} |P|^{-n/3+17/6+\ve}\ll |P|^{n-4} |P|^{-1/6+\ve},
\end{align*}
as long as $n\geq 9$ and $\ve\leq 1/12$.
\subsection{Good $\uc$ contribution: $n$ even case.}
We will obtain a bound for the contribution to $E_1$ from the good values of $\uc$, when $2\mid n$. Since we are only aiming for $n\geq 10$ here, the analysis here is somewhat simpler and we may recycle many of our estimates from the previous case. To establish $n\geq 10$, we do not need our refined estimate in Lemma \ref{lem:SigB}, we will be content in using Lemma \ref{lem:sqfreeeasy} instead. When $n$ is even and $\uc$ is good, we will split the sum over $\v$ in $E_1$ into two subsums:
\begin{align*}
\sum_{\substack{ \v\in \cO^n\\
\v\neq \vecnull\\
|\v|\ll \hat Y |P|^{-1}J(Z)
}} =\sum_{\substack{ \v\in \cO^n\\
f^*(\v)\neq 0\\
|\v|\ll \hat Y |P|^{-1}J(Z)
}}+\sum_{\substack{ \v\neq\vecnull\in \cO^n\\f^*(\v)=0\\
|\v|\ll \hat Y |P|^{-1}J(Z)
}}.
\end{align*}
We again call the corresponding contributions $E_{1,1}$ and $E_{1,2}$ respectively.

As always, we write $r=br_1$, where $b$ denotes the square-free part of $r$ which is co-prime to $f^*(\v)dD_\F$ if $\uc$ is good and co-prime to $Q_\uc^*(\v)dD_\F$ if $\uc$ is bad. Analogous to \eqref{eq:E11Bound},  we
 apply Lemma \ref{lem:sqfreeeasy} to \eqref{eq:E(P)-1} to obtain $\b_1,N_1$  such that
\begin{equation}\label{eq:E(P)-2}
\eone(d\uc,Y,Z)
\ll \frac{|P|^{n+\ve}}{\hat Y^{n}}\hat{B}^{n/2+1}
\sum_{\substack{ \v\in \cO^n\\
f^*(\v)\neq 0\\
|\v|\ll \hat{V}
}} 
\sum_{\substack{
r_1\in \cO, d\mid r_1\\ |r_1|\leq \hat{Y}\\ r_1 \textrm{ monic }}} 
|S_{d\uc,r_1,\b_1,N_1}(\v)|J(Z)^{-n/2+1}\hat{Z}\min\{\hat{Z}, |P|^{-2}\}.
\end{equation}
We use the same process as in the beginning of Sec. \ref{sec:minor} and write $r=br_1=br_2r_3r_4$, where $r_2,r_3,r_4$ are chosen exactly as the analysis of $\eone$ in the $2\nmid n $ case, and introduce {\em dyadic} averages following the notation in \eqref{eq:udsum} to get:
\begin{equation}\label{eq:E(P)-3}
\begin{split}
\eone(\d,Y,Z)&:=\sum_{\substack{\ud=\d\\ \uc \textrm{ good }}}\eone(d\uc,Y,Z)\\
&\ll \frac{|P|^{n+\ve}}{\hat Y^{n}}\hat{B}^{n/2+1}
\sum_{\substack{\ud=\d\\ \uc \textrm{ good }}}\sum_{\substack{ \v\in \cO^n\\
f^*(\v)\neq 0\\
|\v|\ll \hat{V}
}} 
|S_{d\uc,r_1,\b_1,N_1}(\v)|J(Z)^{-n/2+1}\hat{Z}\min\{\hat{Z},|P|^{-2}\}.
\end{split}
\end{equation}
Note that our bounds in \eqref{eq:Splitbounds} to bound the exponential sums modulo $r_2,r_3$ and $r_4$ still hold when $2\mid n$. Therefore, this contribution is clearly less than our bounds for $\eone^1$ when $n\geq 9$ was odd (as compared with the corresponding bound \eqref{eq:sumfinal1}). Therefore, our analysis in Sec. \ref{subsec:ucgood} still holds and is enough to establish a suitable bound here. Note that the only auxiliary counting estimate which used the fact that $n$ was odd was in Lemma \ref{lem:Bro}, which was used to bound the number of solutions of $f^*(\v)=\msquare$, which is not necessary here, and it was only used to bound $\eone^2$ in Section \ref{subsec:ucgood}.

In a similar vein, when $f^*(\v)=0$, \cite[Lemma 2.9]{Browning_Vishe15} gives us:
$$\#\{|\v|\leq \hat{V}:f^*(\v)=0\}\ll \hat{V}^{n-3/2}.$$
Thus, the contribution $\etwo(\d,Y,Z)$ can be bound using the same process as from the corresponding bound when $2\nmid n$. Namely, the analysis in \eqref{eq:M} hands us a suitable bound for this contribution.

\subsection{Bad $\uc$ contribution}
We now focus on the contribution of the bad values of $\uc$ to $E_1$. We will deal with both odd and even values of $n$ here. Throughout, let $\uc$ denote an arbitrary, but fixed bad pair. We know that $|\uc|\ll 1$. In this case, there are no type II primes, as these are already included in our list of bad primes. However, an extra complication here arises due to the fact that when $\vp$ is a good prime satisfying $\vp\mid Q_\uc^*(\v')$, Lemma \ref{lem:expsumsingular} hands us the bound $|S_{\vp^k}(\v)|\ll |\vp|^{k(n+3/2)}$, which carries an extra factor of size $O(|\vp|^{1/2})$ as compared with the worst bound in Lemma \ref{lem:Expsum}. For a fixed $\v$, this bound only affects $\vp\mid Q_\uc^*(\v')$, which is a small set if $\uc$ and $\v$ are treated to be fixed. However, this would hinder us from obtaining any saving from the congruence condition $\vp^k\mid ((S^{-1})^t\v)_n$. Therefore, we will instead save more from the sum over $d$. This will be facilitated by the bound in Lemma \ref{lem:TypeIgen}. To this end, we split the sum $E_1(d\uc,Y,Z)$ into two subsums:
 \begin{align*}
\sum_{\substack{ \v\in \cO^n\\
\v\neq \vecnull\\
|\v|\ll \hat{V}
}} =\sum_{\substack{ \v\in \cO^n\\
Q_\uc^*(\v)\neq 0, \scrF^*(\v)\neq 0\\
|\v|\ll \hat{V}
}}+\sum_{\substack{ \v\neq\vecnull\in \cO^n\\
Q_\uc^*(\v)=0\textrm{ or }\scrF^*(\v)=0\\
|\v|\ll \hat{V}
}}.
\end{align*}
We call the contribution from the first sum on the right hand side as $E_3$ and from the second sum as $E_4$. When $Q_\uc^*(\v)\neq 0 $ and $\scrF^*(\v)\neq 0$, we write $d=d_1d_2d_3$, where $d_1,d_2,d_3$ are pairwise co-prime. Here $d_1d_2$ denote the square-free part of $d$ further satisfying $\gcd(d_1d_2,D_\F)=1$, $\gcd(d_1,r/d_1)=1$ and $d_2^2\mid r$. As a consequence, we may use Lemma \ref{lem:TypeIgen} to deal with the exponential sum $S_{d_1,d_1,\vecnull,1}(\v)$. If $d_2$ is large, we save from the fact that $d_2^2\mid r$, which reduces the number of permitted $r$'s (as opposed to just using the condition $d\mid r$). $d_3$ consists of square-full numbers and bad primes. Therefore, the total number of permitted $d_3$ is at most $O(\hat{D_3}^{1/2+\ve})$. 

To this end, as always we first write $r=br_1$, where $b$ denotes the square-free part of $r$ which is co-prime to $dD_\F Q_\uc^*(\v')$. Next, we write $r_1=d_1r_2r_3$, where $\gcd(r_2,dD_\F)=1$, $ r_3\mid (d_2d_3D_\F)^\infty$. In other words, $r_2$ consists of the part of $r_1$ which is free of the bad primes and of the primes dividing $d$, $r_3$ only consist of the powers of primes diving $d_2d_3D_\F$. Thus, for any given $d$, there are only $O(\hat{R_3}^\ve)$ choices for $r_3$, and $O(\hat{R_2}^{1/2})$ choices for $r_2$. We split our sum into analogous dyadic sums:
$$\ud=(|d_1|,|d_2|,|d_3|,|r_2|,|r_3|)=\d:=(\hat{D_1},\hat{D_2},\hat{D_3},\hat{R_2},\hat{R_3}),$$
where as before, let $|b|=B$, $B+D_1+R_2+R_3=Y, D=D_1+D_2+D_3, 2D\leq Y\leq Q/2+D$. Since $d_2^2\mid r_3$, we must have $B+R_2\leq Y-D-D_2\leq Q/2-D_2$. We begin by applying Lemma \ref{lem:sqfreeeasy} to \eqref{eq:E(P)-1} to get:
\begin{equation}\label{eq:moon}
\begin{split}
E_{3}(\d,Y,Z)
\ll~& \frac{|P|^{n+\ve}}{\hat Y^{n}}
\sum_{\substack{ \v\in \cO^n\\
\scrF^*(\v)Q_1^*(\v)\neq 0\\
|\v|\ll \hat Y |P|^{-1}J(Z)
}} 
\sum\limits_{\substack{\ud=\d}}
|S_{d\uc,r_1,\b_1,N_1}(\v)|J(Z)^{-n/2+1}\hat{B}^{n/2+(3+\delta_{2\mid n})/2}\times
 \hat{Z}\min\{\hat{Z},|P|^{-2}\}.
\end{split}
\end{equation} 
Here, the term $\delta_{2\mid n}$ is $1$ when $n$ is even and $0$ otherwise. Using the multiplicativity relation in Lemma \ref{lem:Multipli}, we may write
$$S_{d\uc,r_1,\b_1,N_1}(\v)=S_{d_1\uc,d_1,\vecnull,1}(\v)S_{r_2}(\v)S_{d_2d_3\uc,r_3,\b_2,N_2}(\v).$$
Lemmas \ref{lem:TypeIgen} and \ref{lem:ExpsumsComb} imply
$$|S_{d_1\uc,d_1,\vecnull,1}(\v)|\ll \hat{D_1}^{n/2+3/2}\gcd(d_1,Q_\uc^*(\v')\scrF^*(\v))^{1/2}. $$
Lemma \ref{lem:expsumsingular}, in conjunction with an argument similar to \eqref{eq:r2sum} implies that for a fixed $\v$ satisfying $Q^*_\uc(\v')\neq 0$, we have
\begin{align*}
\sum_{r_2}|S_{r_2}(\v)|\ll \hat{R_2}^{n/2+1}\sum_{r_2}\gcd(r_2,Q_\uc^*(\v'))^{1/2}\ll \hat{R_2}^{n/2+3/2+\ve}.
\end{align*}
Lastly, Lemma \ref{lem:ExpsumsComb} gives us
$$|S_{d_2d_3\uc,r_3,\b_2,N_2}(\v)|\ll \hat{D_2}\hat{D_3}\hat{R_3}^{n/2+1}(\gcd(r_3/d_2d_3,((S^{-1})^t\v)_n)^{1/2}. $$
As before, for a fixed $d$, there are only $O(\hat{R_3}^{\ve})$ choices for $r_3$. We evaluate the sums in the following order
\begin{align*}
\sum_{d_1,d_2,d_3}\sum_{r_3}\sum_{\v}\sum_{r_2}.
\end{align*}  
Combining our bounds, we get:
\begin{align*}
\sum_{\ud=\d}|S_{d\uc,r_1,\b_1,N}(\v)|&\ll \hat{D}\sum_{d_1,d_2,d_3}\sum_{r_3}
\sum_{x_1\mid r_3/(d_2d_3)}\sum_{x_2\mid d_1}
\sum_{\substack{ \v\in \cO^n\\|\v|\leq \hat{V}\\x_1\mid ((S^{-1})^t\v)_n\\
 x_2\mid Q_\uc^*(\v')\scrF^*(\v)
}} |x_1x_2|^{1/2}\hat{D_1}^{n/2+1/2}\hat{R_3}^{n/2+1}\hat{R_2}^{n/2+3/2}.
\end{align*}
We may use \eqref{eq:Vsumbound1} to bound the number of permissible $\v$'s satisfying $x_1\mid ((S^{-1})^t\v)_n$ and \cite[Lemma 4]{Browning-Heath-Brown09} to bound the number of $\v$'s satisfying $x_2\mid Q_\uc^*(\v')\scrF^*(\v)$, to obtain:
\begin{align*}
\sum_{\ud=\d}|S_{d\uc,r_1,\b_1,N}(\v)|&\ll \hat{D}\hat{R_1}^{n/2+1}\hat{D_1}^{-1/2}\hat{R_2}^{1/2}\sum_d\sum_{r_3}
\sum_{x_1\mid r_3/(d_2d_3)}\sum_{x_2\mid d_1}
 |x_1x_2|^{1/2}\hat{V}^{n-1}(1+\hat{V}\min\{|x_1|^{-1},|x|_2^{-1}\})\\
 &\ll \hat{D}\hat{R_1}^{n/2+1}\hat{D_1}^{-1/2}\hat{R_2}^{1/2}\sum_d\sum_{r_3}
\sum_{x_1\mid r_3/(d_2d_3)}\sum_{x_2\mid d_1}
 |x_1x_2|^{1/2}\hat{V}^{n-1}(1+\hat{V}/|x_1x_2|^{1/2})\\
 &\ll \hat{D}\hat{D_1}^{1/2}\hat{D_2}\hat{D_3}^{1/2}\hat{R_1}^{n/2+1+\ve}\hat{R_2}^{1/2}\hat{V}^{n-1}
 (\hat{R_3}^{1/2}\hat{D_1}^{1/2}/(\hat{D_2}\hat{D_3})^{1/2}+\hat{V}).
\end{align*}
Feeding this bound back to \eqref{eq:moon}, we get:
\begin{align*}\notag
E_{3}(\d,Y,Z)
\ll~& \hat{B}^{\frac{\delta_{2\mid n}}{2}}\frac{|P|^{n+\ve}}{\hat{Y}^{n/2-1}}\hat{D}\hat{D_1}^{1/2}\hat{D_2}\hat{D_3}^{1/2}(\hat{B}\hat{R_2})^{1/2}J(Z)^{-n/2+1}\hat{V}^{n-1}\hat{Z}\min\{\hat{Z},|P|^{-2}\}\times\\
~&(\hat{R_3}^{1/2}\hat{D_1}^{1/2}/(\hat{D_2}\hat{D_3})^{1/2}+\hat{V}).
\end{align*}
 Thus, when $2\nmid n$, $E_3$ can be bounded by 
\begin{align}\label{eq:R}
\ll (\hat{B}\hat{R_2})^{1/2}\frac{|P|^{n+\ve}}{\hat Y^{n/2-1}}\hat{D}\hat{D_1}^{1/2}\hat{D_2}\hat{D_3}^{1/2}
\hat{V}^{n-1}((\hat{R_3}/(\hat{D_2}\hat{D_3})^{1/2}+\hat{V} )(1+|P|^2\hat{Z})^{-n/2+1}\min\{\hat{Z},|P|^{-2}\}.
\end{align}
After replacing $\hat{V}=\hat{Y}|P|^{-1}(1+|P|^2\hat{Z})$, clearly, the contribution is maximum when $Z=-Y-Q/2$, which we assume from now on. Let us first investigate the contribution coming from the term $(\hat{R_3}/(\hat{D_2}\hat{D_3}))^{1/2}$ on the right hand side of \eqref{eq:R}. This contribution is 
\begin{equation}\label{eq:oone}
\begin{split}
&\ll \frac{|P|^{n+\ve}}{\hat Y^{n/2-3/2}}\hat{D}^{3/2}\hat{V}^{n-1}(1+|P|^2\hat{Z})^{-n/2+1}|P|^{-2}\hat{Z}\ll |P|^{-1+\ve}\hat{Y}^{n/2+1/2+3/4}(1+|P|^2\hat{Z})^{n/2}\hat{Z}\\
&\ll |P|^{-1-2/3+\ve}\hat{Y}^{n/2+1/4}(\hat{Q}/\hat{Y})^{n/2}\ll |P|^{-5/3+\ve}\hat{Q}^{n/2+1/4}=|P|^{2n/3-4/3+\ve}=|P|^{n-4-(n-8)/3+\ve}.
\end{split}
\end{equation}
This is admissible for $n\geq 9$ and odd, as long as $\ve\leq 1/16$. 

Now let us turn to the remaining contribution to $E_3$. Here, we will use that $B+R_2\leq Y-D_1-2D_2-D_3$. Thus, this contribution is 
\begin{equation}\label{eq:Two}
\begin{split}
&\ll (\hat{B+R_2})^{1/2}\frac{|P|^{n+\ve}}{\hat Y^{n/2-1}}\hat{D}\hat{D_1}^{1/2}\hat{D_2}\hat{D_3}^{1/2}\hat{V}^{n}(1+|P|^2\hat{Z})^{-n/2+1}|P|^{-2}\hat{Z}\\
&\ll |P|^{-2+\ve}\hat{Y}^{n/2+3/2}\hat{D}(1+|P|^2\hat{Z})^{n/2+1}\hat{Z}\ll |P|^{-2+\ve}\hat{Y}^{n/2+2}(1+|P|^2\hat{Z})^{n/2+1}\hat{Z}.
\end{split}
\end{equation}
We may again assume that $Z=-Y-Q/2$ to obtain that this is
\begin{align*}
&\ll |P|^{-8/3+\ve}\hat{Y}^{n/2+1}(\hat{Q}/\hat{Y})^{n/2+1}
\ll |P|^{-8/3+\ve}\hat{Q}^{n/2+1}=|P|^{2n/3-4/3+\ve},
\end{align*}
which is clearly enough from our previous calculation.

When $2\mid n$, the bound in \eqref{eq:R} gets multiplied with an extra factor of size $O(\hat{B}^{1/2})$. Here, we will use a weaker bound $\hat{B}\leq \hat{Y}/\hat{D}$ and combine it with our above bounds. Note that in the extreme case when $D=Y/2$, $\hat{B}^{1/2}$ factor amounts to the introduction of an extra factor of  size $O(\hat{Y}^{1/4}) $ in the final computation. In particular, when $2\mid n$, the bound corresponding to \eqref{eq:oone} is given by 
\begin{align*}
\ll \frac{|P|^{n+\ve}}{\hat Y^{n/2-2}}\hat{D}\hat{V}^{n-1}(1+|P|^2\hat{Z})^{-n/2+1}|P|^{-2}\hat{Z}\ll |P|^{-5/3+\ve}\hat{Y}^{n/2+1/2}(\hat{Q}/\hat{Y})^{n/2}\ll |P|^{n-4+\ve-(n-9)/3},
\end{align*}
which is admissible as long as $n\geq 10$ and $\ve\leq 1/6$. 

Similarly, when $2\mid n$, the contribution corresponding to \eqref{eq:Two} to $E_3$ is bounded by:
\begin{align*}
\ll |P|^{-2+\ve}\hat{Y}^{n/2+2}\hat{D}^{1/2}(1+|P|^2\hat{Z})^{n/2+1}\hat{Z}\ll |P|^{-8/3+\ve}\hat{Q}^{n/2+5/4}\ll |P|^{n-4+\ve-(n-9)/3},
\end{align*}
again enough when $n\geq 10$ and $\ve\leq 1/6$.

We now turn to the term $E_4$. When either $\scrF^*(\v)=0$ or $Q_\uc^*(\v)=0$, we gain from the sparseness of such $\v$'s. We write $r_1=r_2r_3$, where $\gcd(r_2,dD_\F)=1$ and that $ r_3\mid (d D_F)^\infty$. We split our sum into the dyadic sums:
$$\ud=(|d|,|r_2|,|r_3|)=\d:=(\hat{D},\hat{R_2},\hat{R_3}).$$
Here, $D\leq Y/2$, $Y=B+R_2+R_3$. In this case, we will use the following softer bound coming from Lemma \ref{lem:Expweak}: $$|S_{d\uc,r_1,\b_1,N_1}(\v)|\ll |d|^{1/2}|r_1|^{n/2+3/2}.$$
Thus, following the recipe before, 
\begin{align*}
E_4(\d,Y,Z)\ll \hat{B}^{\frac{\delta_{2\mid n}}{2}}\frac{|P|^{n+\ve}}{\hat{Y}^n}\hat{V}^{n-3/2}\hat{D}^{3/2}\hat{Y}^{n/2+3/2}J(Z)^{-n/2+1}\hat{Z}\min\{\hat{Z},|P|^{-2}\}.
\end{align*}
Again, when $2\mid n$, an extra factor of $\hat{B}^{1/2}$ arises due to our worse bounds in Lemma \ref{lem:sqfreeeasy}.
When $2\nmid n$, this contribution is clearly sufficient from our bounds for $E_{1,2}$, cf. \eqref{eq:M}, as long as $\ve\leq 1/12 $. Similarly, when $2\mid n$, the extra factor of size $\hat{B}^{1/2}$ ultimately, adds a factor of size $\hat{Q}^{1/4}$ to our worst case scenario, i.e. when $D=Y/2=Q/2 $. Therefore, following similar steps as in our bounds for $E_3$, this can be shown to be satisfactory as long as $n\geq 10$ and $\ve\leq 1/12$.
\subsection{Bounding $E_2$}
Finally, we turn to the term $E_2$. Note that the bounds for $E_2$ are superseded by those for $E_1$ as long as $\hat{V}=\frac{\hat{Y}}{|P|} J(Z)\geq 1$. Thus, we only need consider bounding $E_2$ when both conditions $\hat{Q}^\Delta\leq \hat{Y}\leq |P|$ and  $\hat{Z}\leq (|P|\hat{Y})^{-1}$ are satisfied. Here, we may use the bound in \eqref{eq:singul} to get:
\begin{equation}
\label{eq:Singular1}
|r_N|^{-n}\sum_{|d\uc|\leq \hat{Y}^{1/2}}\sum_{\substack{|r|=\hat{Y}\\ d\mid r}}|S_{d\uc,r,\b,N}(\vecnull)|\ll \hat{Y}^{(7-n)/2}.
\end{equation}
Thus, 
\begin{align*}
E_2(Y,P,Z)&:=|P|^{n+\ve}\sum\limits_{\substack{|r|=\hat{Y}\\r \textrm{ monic }}}\,\,\,\sum_{\substack{d\mid r\textrm{ monic, }\uc \textrm{ primitive}\\ \hat{Y-Q/2}\leq|d\uc|\leq \hat{Y/2}\\ |dc_2|<\hat{Y/2}}}|r_N|^{-n}
 \int_{|\uz|=\hat{Z}}S_{d\uc,r,\b,N}(\vecnull)I_{r_N}(\uz;\v)d\uz\\
&\ll |P|^{n+\ve} \hat{Y}^{-(n-7)/2}J(Z)^{-n/2+1}\hat{Z}\min\{\hat{Z},|P|^{-2}\}\\
&\ll |P|^{n-2+\ve} \hat{Y}^{-(n-7)/2}(1+|P|^2\hat{Z})^{-1}\hat{Z}\ll |P|^{n-4+\ve} \hat{Y}^{-(n-7)/2}\\
&\ll |P|^{n-4+\ve-\Delta/2}.
\end{align*}
as long as $n\geq 8$ and $\ve\leq \Delta/4$.

\bibliographystyle{plain}
\bibliographystyle{plain}
\end{document}